\documentclass[a4paper,reqno]{amsart}

\usepackage[utf8]{inputenc}
\usepackage{amssymb}
\usepackage{enumitem}
\usepackage{mathrsfs}
\usepackage{mathtools}
\usepackage{amsmath}
\usepackage [dvipsnames] { xcolor }
\usepackage{multirow}
\usepackage[colorlinks=true]{hyperref}

\hypersetup{linkcolor=Green,urlcolor=Cyan,citecolor=Cyan}

\setlist[enumerate]{label=\emph{(\roman*)}}

\usepackage{environ}

\newtheorem{theorem}{Theorem}[section]

\newtheorem{lemma}[theorem]{Lemma}
\newtheorem{proposition}[theorem]{Proposition}

\theoremstyle{definition}
\newtheorem{definition}[theorem]{Definition}
\newtheorem{remark}[theorem]{Remark}
\numberwithin{equation}{section}
\newcommand{\R}{\mathbb{R}}

\newcommand{\dd}{{\rm d}}
\newcommand{\vertiii}[1]{{\left\vert\kern-0.25ex\left\vert\kern-0.25ex\left\vert #1 
    \right\vert\kern-0.25ex\right\vert\kern-0.25ex\right\vert}}

\begin{document}

\title[Minimal Mass Blow-up solution]{Construction of blow-up solution with minimal mass for 2D cubic Zakharov--Kuznetsov equation}

\author{Yang Lan}
\address{Yau Mathematical Sciences Center, Tsinghua University, Beijing 100084, P. R. China.}
\email{lanyang@mail.tsinghua.edu.cn}

\author{Xu Yuan}
\address{State Key Laboratory of Mathematical Sciences, Academy of Mathematics and Systems Science, Chinese Academy of Sciences, Beijing 100190, P. R. China.}
\email{xu.yuan@amss.ac.cn}

\begin{abstract}
In this article, we construct a minimal mass blow-up solution of the two-dimensional cubic (mass-critical) Zakharov--Kuznetsov equation:
	\begin{equation*}
		\partial_t \phi+\partial_{x_1}(\Delta \phi+\phi^3)=0,\quad (t,x)\in [0,\infty)\times \mathbb{R}^2.
	\end{equation*}
	Let $s>\frac{3}{4}$. Bhattacharya-Farah-Roudenko~\cite{BGFR} show that $H^{s}$ solutions with $\|\phi\|_{L^{2}}<\|Q\|_{L^{2}}$ are global in time. For such low regularity solutions, we study the dynamics at the threshold $\|\phi\|_{L^{2}}=\|Q\|_{L^{2}}$ and demonstrate that finite time blow-up singularity formation may occur. This result and its proof are inspired by the recent blow-up result~\cite{MartelPilod} for the mass-critical gKdV equation. This result is also complement of previous result~\cite{CLYMINI} for nonexistence of minimal mass blow-up element in the energy space $H^{1}(\R^{2})$ of the two-dimensional cubic Zakharov--Kuznetsov equation.

\end{abstract}

\maketitle

\section{Introduction}
\subsection{Main result}
In this article, we consider the minimal mass dynamics for
the 2D cubic Zakharov--Kuznetsov (ZK) equation,
\begin{equation}\label{CP}
\partial_t \phi+\partial_{x_1}(\Delta \phi+\phi^3)=0,\quad (t,x)\in [0,\infty)\times \mathbb{R}^2.
\end{equation}
Here, $x=(x_{1},x_{2})\in \R^{2}$ and  $\Delta=\partial_{x_{1}}^{2}+\partial_{x_{2}}^{2}$ is the Laplace operator on $\R^{2}$. The Cauchy problem for~\eqref{CP} is locally well-posed in the Sobolev space $H^{s}(\R^{2})$ with $s\ge \frac{1}{4}$ (see~\cite{Faminskii,Kinoshita,LinPas}). More precisely, for any initial data $u_{0}\in H^{s}(\R^{2})$ with $s\ge \frac{1}{4}$, there exists a unique (in a certain class) maximal solution $\phi$ of~\eqref{CP} in 
$C\left([0,T);H^{s}(\R^{2})\right)$. For this Cauchy problem, the following blow-up criterion holds:
\begin{equation}\label{equ:blowCauchy}
	T<\infty\Longrightarrow 
	\||\nabla |^{s}u(t)\|_{L^{2}}=\infty\quad \mbox{as}\ t\uparrow T.
\end{equation}

\smallskip
Recall that, if $\phi$ is a solution of~\eqref{CP} then for any $\lambda>0$, the function 
\begin{equation*}
	\phi_{\lambda}(t,x)=\frac{1}{\lambda}\phi\left(\frac{t}{\lambda^{3}},\frac{x}{\lambda}\right),\quad \mbox{for}\ (t,x)\in [0,\infty)\times \R^{2},
\end{equation*}
is also a solution of~\eqref{CP}. This scaling symmetry keeps the $L^{2}$-norm invariant so that the problem is \emph{mass-critical}. 
Recall also that, for any solution $\phi$ of~\eqref{CP}, the mass $M(\phi)$ and the energy $E(\phi)$ are formally conserved, where
\begin{equation*}
	M(\phi)=\int_{\R^{2}}\phi^{2} \dd x\quad \mbox{and}\quad 
	E(\phi)=\frac{1}{2}\int_{\R^{2}}|\nabla \phi|^{2}\dd x
	-\frac{1}{4}\int_{\R^{2}}\phi^{4}\dd x.
\end{equation*} 

We recall the family of solitary wave solutions of~\eqref{CP}. Let $Q\in H^{1}(\R^{2})$ be the unique radial positive solution of the equation
\begin{equation*}
	-\Delta Q+Q-Q^{3}=0,\quad \mbox{on}\ \R^{2}.
\end{equation*}
Then, for any $(\lambda_{0},x_{10},x_{20})\in (0,\infty)\times \R^{2}$, the function 
\begin{equation*}
	\phi(t,x)=\lambda_{0}^{-1}Q\left(\lambda_{0}^{-1}\left(x_{1}-\lambda_{0}^{-2}t-x_{10}\right), \lambda_{0}^{-1}\left(x_{2}-x_{20}\right)\right),
\end{equation*}
is a solution of~\eqref{CP} called \emph{solitary wave} or \emph{soliton}. It is well-known that $E(Q)=0$ and that $Q$ is related to the following
sharp Gagliardo-Nirenberg inequality (see~\cite{Weinstein})
\begin{equation*}
	\int_{\R^{2}}|f(x)|^{4}\dd x
	\le 2\left(\int_{\R^{2}}|\nabla f(x)|^{2}\dd x\right)
	\left(\frac{\int_{\R^{2}}|f(x)|^{2}\dd x}{\int_{\R^{2}}|Q(x)|^{2}\dd x}\right),\quad \mbox{for any}\ f\in H^{1}(\R^{2}).
\end{equation*}
Following the above inequality, the conservation of energy and the blow-up criterion~\eqref{equ:blowCauchy} imply that any initial data $\phi_{0}\in H^{1}(\R^{2})$ with subcritical mass, \emph{i.e.} satisfying $\|\phi_{0}\|_{L^{2}}<\|Q\|_{L^{2}}$, generates a \emph{global and bounded} solution in $H^{1}(\R^{2})$.

\smallskip
To go beyond the threshold mass $\|Q\|_{L^{2}}$, it is natural to restrict to solutions with small supercritical mass, \emph{i.e.} satisfying
\begin{equation}\label{equ:massabove}
	\|Q\|_{L^{2}}<\|\phi_{0}\|_{L^{2}}<(1+\delta)\|Q\|_{L^{2}},\quad
	\mbox{for}\ \ 0<\delta\ll 1.
\end{equation}
Based on a rigidity property (see~\cite[Proposition 1.3]{FHRY}) of the ZK flow around the family of soliton, the first proof for the existence of blow-up solutions with small supercritical mass was given in Farah-Holmer-Roudenko-Yang~\cite[Theorem 1.1]{FHRY}. More precisely, they showed that, for any negative energy initial data $\phi_{0}\in H^{1}(\R^{2})$ satisfying~\eqref{equ:massabove}, the corresponding solution $\phi(t)$ blows up in finite/infinite forward time. Then, Chen-Lan-Yuan~\cite{CLY} revisited the blow-up analysis for the 2D  cubic ZK equation in light of previous developments related to the blow-up of the mass-critical generalized Korteweg-de Vries (gKdV) equation (see \S\ref{SS:Stable blow-up} for more discussion).

\smallskip
At the threshold mass $\|\phi_{0}\|_{L^{2}}=\|Q\|_{L^{2}}$, Chen-Lan-Yuan~\cite{CLYMINI} showed that there exists no \emph{minimal mass blow-up} solution in $H^{1}(\R^{2})$ for~\eqref{CP}. Here, we say that $\phi(t)$ is a minimal mass blow-up solution in $H^{1}(\R^{2})$ for~\eqref{CP} if $\|\phi_{0}\|_{L^{2}}=\|Q\|_{L^{2}}$ and there exists $0<T\le \infty$ such that 
\begin{equation*}
	\phi(t)\in C([0,T);H^{1}) \ \ \mbox{with}\ \
	\|\nabla \phi(t)\|_{L^{2}}\to \infty \ \ \mbox{as} \ \ t\uparrow T.
\end{equation*}
In the case of low regularity, Linares-Pastor~\cite{LinPasJFA} proved the global well-posedness in $H^{s}(\R^{2})$ for $s>\frac{53}{63}$, under the condition of subcritical mass for initial data. Later on, Bhattacharya-Farah-Roudenko~\cite{BGFR} used the $I$-method and obtained such global well-posedness in $H^{s}(\R^{2})$ for $s>\frac{3}{4}$.
In this article, we show that the critical mass condition is sharp by constructing a minimal mass blow-up solution in $H^{\frac{7}{9}}(\R^{2})$, \emph{i.e.} a solution of~\eqref{CP} which blows up in finite time in $H^{\frac{7}{9}}(\R^{2})$ with the threshold mass $\|Q\|_{L^{2}}$. Our main result is formulated as follows.

\begin{theorem}\label{thm:main}
    There exist $T_{0}>0$ and a solution $\mathcal{S}\in C((0,T_{0}];H^{\frac{7}{9}}(\R^{2}))$ to~\eqref{CP} with critical mass 
    $\|\mathcal{S}(t)\|_{L^{2}}=\|Q\|_{L^{2}}$
    such that 
    \begin{equation*}
        \mathcal{S}(t)-\frac{1}{\lambda(t)}Q\left(\frac{\cdot-x(t)}{\lambda(t)}\right)\to0,\quad \mbox{in}\ H^{\frac{7}{9}}(\R^{2})\ \ \mbox{as}\ \ t\downarrow 0.
    \end{equation*}
    Here, the functions $t\mapsto \lambda(t)$ and $t\mapsto x(t)$ satisfy
    \begin{equation*}
    \lambda(t)\sim t^{\frac{1}{3-\theta}}\quad \mbox{and}\quad x(t)\sim (-
    t^{-\frac{\theta-1}{3-\theta}},0
    )\quad \mbox{as}\ t\downarrow 0.
    \end{equation*}
\end{theorem}

\begin{remark}\label{re:constanttheta}
	The constant $\theta$ is defined by 
    \begin{equation*}
        \theta=2\bigg(\int_{\R}\frac{|\widehat{F}(\xi)|^{2}}{1+|\xi|^{2}}\dd \xi\bigg)\bigg/
    \bigg(\int_{\R}|\widehat{F}(\xi)|^{2}\dd \xi\bigg).
    \end{equation*}
    Here, we consider the following two smooth functions:
\begin{equation*}
    \widehat{F}(\xi)=\frac{1}{\sqrt{2\pi}}\int_{\R}F(y_{2})e^{-i\xi y_{2}}\dd y_{2}\quad \mbox{where}\quad F(y_{2})=\int_{\R}\Lambda Q(y_{1},y_{2})\dd y_{1}.
\end{equation*}
Note that, the value of $\theta$ plays a crucial role in our analysis, since it determines the blow-up rate of $\mathcal{S}(t)$. By an elementary numerical computation, we find $\theta \approx 1.66\in \left(\frac{8}{5},\frac{9}{5}\right)$. We refer to~\cite[Appendix A]{CLY} for the details of numerical computation.
\end{remark}

\begin{remark}\label{re:Skdv}
    For the mass-critical gKdV equation, which is the closest model related to~\eqref{CP}, the minimal mass blow-up solution $\mathcal{S}_{{\rm{KdV}}}$ is known to be \emph{global} for $t\in (0,\infty)$ and to be the \emph{unique} minimal mass blow-up element (see~\cite{MMR1} and \S\ref{SS:Comparison} for more discussion). For the ZK equation~\eqref{CP}, such properties are open problems.
\end{remark}

\begin{remark}\label{re:indexSobolev}
	   We expect that the index of the Sobolev space in Theorem~\ref{thm:main} can be improved. Some additional analyses, which are related to the linear estimate of the ZK flow, will be needed to obtain such improvement (see \S\ref{S:Endproof} for more discussion).
       We also expect that the optimal index is $\frac{\theta}{2}$, which matches the difference in size between ${(\lambda,b)}$ (see~\cite[Section 1.3]{CLYMINI} for more discussion).
\end{remark}

\subsection{Soliton dynamics for the ZK equation}\label{SS:Stable blow-up}
First, we briefly review the literature related to the singularity formation for the ZK models. As mentioned above, for the 2D cubic ZK equation~\eqref{CP}, Farah-Holmer-Roudenko-Yang~\cite{FHRY} obtained the first qualitative information on the flow for small supercritical mass initial data. In particular, they proved the existence of a finite/infinite time blow-up solution with negative energy. Then, Chen-Lan-Yuan~\cite{CLY} revisited the study of singularity formation and soliton dynamics for~\eqref{CP}. More precisely, for $H^{1}$ solution satisfying a suitable space decay property (see the definition of initial data set $\mathcal{A}$ in~\cite[Definition 1.1]{CLY}), Chen-Lan-Yuan~\cite{CLY} proved that only three possible behaviors can occur:

\begin{description}
    \item[Exit] The solution eventually exits any small neighborhood of the solitons.
    
    \item [Blow-up] The solution blows up in finite time with blow-up rate $(T-t)^{-\frac{1}{3-\theta}}$.

    \item [Soliton]: The solution is global and locally converges to a soliton.
\end{description}
In Chen-Lan-Yuan~\cite{CLYMINI}, building on the tools developed
in~\cite{CLY}, the authors also proved the nonexistence of the minimal mass blow-up solution in the energy space $H^{1}(\R^{2})$. Last, we refer to Bozgan-Ghoul-Masmoudi-Yang~\cite{BGMY} for a similar result as in~\cite{CLY}, which was proved independently. More importantly, the two proofs use different energy-virial Lyapunov functionals.

\smallskip
We now briefly review the literature related to the stability and asymptotic stability of solitons for the ZK models. Orbital stability of ZK solitons is well known since the work of de Bouard~\cite{DEZK}.
Then, C\^ote-Mu\~noz-Pilod-Simpson~\cite{CMPS} showed the asymptotic stability of solitons and stability of multi-solitons for the 2D quadratic ZK equation via a virial-type estimate. It is worth mentioning here that, this estimate relies on a simple sign condition which can be checked numerically. However, this numerical computation does not apply to the 3D case. Recently, Farah-Holmer-Roudenko-Yang~\cite{FHRY1} obtained a virial-type estimate for the 3D case via passage to a dual problem. This estimate relies on the numerical verification of a spectral condition with robust numerical analysis. We refer to Farah-Holmer-Roudenko~\cite{FHR,FHR1} for the results on the instability of solitons for the 2D mass-critical and mass-supercritical problems, respectively. Last, for the 2D and 3D quadratic ZK equations, we also refer to~\cite{PV2,PV1} for the results on the asymptotic stability of multi-solitons and a description of the collision of two nearly equal solitary waves.

\subsection{Comparison with NLS and gKdV equations}\label{SS:Comparison}
In the last twenty years, there have been significant progress in the study of the singularity formation for the mass-critical nonlinear Schr\"odinger (NLS) and generalized Korteweg-de Vries (gKdV) equations, in particular for the description of stable blow-up dynamics and the classification of minimal mass blow-up solutions. In what follows, we will briefly survey the literature related to the minimal mass blow-up of the mass-critical NLS and gKdV equations. We mention here that, the 2D cubic ZK equation is a natural extension of the mass-critical gKdV equations in two dimensions.

\smallskip
Consider the following mass-critical NLS equation:
\begin{equation*}
	i\partial_{t}\phi+\Delta \phi+|\phi|^{\frac{4}{N}}\phi=0,\quad (t,x)\in [0,\infty)\times \R^{N}.
\end{equation*}
Based on the structure of the NLS equation, the pseudo-conformal symmetry for this equation implies that if $\phi(t,x)$ is a solution of this equation then 
\begin{equation*}
	\Phi(t,x)=\frac{1}{t^{\frac{N}{2}}}\overline{\phi}\left(\frac{1}{t},\frac{x}{t}\right)e^{i\frac{|x|^{2}}{4t}},
\end{equation*}
is also a solution of this equation.
It is well-known (see~\cite{Caz}) that the pseudo-conformal symmetry generates an
\emph{explicit minimal mass blow up solution}
\begin{equation*}
	\mathcal{S}_{\rm{NLS}}(t,x)=\frac{1}{t^{\frac{N}{2}}}e^{-i\frac{|x|^{2}}{4t}-\frac{i}{t}}Q_{{\rm{NLS}}}\left(\frac{x}{t}\right),
\end{equation*}
defined for all $t> 0$ and blowing up as $t\downarrow 0$. In particular, we have 
\begin{equation*}
 \|\mathcal{S}_{{\rm{NLS}}}(t)\|_{L^{2}}= \|Q_{{\rm{NLS}}}\|_{L^{2}}\quad \mbox{and}\quad  \|\nabla \mathcal{S}_{{\rm{NLS}}}(t)\|_{L^{2}}\sim \frac{1}{t}\ \mbox{as}\ t\downarrow 0.
\end{equation*}
Here, $Q_{\rm{NLS}}\in H^{1}(\R^{N})$ is the unique radial ground state of the mass-critical NLS equation. Later on, using the pseudo-conformal symmetry, the pioneering work of Merle~\cite{Merlenls} showed that $\mathcal{S}_{\rm{NLS}}$ is the unique (up to the symmetries) minimal mass blow-up solution of the mass-critical NLS equation in the energy
space $H^{1}(\R^{N})$. Recently, Dodson~\cite{Dodson1,Dodson2} proved the uniqueness result in the critical space $L^2(\mathbb{R}^N)$ for $1\leq N\leq 15$. Note that, the \emph{existence and uniqueness} of minimal mass blow-up solution both rely on the explicit pseudo-conformal which is absent from the inhomogeneous NLS equation. The study of inhomogeneous problem was start by Merle~\cite{Merlenonnls}, and in particular, this work obtained a sufficient condition to ensure the \emph{nonexistence} of minimal mass blow-up elements for the inhomogeneous problem. Then, a novel methodology to construct minimal mass blow-up elements for an inhomogeneous NLS equation was developed in Rapha\"el-Szeftel~\cite{RaSz}. More precisely, they gave a necessary and sufficient condition to ensure the \emph{existence} of minimal mass blow-up elements, and gave a complete classification in the energy space for such elements.

\smallskip
We also consider the following mass-critical gKdV equation:
\begin{equation*}
	\partial_{t}\phi+\partial_{x}\left(\partial_{x}^{2}\phi+\phi^{5}\right)=0,\quad (t,x)\in [0,\infty)\times \R.
\end{equation*}
First, Martel-Merle~\cite{MMDUKE} obtained
the \emph{global existence} result for minimal mass solutions with decay on the right. In other words, minimal mass blow up is not compatible with the suitable decay on the right. Then, the \emph{existence and description} of the minimal mass blow-up solution were studied by Martel-Merle-Rapha\"el~\cite{MMR1}. More precisely, we denote by $Q_{\rm{KdV}}\in H^{1}(\R)$ the unique radial ground state of the mass-critical gKdV equation. Based on the work of~\cite{MMR1}, we know that there exists a unique solution $\mathcal{S}_{{\rm{KdV}}}(t)$ on $(0,\infty)$ such that 
\begin{equation*}
	\|\mathcal{S}_{{\rm{KdV}}}(t)\|_{L^{2}}= \|Q_{{\rm{KdV}}}\|_{L^{2}}\quad \mbox{and}\quad  \|\partial_{x} \mathcal{S}_{{\rm{KdV}}}(t)\|_{L^{2}}\sim \frac{1}{t}\ \mbox{as}\ t\downarrow 0.
\end{equation*}
Moreover, the sharp asymptotics, both in the time and space variables, were derived in~\cite{Comkdv}, for any order derivative of $\mathcal{S}_{\rm{KdV}}(t)$.
 We refer to ~\cite{Banica, Com, Krieger,LeCoz,MartelPilodAnn,Suminimal,TangXu} and the references therein for the existence of minimal mass blow-up solutions for other dispersive equations with local or nonlocal structures.

\smallskip
Two differences with the 2D cubic ZK case are thus apparent. First, Chen-Lan-Yuan~\cite{CLYMINI} showed that there exists \emph{no finite/infinite time blow-up solution} with minimal mass in the energy space $H^{1}(\R^{2})$. This is clearly different from the cases of the NLS and gKdV equations. Second, in this article, we prove the existence of a minimal mass blow-up solution in the \emph{low-regularity Sobolev space} $H^{\frac{7}{9}}(\R^{2})$ rather than in the standard energy space $H^{1}(\R^{2})$ as in the NLS and gKdV cases. Last, we refer to~\cite{Coll,SunZheng} and references therein for the study of the low regularity blow-up solution for the mass-critical NLS equation.  To our knowledge, Theorem~\ref{thm:main} proves the first statement of the existence of minimal mass blow-up solutions in the low-regularity Sobolev space which do not exist in standard energy space.

\subsection{A Roadmap of the proof}\label{SS:Sketch}
As usual in investigating blow-up phenomenon, we first introduce the rescaled space-time variables $(s,y)\in (-\infty,0)\times \R^{2}$ with 
\begin{equation*}
	\frac{\dd s}{\dd t}=\frac{1}{\lambda^{3}(t)}\quad \mbox{and}\quad 
	y=(y_{1},y_{2})=\left(\frac{x_{1}-x_{1}(t)}{\lambda(t)},\frac{x_{2}}{\lambda(t)}\right).
\end{equation*}
The time dependent functions $(\lambda,x_{1})$ are to be determined. We look for a solution $\phi$ of~\eqref{CP} blowing up at $(t,x)=(0,0)$ with $\phi(t,x)=\lambda^{-1}(s)\omega(s,y)$. From~\eqref{CP} and an elementary, we directly have
\begin{equation}\label{equ:rescal}
    \partial_{s}\omega+\partial_{y_{1}}\left(\Delta \omega-\omega+\omega^{3}\right)-\frac{\lambda_{s}}{\lambda}
    \Lambda \omega-\left(\frac{x_{1s}}{\lambda}-1\right)\partial_{y_{1}}\omega
    =0.
\end{equation}
We consider the solution of~\eqref{equ:rescal} close to a soliton, that is, $\omega=Q+\zeta$ where $\zeta$ is a small function. It follows directly from~\eqref{equ:rescal} that 
\begin{equation*}
    \partial_{s}\zeta+\partial_{y_{1}}\left(\Delta \zeta-\zeta+\zeta^{3}\right)-\frac{\lambda_{s}}{\lambda}
    \Lambda \zeta-\left(\frac{x_{1s}}{\lambda}-1\right)\partial_{y_{1}}\zeta
    ={\rm{L.O. T}}.
\end{equation*}
Recall that, in the previous work~\cite{CLY}, we consider $\zeta\sim bP_{b}$ where $P_{b}$ is a suitable localized profile, and such an approximate solution can help us obtain a stable blow-up solution around $Q$. However, in our current case, the minimal mass blow-up solution does not belong to the space $\mathcal{A}$ which means that the refined control of $b$ is not allowed.
Thus, we do not expect such an approximate solution to help us construct a minimal mass element with explicit asymptotic behavior.

\smallskip
Inspired by the previous work~\cite{MartelPilod} for the study of blow-up dynamics for the gKdV equation, we consider an expansion $\zeta=\zeta_{1}+\zeta_{2}+\cdots$ in powers of $\lambda^{\theta}$ up to a sufficiently high order. More precisely, we define the first term by $\zeta_{1}\sim \lambda^{\theta}P_{\lambda}$ where $P_{\lambda}$ is a suitable localized profiled related to the parameters $\lambda$. To obtain an additional decisive orthogonality condition for the remainder term, the approximate solution involve an extra term $bP_{b}$, that is, we should consider $\omega\sim Q+\lambda^{\theta}P_{\lambda}+bP_{b}$. Roughly speaking, the geometrical parameters $(\lambda,b)$ satisfy
\begin{equation*}
    \frac{\lambda_{s}}{\lambda}+\lambda^{\theta}+b=0\quad \mbox{and}\quad 
    b_{s}-b\lambda^{\theta}=0.
\end{equation*}
Note that, the above ODE admits a special blow-up solution 
\begin{equation*}
    \lambda(s)=|s|^{-\frac{1}{\theta}}\ \ \mbox{for} \ \ s\in (-\infty,0)
    \Longleftrightarrow \lambda(t)=t^{\frac{1}{3-\theta}}\ \ \mbox{for}\ \ t\in (0,\infty),
\end{equation*}
which gives a finite time blow-up solution of~\eqref{CP}.

\smallskip
We briefly discuss some technical aspects of the proof. First, from the above ODE argument and the energy-virial Lyapunov functional which was first introduced in~\cite{CLY}, we construct a sequence of approximate blow-up solutions $(\phi_{n})_{n\in \mathbb{N}^{+}}$ of~\eqref{CP} (see Section~~\ref{S:Modula}--\ref{S:EV} for more details). More precisely, for each $n\in \mathbb{N}^{+}$, the solution $\phi_{n}$ is well-defined on $[T_{n},T_{0}]$ with $T_{n}\to 0$ as $n\to \infty$ and the scaling parameter satisfies $\lambda_{n}(t)\sim t^{\frac{1}{3-\theta}}$. Compared to previous works on the construction of a blow-up solution (see \emph{e.g.}~\cite{Krieger,MMR1,MartelPilodAnn}), a significant problem encountered when studying the ZK equation is the lack of uniform control for the $H^{1}$ norm of $(\phi_{n})_{n\in \mathbb{N}^{+}}$. Actually, such uniform control should not exists since the noexistence of minimal mass element in $H^{1}(\R^{2})$. To overcome the difficulty, we employ some linear estimates and space-time mixed estimates to obtain a uniform control for the $H^{\frac{7}{9}}$ norm of $(\phi_{n})_{n\in \mathbb{N}^{+}}$ (see Section~\ref{SS:Uniform} and Appendix~\ref{App:linear} for more details). Based on such control and weak $H^{\frac{7}{9}}$ stability of ZK flow, we argue by compactness and obtain the minimal mass blow-up solution $\phi$ as the weak limit of the sequence of solutions $(\phi_{n})_{n\in \mathbb{N}^{+}}$.

\subsection{Notation and conventions}\label{SS:Nota}
For any $\beta=(\beta_{1},\beta_{2})\in \mathbb{N}^{2}$, we denote 
\begin{equation*}
    |\beta|=|\beta_{1}|+|\beta_{2}|\quad \mbox{and}\quad 
    \partial_{y}^{\beta}=\frac{\partial^{|\beta|}}{\partial_{y_{1}}^{\beta_{1}}\partial_{y_{2}}^{\beta_{2}}}.
\end{equation*}
For any $(r,j)\in (-1,\infty)\times \{1,2\}$ and regular function $f:\R^{2}\to \R$, we denote the fractional derivatives of $f$ by
\begin{equation*}
D_{x_{j}}^{r}f(x)=\left(\mathcal{F}^{-1}\left(|\xi_{j}|^{r}\mathcal{F}f\right)\right)(x),\quad \mbox{on}\ \R^{2}.
\end{equation*}

Denote by $\mathcal{Y}(\mathbb{R}^N)$ the set of smooth function $f$ on $\mathbb{R}^N$, such that for all $n\in\mathbb{N}$, there exist $r_n>0$ and $C_n>0$ such that
\begin{equation*}
\sum_{|\beta|=n}|\partial_{y}^{\beta}f(y)|\leq C_n(1+|y|)^{r_n}e^{-|y|},\quad \mbox{on}\ \R^{N}.
\end{equation*}

Denote also by $\mathcal{Z}(\R^{N})$ the set of smooth functions $f$ on $\R^{N}$, such that for all $n\in \mathbb{N}$, there exist $r_{n}>0$ and $C_{n}>0$ such that 
\begin{equation*}
    \sum_{|\beta|=n}\left|\partial_{y}^{\beta}f(y)\right|
    \le C_{n}(1+|y|)^{r_{n}}e^{-\frac{|y|}{2}},\quad \mbox{on}\ \R^{N}.
\end{equation*}
For any $f\in L^{2}(\R^{2})$ and $g\in L^{2}(\R^{2})$, we denote the $L^{2}$-scalar product by 
\begin{equation*}
    (f,g)=\int_{\R^{2}}f(x)g(x)\dd x.
\end{equation*}
In addition, for any $f\in L^{2}(\R^{2})$, we define the weighted norm
\begin{equation*}
	\|f\|_{L^{2}_{\rm{loc}}}=\left(\int_{\R^{2}}f^{2}(y)e^{-\frac{|y|}{10}}\dd y\right)^{\frac{1}{2}}.
\end{equation*}

Let $\chi:\R\to [0,1]$ be a $C^{\infty}$ nondecreasing function such that
\begin{equation*}
    \chi_{|(-\infty,-2)}\equiv 0\quad \mbox{and}\quad 
    \chi_{|(-1,\infty)}\equiv 1.
\end{equation*}

Let $\sigma:\R \to [0,1]$ be a $C^{\infty}$ cut-off function such that
\begin{equation*}
\sigma_{|(-1,1)}\equiv1\quad \mbox{and}\quad 
\sigma_{|(-\infty,-2)}\equiv\sigma_{|(2,\infty)}\equiv0.
\end{equation*}

Let ${\textbf{e}}_{1}$ be the first vector of the canonical basis of $\R^{2}$, that is, ${\textbf{e}}_{1}=(1,0)$.

\smallskip
We denote the scaling operator by
\begin{equation*}
\Lambda f= f+y\cdot \nabla f,\quad \mbox{for}\ f\in H^{1}(\R^{2}).
\end{equation*}
Moreover, we denote the linearized operator around $Q$ by
\begin{equation}\label{def:L}
	\mathcal{L}f=-\Delta f+f-3Q^{2}f,\quad \mbox{for}\ f\in H^{1}(\R^{2}).
\end{equation}

For all $(f_1,f_2)\in C^\infty(\mathbb{R})\times C^{\infty}(\R)$, we set
\begin{equation*}
    (f_1\otimes f_2)(y_1,y_2)=f_1(y_1)f_2(y_2),\quad \mbox{on}\ \R^{2}.
\end{equation*}

Recall that, we consider the following two smooth functions:
\begin{equation*}
    \widehat{F}(\xi)=\frac{1}{\sqrt{2\pi}}\int_{\R}F(y_{2})e^{-i\xi y_{2}}\dd y_{2}\quad \mbox{where}\quad F(y_{2})=\int_{\R}\Lambda Q(y_{1},y_{2})\dd y_{1}.
\end{equation*}
In addition, the constant $\theta$ is defined by 
\begin{equation*}
    \theta
    =2\bigg(\int_{\R}\frac{|\widehat{F}(\xi)|^{2}}{1+|\xi|^{2}}\dd \xi\bigg)\bigg/
    \bigg(\int_{\R}|\widehat{F}(\xi)|^{2}\dd \xi\bigg).
\end{equation*}
By using the Fast Fourier Transform in Mathematica, one has $\theta\approx 1.66032$. We refer to~\cite[Appendix A]{CLY}) for the numerical computation related to $\theta$.  

\smallskip
For a given small constant $\alpha$, we denote by $\varrho(\alpha)$ a generic small constant with 
\begin{equation*}
    \varrho(\alpha)\to 0,\quad \mbox{as}\ \ \alpha\to 0.
\end{equation*}

\smallskip
We denote by $h_{2}\in \mathcal{Y}(\R)$ the even solution of the following second-order ODE:
\begin{equation}\label{def:h2}
-h_{2}''(y_{2})+h_{2}(y_{2})=F''(y_{2}),\quad \mbox{on}\ \R.
\end{equation}
Then, we denote
\begin{equation*}
G: y_{1}\longmapsto \int_{\R}h_{2}(y_{2})Q(y_{1},y_{2})\dd y_{2}.
\end{equation*}
We fix a regular function $h_{1}\in \mathcal{Y}(\R)$ such that $\int_{\R}h_{1}(y_{1})\dd y_{1}=1$ and $h_{1}$ is orthogonal to $G$ in the $L^{2}(\R)$ sense. It follows that 
\begin{equation*}
\left(h_{1}\otimes h_{2},Q\right)=\int_{\R}h_{1}(y_{1})\left(\int_{\R}h_{2}(y_{2}) Q(y_{1},y_{2})\dd y_{2}\right)\dd y_{1}=0.
\end{equation*}

As in the previous works of Chen-Lan-Yuan~\cite{CLY,CLYMINI}, we must always carefully check the dependence on the large-scale constant $B$. The following conventions are set: the implied constants in $\lesssim$ and $O$ are \emph{independent of} $B$ from Section~\ref{S:Blow-up} to Section~\ref{S:EV} and can \emph{depend on} the large constant $B$ in Section~\ref{S:Endproof}.

\subsection*{Acknowledgments}
The authors would like to thank Prof. Yvan Martel for helpful discussions on topics related to this work.

\section{Blow-up profile}\label{S:Blow-up}
In this section, we deal with the construction of a suitable blow-up profile for~\eqref{equ:rescal}, that is, an approximate solution close to $Q$ of the rescaled equation~\eqref{equ:rescal} with the expected blow-up behavior and a sufficiently high order of precision.
\subsection{The linearized operator}
 In this subsection, we recall some standard properties of
 the linearized operator $\mathcal{L}$ and then introduce useful functions related to the construction of a suitable blow-up profile.
 \begin{proposition}
     [Spectral theory of $\mathcal{L}$]\label{prop:Spectral}
     The self-adjoint operator $\mathcal{L}$ on $L^{2}(\R^{2})$ defined by~\eqref{def:L} satisfies the following properties.

     \begin{enumerate}
         \item \emph{Spectrum of $\mathcal{L}$}. 
         The operator $\mathcal{L}$ has only one negative eigenvalue $-\mu_{0}$ with $\mu_{0}>0$, ${\rm{Ker}}\mathcal{L}={\rm{Span}}\left(\partial_{y_{1}}Q,\partial_{y_{2}}Q\right)$ and 
         $\sigma_{\rm{ess}}\left(\mathcal{L}\right)=[1,\infty)$.
We denote by $Y$ the $L^{2}$-normalized eigenvector of $\mathcal{L}$  corresponding to the eigenvalue $-\mu_{0}$. It holds, for all $\beta\in \mathbb{N}^{2}$,
\begin{equation*}
    \left|\partial_{y}^{\beta}Y(y)\right|\lesssim e^{-|y|},\quad \mbox{on}\ \R^{2}.
\end{equation*}
         \item \emph{Scaling identities.} It holds 
         \begin{equation*}
            \mathcal{L}\Lambda Q=-2Q\quad \mbox{and}\quad 
            (\Lambda Q, Q)=0.
         \end{equation*}

         \item \emph{Coercivity of $\mathcal{L}$.}
         There exists $\mu_{1}>0$ such that, for all $f\in H^{1}(\R^{2})$, 
         \begin{equation*}
         \begin{aligned}
             \left(\mathcal{L}f,f\right)&\ge 
             \mu_{1} \|f\|_{H^{1}}^{2}-\frac{1}{\mu_{1}}\big(\left(f,Y\right)^{2}+\left|\left(f,\nabla Q\right)\right|^{2}\big),\\
              \left(\mathcal{L}f,f\right)&\ge 
             \mu_{1} \|f\|_{H^{1}}^{2}-\frac{1}{\mu_{1}}\big(\left(f,Q^{3}\right)^{2}+\left|\left(f,\nabla Q\right)\right|^{2}\big).
             \end{aligned}
         \end{equation*}

         \item \emph{Inversion of $\mathcal{L}$.} 
         For any $g\in L^{2}(\R^{2})$ orthogonal to $\nabla Q$ in the $L^{2}$ sense, there exists a unique function $f\in H^{2}(\R^{2})$ orthogonal to $\nabla Q$ in the $L^{2}$ sense such that $\mathcal{L}f=g$. Moreover, if $g\in \mathcal{Y}(\R^{2})$, then we have $f\in \mathcal{Z}(\R^{2})$. In addition, if $g$ is even in either $y_{1}$ or $y_{2}$, then we have $f$ is also even in $y_{1}$ or $y_{2}$.
     \end{enumerate}
 \end{proposition}

 \begin{proof}
     The properties of $\mathcal{L}$ in (i)--(iv) are easily checked. We refer to~\cite[Theorem 3.1 and Lemma 3.2]{FHR} and~\cite[Proposition 2.1]{CLY} for the details of the proofs.
 \end{proof}

\smallskip
We now recall the existence of a non-localized profile 
from~\cite[Lemma 2.2]{CLY}.
\begin{lemma}[Non-localized profile]\label{le:Nonloca}
There exists a smooth function $P\in C^{\infty}(\R^{2})$ such that $\partial_{y_{1}}P\in \mathcal{Z}(\R^{2})$ with $P$ is even in $y_{2}$ and the following properties are true.
\begin{enumerate}
	\item \emph{First property of $P$.} We have 
\begin{equation*}
\partial_{y_{1}}\mathcal{L}P=\Lambda Q,\quad \lim_{y_{1}\to \infty}\partial_{y_{2}}^{\beta_{2}}P(y_{1},y_{2})=0,\ \forall \beta_{2}\in \mathbb{N},
\end{equation*}
\begin{equation*}
|(\nabla P,Q)|=0\quad \mbox{and}\quad (P,Q)=\frac{1}{4}\int_{\mathbb{R}}|F(y_{2})|^{2}\dd y_2.
\end{equation*}
In addition, we have 
\begin{equation*}
	\begin{aligned}
	\lim_{y_{1}\to -\infty}P(y_{1},y_{2})&=-F(y_{2})-h_{2}(y_{2}),\\
	\lim_{y_{1}\to -\infty}\partial_{y_{2}}P(y_{1},y_{2})&=-F'(y_{2})-h'_{2}(y_{2}).
	\end{aligned}
\end{equation*}

\item \emph{Second property of $P$.} 
For any $\beta=(\beta_{1},\beta_{2})\in \mathbb{N}^{2}$ with $\beta_{1}\ne 0$, we have 
\begin{equation*}
\left|\partial_{y}^{\beta}P(y_{1},y_{2})\right|\lesssim e^{-\frac{|y|}{3}},\quad \ \mbox{on}\ \mathbb{R}^{2}.
\end{equation*}
In addition, for any $\beta=(\beta_{1},\beta_{2})\in \mathbb{N}^{2}$, we have
\begin{equation*}
\begin{aligned}
\left|\partial_{y}^{\beta}P(y_{1},y_{2})\right|&\lesssim e^{-\frac{|y_{2}|}{3}},\quad \mbox{on}\ \R^{2},\\
\left|\partial_{y}^{\beta}P(y_{1},y_{2})\right|&\lesssim e^{-\frac{|y|}{3}},\quad \ \mbox{on}\ (0,\infty)\times \R.
\end{aligned}
\end{equation*}

\end{enumerate}
\end{lemma}
\begin{proof}
    The proof relies on the definition of $(h_{1},h_{2})\in \mathcal{Y}(\R)\times \mathcal{Y}(\R)$ in Section~\ref{SS:Nota} and (iv) of Proposition~\ref{prop:Spectral} (see \emph{e.g.}~\cite[Lemma 2.2]{CLY}), and we omit it.
\end{proof}

\begin{remark}\label{re:deftheta}
Using integration by parts and (i) of Lemma~\ref{le:Nonloca},
	\begin{equation*}
		\begin{aligned}
		(P,\partial_{y_2}^2\partial_{y_1}P)&=\frac{1}{2}\int_{\mathbb{R}}\left(F'(y_2)+h'_2(y_2)\right)^2\dd y_2,\\
			\left(P,\partial_{y_{1}}^{3}P-\partial_{y_{1}}P\right)&=\frac{1}{2}\int_{\mathbb{R}}\left(F(y_2)+h_2(y_2)\right)^2\dd y_2.
		\end{aligned}
	\end{equation*}
	Taking the Fourier transform on the both sides of~\eqref{def:h2}, we obtain
	\begin{equation*}
	\widehat{h_{2}}(\xi)=-\frac{|\xi|^{2}}{	1+|\xi|^{2}}\widehat{F}(\xi)\Longrightarrow
		\widehat{F}(\xi)+\widehat{h_{2}}(\xi)=\frac{\widehat{F}(\xi)}{1+|\xi|^2}.
	\end{equation*}
	It follows from Plancherel theorem that
	\begin{equation*}
		(P, \Delta \partial_{y_{1}}P-\partial_{y_{1}}P)=
		\frac{1}{2}\int_{\mathbb{R}}(1+|\xi|^2)\left(\widehat{F}(\xi)+\widehat{h_{2}}(\xi)\right)^2\dd \xi=\frac{1}{2}\int_{\mathbb{R}}\frac{|\widehat{F}(\xi)|^2}{1+|\xi|^2}\,\dd \xi.
	\end{equation*}
	On the other hand, using Lemma~\ref{le:Nonloca} and integration by parts, we have
	\begin{equation*}
			(\Lambda P,Q)=-(P,\Lambda Q)
			=\left(P, \Delta \partial_{y_{1}}P-\partial_{y_{1}}P\right)-3\left(\partial_{y_{1}} \left(QP^2\right),Q\right).
	\end{equation*}
	Combining the above identities with Lemma~\ref{le:Nonloca} and the definition of $\theta$, we find
	\begin{equation*}
		\theta=
		2\bigg(\int_{\R}\frac{|\widehat{F}(\xi)|^{2}}{1+|\xi|^{2}}\dd \xi\bigg)\bigg/
		\bigg(\int_{\R}|\widehat{F}(\xi)|^{2}\dd \xi\bigg)
		=\left(
	\Lambda P+	3\partial_{y_{1}} \left(QP^2\right),Q
		\right)/(P,Q).
	\end{equation*}
	We mention here that, the constant $\theta$ plays a crucial role in our construction of the blow-up profile since it appears in the leading-order term related to the profile (see Proposition~\ref{prop:W}  and~\cite[Lemma 2.3]{CLY} for more details).
\end{remark}
We now introduce the following function space $\mathcal{Z}_{k}(\R^{2})$ which is a generalization of $\mathcal{Z}(\R^{2})$. Roughly speaking, the function belonging to $\mathcal{Z}_{k}(\R^{2})$ is allowed to have polynomial growth (degree not greater than $k$) as $y_{1}\to -\infty$.
\begin{definition}
For any $k\in \mathbb{N}$, we denote by $\mathcal{Z}_{k}(\R^{2})\subset C^{\infty}(\R^{2})$ the function space consisting of all functions $f\in C^{\infty}(\R^{2})$ which satisfy
\begin{equation*}
    \partial_{x_{1}}^{k+1}f\in \mathcal{Z}(\R^{2})\quad \mbox{and}\quad 
    \left(\partial_{x_{1}}^{\ell}f\right)\chi\in \mathcal{Z}(\R^{2})\ \ \mbox{for any}\ \ell\in \left\{0,1,\dots,k\right\}.
\end{equation*}
In addition, we denote $\mathcal{Z}_{-1}(\R^{2})=\mathcal{Z}(\R^{2})$.
    
\end{definition}

\begin{lemma}\label{le:L1}
Let $k\in \mathbb{Z}\cap [-1,\infty)$. Suppose $g\in \mathcal{Z}_{k}$ with $(g,Q)=0$ and $g$ is even in $y_2$.
Then there exists a smooth function $f\in\mathcal{Z}_{k+1}$ such that
\begin{equation*}
    \partial_{y_1}\mathcal{L}f=g\quad \mbox{and}\quad |(f,\nabla Q)|=0.
\end{equation*}
In addition, we have $f$ is also even in $y_2$.
\end{lemma}
\begin{proof}
We prove Lemma~\ref{le:L1} by induction.

\smallskip
{\textbf{Step 1.}} For $k=-1$. The proof is similar to~\cite[Lemma 2.2]{CLY}. Below, we provide a sketch of the proof
for the sake of completeness and for the readers’ convenience.
We set 
\begin{equation*}
	-G''_{2}(y_{2})+G_{2}(y_{2})=H''_{2}(y_{2}),\quad \mbox{where}\ \
	H_{2}(y_{2})=\int_{\R}g(\tau,y_{2})\dd \tau.
\end{equation*}
We consider $f\in C^{\infty}(\R)$ of the form
\begin{equation*}
	f(y_{1},y_{2})=\widetilde{f}(y_{1},y_{2})-\int_{y_{1}}^{\infty}g(\tau,y_{2})\dd \tau-\int_{y_{1}}^{\infty}h_{1}(\tau)G_{2}(y_{2})\dd \tau-c\partial_{y_{1}}Q.
\end{equation*}
Here, we consider $\widetilde{f}\in \mathcal{Z}_{-1}(\R^{2})$ and $c\in \R$ to be chosen later. By an elementary computation, we find $\partial_{y_{1}}\mathcal{L}f=g$ is equivalent to 
\begin{equation*}
\partial_{y_{1}}\mathcal{L}\widetilde{f}=g+\partial_{y_{1}}\mathcal{L}\int_{y_{1}}^{\infty}\left(g(\tau,y_{2})+h_{1}(\tau)G(y_{2})\right)\dd \tau=\partial_{y_{1}}R.
\end{equation*}
Here, we denote 
\begin{equation*}
	\begin{aligned}
	R(y_{1},y_{2})
	&=\partial_{y_{1}}g(y_{1},y_{2})
	-\int_{y_{1}}^{\infty}\partial_{y_{2}}^{2}g(\tau,y_{2})\dd \tau
	-3Q^{2}\int_{y_{1}}^{\infty}g(\tau,y_{2})\dd \tau\\
	&+h_{1}'(y_{1})G_{2}(y_{2})
	+\int_{y_{1}}^{\infty}h_{1}(\tau)H''_{2}(y_{2})\dd \tau
	-3Q^{2}\int_{y_{1}}^{\infty}h_{1}(\tau) G_{2}(y_{2})\dd \tau.
\end{aligned}
\end{equation*}
From $(g,Q)=0$ and the definition of $R$, we find $R\in \mathcal{Y}(\R^{2})$ with $\left|(R,\nabla Q)\right|=0$ and $R$ is even in $y_{2}$.
Therefore, using (iv) of Proposition~\ref{prop:Spectral}, there exists
$\widetilde{f}\in \mathcal{Z}(\R^{2})$ such that 
\begin{equation*}
\partial_{y_{1}}	\mathcal{L}\widetilde{f}=\partial_{y_{1}}R\  \mbox{and $\widetilde{f}$ is even in $y_{2}$} \Longrightarrow 
\partial_{y_{1}}	\mathcal{L}{f}=g\ \mbox{and $f$ is even in $y_{2}$}.
\end{equation*} 
From $\widetilde{f}\in \mathcal{Z}(\R^{2})$ and the definition of $f$, we know that $f\in \mathcal{Z}_{0}(\R^{2})$. Then, by choosing a suitable constant $c\in \R$, we obtain $\left|(f,\nabla Q)\right|=0$ which means that this Lemma is true for $k=-1$.

\smallskip
{\textbf{Step 2.}} Assume that this Lemma is true for $k\in \mathbb{Z}\cap [-1,\infty)$, we prove that it is also true for $k+1$. First, from the Fundamental theorem, for any $g\in \mathcal{Z}_{k+1}(\R^{2})$ with $(g,Q)=0$ and $g$ is even in $y_{2}$,
there exist even functions $(H_{2\ell})_{\ell=0}^{k+1}\subset \mathcal{Z}(\R)$ such that 
\begin{equation*}
    g-\sum_{\ell=0}^{k+1}H_{1\ell}\otimes H_{2\ell}\in \mathcal{Z}(\R^{2}),\quad \mbox{with}\ \  H_{1\ell}(y_{1})=(1-\chi(y_{1}))y_{1}^{\ell}.
\end{equation*}
Here, we use the fact that
\begin{equation*}
	\partial_{y_{1}}^{\ell}g(y_{1},y_{2})=\partial_{y_{1}}^{\ell}g(0,y_{2})+\int_{0}^{y_{1}}\partial_{y_{1}}^{\ell+1}g(\tau,y_{2})\dd \tau,\quad \mbox{for any}\ \ell\in \left\{0,1,\dots,k\right\}.
\end{equation*}
For any $\ell\in \left\{0,1,\dots,k+1\right\}$, we denote by $G_{2\ell}\in \mathcal{Z}(\R)$ the smooth solution of the following second-order ODE:
\begin{equation*}
   - G''_{2\ell}(y_{2})+G_{2\ell}(y_{2})=H''_{2\ell}(y_{2}),\quad \mbox{on}\ \ \R.
\end{equation*}
Note that, from $(H_{2\ell})_{\ell=0}^{k+1}$ are even functions, we obtain $(G_{2\ell})_{\ell=0}^{k+1}$ are also even functions.
We now consider $f\in C^{\infty}(\R^{2})$ of the form
\begin{equation*}
    f(y_{1},y_{2})=\widetilde{f}(y_{1},y_{2})-\int_{y_{1}}^{\infty}g(\tau,y_{2})\dd \tau
    -\sum_{\ell=0}^{k+1}\int_{y_{1}}^{\infty}H_{1\ell}(\tau)G_{2\ell}(y_{2})\dd \tau
    -c\partial_{y_{1}}Q.
\end{equation*}
Here, we consider $\widetilde{f}\in \mathcal{Z}_{k}(\R^{2})$ and $c\in \mathbb{R}$ to be chosen later. 
Note that, the equation $\partial_{y_{1}}\mathcal{L}f=g$ is equivalent to 
\begin{equation*}
        \partial_{y_{1}}\mathcal{L}\widetilde{f}
=
g+\partial_{y_{1}}\mathcal{L}\left(\int_{y_{1}}^{\infty}g(\tau,y_{2})\dd \tau\right)
        +\sum_{\ell=0}^{k+1}\partial_{y_{1}}\mathcal{L}
        \left(\int_{y_{1}}^{\infty}H_{1\ell}(\tau)G_{2\ell}(y_{2})\dd \tau\right)=\widetilde{g}.
\end{equation*}
On the one hand side, from $(g,Q)=0$ and $\partial_{y_{1}}Q\in {\rm{Ker}}\mathcal{L}$, we obtain $(\widetilde{g},Q)=0$. On the other hand, by an elementary computation, we find
\begin{equation*}
	\begin{aligned}
    \widetilde{g}&=\partial_{y_{1}}^{2}g+\partial_{y_{2}}^{2}\left(g-\sum_{\ell=0}^{k+1}H_{1\ell}\otimes H_{2\ell}\right)+\sum_{\ell=0}^{k+1}
        H''_{1\ell}\otimes H_{2\ell}\\
        &-3\partial_{y_{1}}\left(Q^{2}\int_{y_{1}}^{\infty}g(\tau,y_{2})\dd \tau\right)
        -3\partial_{y_{1}}\left(Q^{2}\int_{y_{1}}^{\infty}H_{1\ell}(\tau)G_{2\ell}(y_{2})\dd \tau\right).
        \end{aligned}
\end{equation*}
From the above identity, the definition of $(H_{1\ell}, H_{2\ell})$ and $g\in \mathcal{Z}_{k+1}(\R^{2})$, we deduce that $\widetilde{g}\in \mathcal{Z}_{k-1}(\R^{2})$. With the assumption of the Induction Hypothesis in hand, we know that there exists $\widetilde{f}\in \mathcal{Z}_{k}(\R^{2})$ such that $\partial_{y_{1}}\mathcal{L}\widetilde{{f}}=\widetilde{g}$ and $\widetilde{f}$ is even in $y_{2}$. Therefore, from the definition of $\widetilde{f}$ and ${f}$, there exists $f\in \mathcal{Z}_{k+2}(\R^{2})$ such that 
$\partial_{y_{1}}\mathcal{L}f=g$ and $f$ is even in $y_{2}$. By choosing a suitable constant $c\in \R$, we obtain $|(f,\nabla Q)|=0$ which means that this Lemma is true for $k+1$. By the induction argument, Lemma~\ref{le:L1} is proved for any $k\in \mathbb{Z}\cap [-1,\infty)$.
\end{proof}

\begin{remark}\label{re:decayZk}
Note that, from the definition of $\mathcal{Z}_{k}(\R^{2})$ and the Fundamental Theorem, for any $f\in \mathcal{Z}_{k}(\R^{2})$ with $k\in \mathbb{N}$ and $\beta=(\beta_{1},\beta_{2})\in \mathbb{N}^{2}$ with $\beta_{1}\ge k+1$, 
\begin{equation*}
    \left|\partial_{y}^{\beta}f(y_{1},y_{2})\right|\lesssim e^{-\frac{|y|}{3}},\quad \mbox{on}\ \ \R^{2}.
\end{equation*}
Note also that, for any $f\in \mathcal{Z}_{k}(\R^{2})$ with $k\in \mathbb{N}$ and $\beta=(\beta_{1},\beta_{2})\in \mathbb{N}^{2}$ with $0\le \beta_{1}\le k$, we have 
\begin{equation*}
\begin{aligned}
\left|\partial_{y}^{\beta}f(y_{1},y_{2})\right|&\lesssim e^{-\frac{|y|}{3}}+\left(1+|y_{1}|^{k-\beta_{1}}\right)e^{-\frac{|y_{2}|}{3}}\textbf{1}_{(-\infty,0)}(y_{1}),\quad \mbox{on}\ \R^{2}.
\end{aligned}
\end{equation*}

\end{remark}

\subsection{Construction of the Blow-up profile}\label{SS:Blowup}
Let $s_{0}<0$ with $|s_{0}|\gg 1$ to be chosen later and let $I\subset (-\infty,s_{0}]$ be a compact interval.
We assume $\lambda:I\to (0,\infty)$ is a $C^1$ real-valued function defined on $I$ such that 
  $0<\lambda(s)\ll 1$ for any $s\in I$.

\smallskip
We define
\begin{equation*}
    \Theta(s,y_1)=\chi\big(\lambda^\frac{8}{5}(s) y_1\big),\quad \mbox{on}\ \ (s,y_{1})\in I\times \R.
\end{equation*}
In addition, for smooth functions
\begin{equation}\label{equ:defA}
	(A_2,A_3,A_4)\in \mathcal{Z}_{1}(\R^{2})\times\mathcal{Z}_{2}(\R^{2})\times\mathcal{Z}_{3}(\R^{2}),
\end{equation}
 and real constants $(c_{1},c_2,c_3)\in \mathbb{R}^{3}$ to be chosen later, we define
\begin{equation*}
    X_1=c_1P, \quad X_2=c_2 P+A_2,\quad X_3=c_3P+A_3\quad \mbox{and}\quad X_4=A_4.
\end{equation*}
We now consider the following blow-up profile:
\begin{equation*}
    W(s,y)=Q(y)+V(s,y),\ \ \mbox{where}\ \ 
    V(s,y)=\sum_{k=1}^{4}V_{k}(s,y)=
    \sum_{k=1}^4\lambda^{k\theta}(s)X_k(y)\Theta(s,y_{1}).
\end{equation*}
\begin{lemma}\label{le:estW1}
    For any smooth functions $(A_2,A_3,A_4)$ satisfying~\eqref{equ:defA} and real constants $(c_{1},c_2,c_3)\in \mathbb{R}^{3}$, the following estimates hold.
    \begin{enumerate}

\item\emph{Closeness to $Q$.} It holds
\begin{equation*}
       \left| \int_{\R^{2}}W^{2}\dd y-\int_{\R^{2}}Q^{2}\dd y\right|\lesssim \lambda^{\theta}\quad \mbox{and}\quad \|W-Q\|_{H^{1}}\lesssim \lambda^{\theta-\frac{4}{5}}.
\end{equation*}
        \item \emph{Estimates of $W$.} For any $\beta=(\beta_{1},\beta_{2})\in \mathbb{N}^{2}$, we have 
        \begin{equation*}
\left|\partial_{y}^{\beta}W\right|+
\left|\partial_{y}^{\beta} \Lambda W\right|\lesssim
e^{-\frac{|y|}{4}}+
 \lambda^{\theta+\frac{8}{5}\beta_{1}}
e^{-\frac{|y_2|}{4}}
\mathbf{1}_{[-2,0]}(\lambda^{\frac{8}{5}} y_1).
\end{equation*}
    \end{enumerate}
\end{lemma}
\begin{proof}
    Proof of (i). From the definition of $V$ and $W$, we find 
    \begin{equation*}
    	\begin{aligned}
\left|W^{2}-Q^{2}\right|\lesssim \sum_{k=1}^{4}Q|V_{k}|+\sum_{k=1}^{4}V_{k}^{2}.
\end{aligned}
    \end{equation*}
   Using the fact that $\theta>\frac{8}{5}$, $Q\in \mathcal{Y}(\R^{2})$ and Remark~\ref{re:decayZk}, we deduce that 
    \begin{equation*}
      \sum_{k=1}^{4} \left |\int_{\R^{2}}Q
       \left|V_{k}\right|
       \dd y\right|\lesssim \lambda^{\theta}\int_{\R^{2}}\left(1+|y_{1}|^{3}\right)e^{-\frac{|y|}{2}}\dd y\lesssim \lambda^{\theta}.
    \end{equation*}
    Then, using again the fact that $\theta>\frac{8}{5}$ and  Remark~\ref{re:decayZk}, for any $k\in \left\{1,2,3,4\right\}$, 
\begin{equation*}
\begin{aligned}
    \int_{\R^{2}}V_{k}^{2}\dd y
    &\lesssim
    \lambda^{2k\theta}
    \int_{\R^{2}}\left(e^{-\frac{|y|}{3}}+\left(1+|y_{1}|^{2(k-1)}\right)\textbf{1}_{[-2,0]}\left(\lambda^{\frac{8}{5}}y_{1}\right)e^{-\frac{|y_{2}}{3}|}\right)\dd y\\
    &\lesssim \lambda^{2k\theta}+\lambda^{2k\theta-\frac{8}{5}(2k-1)}\lesssim
    \lambda^{2k\theta}+
    \lambda^{\theta+\left(\theta-\frac{8}{5}\right)(2k-1)}
    \lesssim \lambda^{\theta}.
    \end{aligned}
\end{equation*}
Using a similar argument as above, we also obtain 
\begin{equation*}
    \|W-Q\|_{H^{1}}\lesssim \sum_{k=1}^{4}\|V_{k}\|_{H^{1}}\lesssim \lambda^{\theta-\frac{4}{5}}.
\end{equation*}
Combining the above estimates, we complete the proof of (i).

\smallskip
Proof of (ii). By an elementary computation, for any $\beta=(\beta_{1},\beta_{2})$ and $y\in \R^{2}$,
\begin{equation*}
	\begin{aligned}
   \partial_{y}^{\beta}V_{k}
   &=\lambda^{k\theta}
   \sum_{\substack{0\le \beta_{11}\le k-1\\ \beta_{11}+\beta_{12}=\beta_{1}}}
   \left(\partial_{y_{2}}^{\beta_{2}}\partial_{y_{1}}^{\beta_{11}}X_{k}\right)
   \left(\partial_{y_{1}}^{\beta_{12}}\Theta\right)\\
   &+\lambda^{k\theta}
   \sum_{\substack{k\le \beta_{11}\le \beta_{1}\\ \beta_{11}+\beta_{12}=\beta_{1}}}
   \left(\partial_{y_{2}}^{\beta_{2}}\partial_{y_{1}}^{\beta_{11}}X_{k}\right)
   \left(\partial_{y_{1}}^{\beta_{12}}\Theta\right).
   \end{aligned}
\end{equation*}
It follows from $\theta>\frac{8}{5}$ and Remark~\ref{re:decayZk} that 
\begin{equation*}
	\begin{aligned}
	\left|\partial_{y}^{\beta}V_{k}\right|
	&\lesssim
	\lambda^{k\theta}
	\left(e^{-\frac{|y|}{3}}+\lambda^{\frac{8}{5}\beta_{12}}\langle y_{1}\rangle^{k-1-\beta_{11}}e^{-\frac{|y_2|}{3}}
	\mathbf{1}_{[-2,0]}(\lambda^{\frac{8}{5}} y_1)\right)\\
	&\lesssim \lambda^{k\theta}\left(e^{-\frac{|y|}{3}}+\lambda^{\frac{8}{5}\beta_{1}-\frac{8}{5}(k-1)}
	e^{-\frac{|y_2|}{3}}
	\mathbf{1}_{[-2,0]}(\lambda^{\frac{8}{5}} y_1)
	\right)\\
	&\lesssim 
\lambda^{k\theta}
	e^{-\frac{|y|}{3}}+
		\lambda^{\theta+\frac{8}{5}\beta_{1}}
	e^{-\frac{|y_2|}{3}}
	\mathbf{1}_{[-2,0]}(\lambda^{\frac{8}{5}} y_1).
	\end{aligned}
\end{equation*}
Combining the above estimates with $Q\in \mathcal{Y}(\R^{2})$, we complete the proof of (ii).
\end{proof}

Last, we consider the following error term related to the blow-up profile $W$:
\begin{equation*}
    \Gamma(s)=\sum_{k=1}^3 c_k\lambda^{ k\theta}(s)\quad
    \mbox{and}\quad 
\mathcal{E}(W)=\partial_sW+\partial_{y_1}\left(\Delta W-W+W^3\right)-\frac{\lambda_s}{\lambda}\Lambda W.
\end{equation*}

To simplify notation, we denote
\begin{equation*}
	\begin{aligned}
		V^{2}=\sum_{k=2}^{8}\lambda^{k\theta}N_{k}\Theta^{2}\quad 
		&	\mbox{and}\quad 
		V^{3}=\sum_{k=3}^{12}\lambda^{k\theta}M_{k}\Theta^{3},\\
		N_{k}=\sum_{k_{1}+k_{2}=k}X_{k_{1}}X_{k_{2}}\quad 
		&	\mbox{and}\quad 
		M_{k}=\sum_{k_{1}+k_{2}+k_{3}=k}X_{k_{1}}X_{k_{2}}X_{k_{3}}.
	\end{aligned}
\end{equation*}
We also denote 
\begin{equation*}
	\begin{aligned}
		\Psi_{W}&=\Psi_{1W}+\Psi_{2W}+\Psi_{3W}+\Psi_{4W}+\Gamma\Lambda Q(1-\Theta),\\
		\Psi_{\lambda}&=
		\sum_{k=1}^{4}\lambda^{k\theta}\left(\Lambda X_{k}-k\theta X_{k}\right)\Theta
		-\frac{3}{5}\sum_{k=1}^{4}\lambda^{k\theta}X_{k}(y_{1}\partial_{y_{1}}\Theta).
	\end{aligned}
\end{equation*}
Here, we set 
\begin{equation*}
	\begin{aligned}
		\Psi_{1W}&=-\frac{3\Gamma}{5}
		\sum_{k=1}^{4}\lambda^{k\theta}X_{k}\left(y_{1}\partial_{y_{1}}\Theta\right)+\sum_{k=5}^{7}\lambda^{k\theta}\left(\sum_{\ell=k-3}^{4}c_{k-\ell}\left(\Lambda X_{\ell}-\ell\theta X_{\ell}\right)\right)\Theta,\\
		\Psi_{2W}&=-\sum_{k=1}^{4}\lambda^{k\theta}
		\left(\mathcal{L}X_{k}-2\partial_{y_{1}}^{2}X_{k}\right)\partial_{y_{1}}\Theta+\sum_{k=1}^{4}\lambda^{k\theta}
		\left(3\left(\partial_{y_{1}}X_{k}\right)\partial_{y_{1}}^{2}\Theta+X_{k}\partial_{y_{1}}^{3}\Theta\right), 
	\end{aligned}
\end{equation*}
\begin{equation*}
	\begin{aligned}
		\Psi_{3W}&=\sum_{k=3}^{4}\lambda^{k\theta}
		\left(M_{k}\partial_{y_{1}}\Theta+
		\partial_{y_{1}}\left(M_{k}\left(\Theta^{3}-\Theta\right)\right)
		\right)+\sum_{k=5}^{12}\lambda^{k\theta}\partial_{y_{1}}\left(M_{k}\Theta^{3}\right),\\
		\Psi_{4W}&=3\sum_{k=2}^{4}\lambda^{k\theta}
		\left(QN_{k}\partial_{y_{1}}\Theta
		+\partial_{y_{1}}
		\left(QN_{k}\left(\Theta^{2}-\Theta\right)\right)
		\right)
		+3\sum_{k=5}^{8}\lambda^{k\theta}\partial_{y_{1}}\left(QN_{k}\Theta^{2}\right).
	\end{aligned}
\end{equation*}

\begin{proposition}\label{prop:W}
	Let $c_{1}=-\theta^{-1}$.
Then there exist smooth functions $\left(A_2,A_3,A_4\right)$ satisfying~\eqref{equ:defA} and real constants $(c_2,c_3)\in \R^{2}$ such that 
\begin{equation*}
\mathcal{E}(W)=-\left(\frac{\lambda_s}{\lambda }+{\Gamma}\right)
\left(\Lambda Q+\Psi_\lambda\right)+\Psi_W.
\end{equation*}
In addition, the error terms $\Psi_{\lambda}$ and $\Psi_{W}$ satisfy the following estimates.

\begin{enumerate}
	\item \emph{Estimates for $\Psi_{\lambda}$.} 
	For any $\beta=(\beta_{1},\beta_{2})$ with $\beta_{1}\in \left\{0,1,2,3\right\}$, we have 
	\begin{equation*}
		|\partial_{y}^{\beta} \Psi_\lambda|
	\lesssim 
	\lambda^{\theta}e^{-\frac{|y|}{4}}
	+
	\lambda^{\theta+\frac{8}{5}\beta_{1}}
	e^{-\frac{|y_2|}{4}}
	\mathbf{1}_{[-2,0]}(\lambda^{\frac{8}{5}} y_1).
	\end{equation*}
	\item \emph{Estimates for $\Psi_{W}$.}
For any $\beta=(\beta_{1},\beta_{2})$ with $\beta_{1}\in \left\{0,1,2,3\right\}$, we have 
		\begin{equation*}
			\begin{aligned}
			|\partial_{y}^{\beta}\Psi_W|
			&\lesssim
			\lambda^{5\theta}e^{-\frac{|y|}{4}}
			+\lambda^{\theta+\frac{8}{5}(1+\beta_{1})}
			e^{-\frac{|y_2|}{4}}
			\mathbf{1}_{[-2,-1]}(\lambda^{\frac{8}{5}} y_1)\\
			&+\lambda^{5\theta}\left( 1+y_{1} ^{3-\beta_{1}}\right)
			e^{-\frac{|y_2|}{4}}
			\mathbf{1}_{[-2,0]}(\lambda^{\frac{8}{5}} y_1).
		\end{aligned}
		\end{equation*}
\end{enumerate}
\end{proposition}

\begin{proof}
	\textbf{Step 1.} General computation.
Note that
\begin{equation*}
\begin{aligned}
\partial_sW &=\frac{\lambda_s}{\lambda}\sum_{k=1}^4
\big(k\theta\lambda^{k\theta}X_{k}\Theta+\frac{8}{5}\lambda^{k\theta}X_k(y_1\partial_{y_1}\Theta)\big),\\
\frac{\lambda_s}{\lambda}\Lambda W
&=\frac{\lambda_s}{\lambda}\Lambda Q
+\frac{\lambda_{s}}{\lambda}\sum_{k=1}^{4}\lambda^{k\theta}
\left(\left(\Lambda X_{k}\right)\Theta+X_{k}\left(y_1\partial_{y_1}\Theta\right)\right).
\end{aligned}
\end{equation*}
Then, from Lemma~\ref{le:Nonloca} and the definition of $W$, we find
\begin{equation*}
	\begin{aligned}
		&\partial_{y_{1}}\left(\Delta W-W+W^{3}\right)\\
		&=-\Gamma (\Lambda Q)\Theta
		+\partial_{y_{1}}\left(3Q V^{2}+V^{3}\right)
		-\sum_{k=2}^{4}\lambda^{k\theta}
		\left(\partial_{y_{1}}\mathcal{L}A_{k}\right)\Theta\\
		&-\sum_{k=1}^{4}\lambda^{k\theta}\left(\mathcal{L}X_{k}-2\partial_{y_{1}}^{2}X_{k}\right)\partial_{y_{1}}\Theta
		+\sum_{k=1}^{4}\lambda^{k\theta}
		\left(3\left(\partial_{y_{1}}X_k\right)\partial_{y_1}^{2}\Theta+X_{k}\partial_{y_{1}}^{3}\Theta\right).
	\end{aligned}
\end{equation*}

Combining the above identities, we decompose
\begin{equation*}
\begin{aligned}
\mathcal{E}(W)&=-\left(\frac{\lambda_s}{\lambda }+\Gamma\right)(\Lambda Q+\Psi_\lambda)+\Psi_{W}+\sum_{k=2}^4\lambda^{k\theta}\left(F_k
-\partial_{y_1}\mathcal{L}A_k\right)\Theta.
\end{aligned}
\end{equation*}
Here, we denote
\begin{equation*}
	F_{2}=c_{1}\left(\Lambda X_{1}-\theta X_{1}\right)+3\partial_{y_{1}}(QN_{2}),\quad \quad \quad \quad \quad \quad \quad \quad 
	\end{equation*}
\begin{equation*}
\begin{aligned}
	F_{3}&=\sum_{\ell=1,2}c_{3-\ell}\left(\Lambda X_{\ell}-\ell\theta X_{\ell}\right)+3\partial_{y_{1}}(QN_{3})+\partial_{y_{1}}M_{3},\\
	F_{4}&=\sum_{\ell=1,2,3}c_{4-\ell}\left(\Lambda X_{\ell}-\ell\theta X_{\ell}\right)+3\partial_{y_{1}}(QN_{4})+\partial_{y_{1}}M_{4}.
	\end{aligned}
\end{equation*}

\textbf{Step 2.} Choice of $(c_{2},c_{3})$ and construction of $(A_{2},A_{3},A_{4})$.
We claim that, there exists smooth function $A_{2}\in \mathcal{Z}_{1}(\R^{2})$ such that $\partial_{y_{1}}\mathcal{L}A_{2}=F_{2}$ and $A_{2}$ is even in $y_{2}$.

\smallskip
Indeed, using the definition of $F_{2}$ and $X_{1}$, we directly have 
\begin{equation*}
	F_{2}=\theta^{-2}\left(\Lambda P
	+3\partial_{y_{1}}\left(QP^{2}\right)
	-\theta P\right)\Longrightarrow 
	F_{2}\in \mathcal{Z}_{0}(\R^{2}) \ \ \mbox{and   $F_{2}$ is even in $y_{2}$}.
\end{equation*}
Then, using Remark~\ref{re:deftheta}, we find
\begin{equation*}
	(F_{2},Q)=\theta^{-2}\left(\Lambda P
	+3\partial_{y_{1}}\left(QP^{2}\right)
	,Q\right)-\theta^{-1}(P,Q)=0.
\end{equation*}
Therefore, from Lemma~\ref{le:L1}, we complete the proof of the existence for $A_{2}\in \mathcal{Z}_{1}(\R^{2})$.

\smallskip
We claim that, there exist smooth functions $(A_{3},A_{4})\in \mathcal{Z}_{2}(\R^{2})\times \mathcal{Z}_{3}(\R^{2})$ and real constants $(c_{2},c_{3})\in \R^{2}$ such that $\partial_{y_{1}}\mathcal{L}A_{k}=F_{k}$ and $A_{k}$ is even in $y_{2}$ for $k=3,4$.

\smallskip
Indeed, using the definition of $F_{3}$, we decompose $F_{3}=-\theta^{-1}c_{2}F_{3,1}-\theta^{-1}F_{3,2}$ where
\begin{equation*}
\begin{aligned}
	&F_{3,1}=2\Lambda P-3\theta P+6\partial_{y_{1}}(QP^2),\\
	&F_{3,2}=\Lambda A_2-2\theta A_2+6\partial_{y_{1}}(QPA_2)+\theta^{-2}\partial_{y_{1}}(P^3).
\end{aligned}
\end{equation*}
It follows that $F_{3}\in \mathcal{Z}_{1}(\R^{2})$ and $F_{3}$ is even in $y_{2}$. Then, using again Remark~\ref{re:deftheta},
\begin{equation*}
	\left(F_{3,1},Q\right)=
	2\left(\Lambda P+3\partial_{y_{1}}(QP^{2}),Q\right)-3\theta(P,Q)=-\theta (P,Q)<0.
\end{equation*}
Based on the above identity, there exists $c_{2}\in \R$ such that $(F_{3},Q)=0$, and thus, using again Lemma~\ref{le:L1}, we complete the proof of the existence for $A_{3}\in \mathcal{Z}_{2}(\R^{2})$.

\smallskip
Using the definition of $F_{4}$, we decompose $F_{4}=-\theta^{-1}c_{3}F_{4,1}-\theta^{-1}F_{4,2}+F_{4,3}$ where 
\begin{equation*}
	\begin{aligned}
		F_{4,1}&=2\Lambda P-4\theta P+6\partial_{y_{1}}(QP^{2}),\\
		F_{4,2}&=\Lambda A_{3}-3\theta A_{3}+6\partial_{y_{1}}(QPA_{3}),\\
		F_{4,3}&=c_{2}\left(\Lambda X_{2}-2\theta X_{2}\right)
		+3\partial_{y_{1}}(QX_{2}^{2}+X_{1}^{2}X_{2}).
	\end{aligned}
\end{equation*}
It follows that $F_{4}\in \mathcal{Z}_{2}(\R^{2})$ and $F_{4}$ is even in $y_{2}$. Then, using again Remark~\ref{re:deftheta},
\begin{equation*}
	\left(F_{4,1},Q\right)=
	2\left(\Lambda P+3\partial_{y_{1}}(QP^{2}),Q\right)-4\theta(P,Q)=-2\theta (P,Q)<0.
\end{equation*}
Based on the above identity, there exists $c_{3}\in \R$ such that $(F_{4},Q)=0$, and thus, using again Lemma~\ref{le:L1}, we complete the proof of the existence for $A_{4}\in \mathcal{Z}_{3}(\R^{2})$.

\smallskip
From the choice of $(c_{2},c_{3})$ and construction of $(A_{2},A_{3},A_{4})$, we see that 
\begin{equation*}
	\sum_{k=2}^4\lambda^{k\theta}\left(F_k
	-\partial_{y_1}\mathcal{L}A_k\right)\Theta=0
	\Longrightarrow
\mathcal{E}(W)=-\left(\frac{\lambda_s}{\lambda }+\Gamma\right)(\Lambda Q+\Psi_\lambda)+\Psi_{W}.
\end{equation*}

\smallskip
\textbf{Step 3.} Conclusion. We now establish the estimates for $\Psi_{\lambda}$ and $\Psi_{W}$. First, by an elementary computation, for any $\beta=(\beta_{1},\beta_{2})\in \mathbb{N}^{2}$ with $\beta_{1}\in \left\{0,1,2,3\right\}$ and $y\in \mathbb{R}^{2}$, we directly have 
\begin{equation*}
	\begin{aligned}
	\partial_{y}^{\beta}\Psi_{\lambda}
	&=
	\sum_{k=1}^{4}\lambda^{k\theta}
	\sum_{\beta_{11}+\beta_{12}=\beta_{1}}
	\left(\partial_{y_{2}}^{\beta_{2}}\partial_{y_{1}}^{\beta_{11}}
	\left(\Lambda X_{k}-k\theta X_{k}\right)
	\right)
	\left(\partial_{y_{1}}^{\beta_{12}}\Theta\right)\\
	&-\frac{3}{5}\sum_{k=1}^{4}\lambda^{k\theta}
	\sum_{\beta_{11}+\beta_{12}=\beta_{1}}
	\left(\partial_{y_{2}}^{\beta_{2}}\partial_{y_{1}}^{\beta_{11}}
	X_{k}
	\right)
	\left(\partial_{y_{1}}^{\beta_{12}}\left(y_{1}\partial_{y_{1}}\Theta\right)\right).
	\end{aligned}
\end{equation*}
Therefore, using again Remark~\ref{re:decayZk} and a similar argument as in the proof for (ii) of Lemma~\ref{le:estW1}, we complete the proof for the estimate of $\Psi_{\lambda}$. 

\smallskip
Second, for any $\beta=(\beta_{1},\beta_{2})\in \mathbb{N}^{2}$ with $\beta_{1}\in \left\{0,1,2,3\right\}$ and $y\in \mathbb{R}^{2}$, we compute
\begin{equation*}
	\begin{aligned}
		\partial_{y}^{\beta}\Psi_{1W}&=-\frac{3\Gamma}{5}
		\sum_{k=1}^{4}\lambda^{k\theta}\sum_{\beta_{11}+\beta_{12}=\beta_{1}}
		\left(\partial_{y_{2}}^{\beta_{2}}\partial_{y_{1}}^{\beta_{11}}X_{k}\right)
	\left(	\partial_{y_{1}}^{\beta_{12}}
			\left(y_{1}\partial_{y_{1}}\Theta\right)\right)\\
		&+\sum_{k=5}^{7}\lambda^{k\theta}
	\sum_{\ell=k-3}^{4}c_{k-\ell}	\sum_{\beta_{11}+\beta_{12}=\beta_{1}}
		\left(\partial_{y_{2}}^{\beta_{2}}\partial_{y_{1}}^{\beta_{11}}\left(\Lambda X_{\ell}-\ell\theta X_{\ell}\right)\right)
		\left(\partial_{y_{1}}^{\beta_{12}}\Theta\right).
	\end{aligned}
\end{equation*}
Therefore, using again Remark~\ref{re:decayZk} and a similar argument as in the proof for (ii) of Lemma~\ref{le:estW1}, we obtain
\begin{equation*}
	\begin{aligned}
&\bigg|	\Gamma
	\sum_{k=1}^{4}\lambda^{k\theta}\sum_{\beta_{11}+\beta_{12}=\beta_{1}}
	\left(\partial_{y_{2}}^{\beta_{2}}\partial_{y_{1}}^{\beta_{11}}X_{k}\right)
	\left(	\partial_{y_{1}}^{\beta_{12}}
	\left(y_{1}\partial_{y_{1}}\Theta\right)\right)\bigg|\\
	&\lesssim
	\lambda^{2\theta+\frac{8}{5}\beta_{1}}e^{-\frac{|y_{2}|}{4}}\textbf{1}_{[-2,-1]}(\lambda^{\frac{8}{5}}y_{1})\lesssim 
	\lambda^{\theta+\frac{8}{5}(1+\beta_{1})}e^{-\frac{|y_{2}|}{4}}\textbf{1}_{[-2,-1]}(\lambda^{\frac{8}{5}}y_{1}).
	\end{aligned}
\end{equation*}
Similarly, from~\eqref{equ:defA} and Remark~\ref{re:decayZk}, we obtain
\begin{equation*}
\begin{aligned}
&\bigg|
\sum_{k=5}^{7}\lambda^{k\theta}
\sum_{\ell=k-3}^{4}c_{k-\ell}	\sum_{\beta_{11}+\beta_{12}=\beta_{1}}
\left(\partial_{y_{2}}^{\beta_{2}}\partial_{y_{1}}^{\beta_{11}}\left(\Lambda X_{\ell}-\ell\theta X_{\ell}\right)\right)
\left(\partial_{y_{1}}^{\beta_{12}}\Theta\right)
\bigg|\\
&\lesssim \lambda^{5\theta}e^{-\frac{|y|}{4}}+\lambda^{5\theta}
\left(1+
y_{1}^{3-\beta_{11}}\right)\lambda^{\frac{8}{5}\beta_{12}}\textbf{1}_{[-2,0]}(\lambda^{\frac{8}{5}}y_{1})\\
&\lesssim
\lambda^{5\theta}e^{-\frac{|y|}{4}}+\lambda^{5\theta}
\left(1+
y_{1}^{3-\beta_{1}}\right)\textbf{1}_{[-2,0]}(\lambda^{\frac{8}{5}}y_{1}).
\end{aligned}
\end{equation*}
Combining the above estimates, we deduce that 
	\begin{equation*}
	\begin{aligned}
		|\partial_{y}^{\beta}\Psi_{1W}|
		&\lesssim 	\lambda^{5\theta}e^{-\frac{|y|}{4}}
		+\lambda^{\theta+\frac{8}{5}(1+\beta_{1})}
		e^{-\frac{|y_2|}{4}}
		\mathbf{1}_{[-2,-1]}(\lambda^{\frac{8}{5}} y_1)\\
		&+\lambda^{5\theta}\left( 1+y_{1} ^{3-\beta_{1}}\right)
		e^{-\frac{|y_2|}{4}}
		\mathbf{1}_{[-2,0]}(\lambda^{\frac{8}{5}} y_1).
	\end{aligned}
\end{equation*}

Then, for any $\beta=(\beta_{1},\beta_{2})\in \mathbb{N}^{2}$ with $\beta_{1}\in \left\{0,1,2,3\right\}$ and $y\in \mathbb{R}^{2}$, we compute
\begin{equation*}
	\begin{aligned}
	\partial_{y}^{\beta}\Psi_{2W}
	&=\sum_{k=1}^{4}\lambda^{k\theta}
	\sum_{\beta_{11}+\beta_{12}=\beta_{1}}
	\left(\partial_{y_{2}}^{\beta_{2}}\partial_{y_{1}}^{\beta_{11}}X_{k}\right)\left(\partial_{y_{1}}^{3+\beta_{12}}\Theta\right)\\
	&+3\sum_{k=1}^{4}\lambda^{k\theta}
	\sum_{\beta_{11}+\beta_{12}=\beta_{1}}
	\left(\partial_{y_{2}}^{\beta_{2}}\partial_{y_{1}}^{1+\beta_{11}}X_{k}\right)\left(\partial_{y_{1}}^{2+\beta_{12}}\Theta\right)\\
	&-\sum_{k=1}^{4}
	\lambda^{k\theta}
	\sum_{\beta_{11}+\beta_{12}=\beta_{1}}
	\left(
	\partial_{y_{2}}^{\beta_{2}}\partial_{y_{1}}^{\beta_{11}}
	(\mathcal{L}X_{k}-2\partial_{y_{1}}^{2}X_{k})
	\right)\left(\partial_{y_{1}}^{1+\beta_{12}}\Theta\right).
	\end{aligned}
\end{equation*}
Therefore, using again Remark~\ref{re:decayZk} and a similar argument as in the proof for (ii) of Lemma~\ref{le:estW1}, we obtain
\begin{equation*}
	\left|\partial_{y}^{\beta}\Psi_{2W}\right|
	\lesssim \lambda^{\theta+\frac{8}{5}(1+\beta_{1})}e^{-\frac{|y_{2}|}{4}}
	\textbf{1}_{[-2,-1]}(\lambda^{\frac{8}{5}}y_{1}).
\end{equation*}
Here, we use the fact that 
\begin{equation*}
	\left|\partial_{y_{1}}^{1+\beta_{12}}\Theta\right|+	\left|\partial_{y_{1}}^{2+\beta_{12}}\Theta\right|+	\left|\partial_{y_{1}}^{3+\beta_{12}}\Theta\right|
	\lesssim \lambda^{\frac{8}{5}(1+\beta_{12})}\textbf{1}_{[-2,-1]}(\lambda^{\frac{8}{5}}y_{1}).
\end{equation*}
Based on a similar argument to the one in estimates of $\Psi_{1W}$ and $\Psi_{2W}$,
for any $\beta=(\beta_{1},\beta_{2})\in \mathbb{N}^{2}$ with $\beta_{1}\in \left\{0,1,2,3\right\}$ and $y\in \R^{2}$, 
 we deduce that 
\begin{equation*}
	\begin{aligned}
	\left|\partial_{y}^{\beta}\Psi_{3W}\right|
	+\left|\partial_{y}^{\beta}\Psi_{4W}\right|
	&\lesssim
\lambda^{5\theta}e^{-\frac{|y|}{4}}
+\lambda^{\theta+\frac{8}{5}(1+\beta_{1})}
e^{-\frac{|y_2|}{4}}
\mathbf{1}_{[-2,-1]}(\lambda^{\frac{8}{5}} y_1)\\
&+\lambda^{5\theta}\left( 1+y_{1} ^{3-\beta_{1}}\right)
e^{-\frac{|y_2|}{4}}
\mathbf{1}_{[-2,0]}(\lambda^{\frac{8}{5}} y_1).
\end{aligned}
\end{equation*}
Combining the above estimates with the definition of $\Psi_{W}$, we complete the proof for the estimates of $\Psi_{W}$, and thus, the proof of Proposition~\ref{prop:W} is complete.
\end{proof}

\subsection{Refined blow-up profile}\label{SS:Refined}
Recall that, in the previous works~\cite{CLY,CLYMINI}, we consider a refined blow-up profile $Q_{b}$, involving a small geometric parameter $b$ and a non-localized profile $P$. This additional term enables us to derive an extra orthogonality condition related to $Q$. We consider the following similar refinement for the blow-up profile $W$.
For any $|b|\ll 1$, we define the following refined blow-up profile:
\begin{equation*}
	W_{b(s),\lambda(s)}(y)=W(\lambda(s),y)+b(s)P_{b(s)}(y),
\end{equation*}
where 
\begin{equation*}
	\chi_{b}(y_{1})=\chi(|b|^{\frac{3}{4}}y_{1})\quad \mbox{and}\quad 
	P_{b}(y)=\chi_{b}(y_{1})P(y).
\end{equation*}
Let 
\begin{equation*}
	\Psi_{b}=\partial_{y_{1}}
	\left(b\Delta P_{b}-bP_{b}+W_{b,\lambda}^{3}-W^{3}\right)
	+b\Lambda Q+b\Gamma \Lambda P_{b}.
\end{equation*}
We now introduce the estimates related to the refined blow-up profile $W_{b,\lambda}$.
\begin{proposition}\label{prop:Wb}
There exists a small constant $0<b^{*}\ll 1$ such that for any $|b|<b^{*}$, the following estimates hold.

\begin{enumerate}
	\item \emph{Closeness to $W$.} It holds 
	\begin{equation*}
    \left|\int_{\R^{2}}W_{b,\lambda}^{2}\dd y-\int_{\R^{2}}W^{2}\dd y\right|\lesssim |b|+|b|^{\frac{1}{4}}\lambda^{\theta}\quad \mbox{and}\quad 
     \|W_{b,\lambda}-W\|_{H^1}\lesssim |b|^{\frac{5}{8}}.
	\end{equation*}
	
	\item \emph{Estimates of $W_{b,\lambda}$.} For any $\beta=(\beta_{1},\beta_{2})\in \mathbb{N}^{2}$, we have 
	\begin{equation*}
		\begin{aligned}
	\left|	\partial_{y}^{\beta}W_{b,\lambda}\right|+
		\left|	\partial_{y}^{\beta}\Lambda W_{b,\lambda}\right|
		&\lesssim 
		e^{-\frac{|y|}{4}}+
		\lambda^{\theta+\frac{8}{5}\beta_{1}}
		e^{-\frac{|y_2|}{4}}
		\mathbf{1}_{[-2,0]}(\lambda^{\frac{8}{5}} y_1)\\
		&+
		|b|^{1+\frac{3}{4}\beta_{1}}e^{-\frac{|y_{2}|}{4}}
		\mathbf{1}_{[-2,0]}(|b|^{\frac{3}{4}}y_{1}).
		\end{aligned}
	\end{equation*}
	
	\item \emph{Estimates for the error term.} For any $\beta=(\beta_{1},\beta_{2})\in \mathbb{N}^{2}$, we have
	\begin{equation*}
		\begin{aligned}
		\left|\partial_{y}^{\beta}\Psi_{b}\right|
		&\lesssim
		|b|(|b|+\lambda^{\theta})e^{-\frac{|y|}{4}}+|b|\lambda^{\theta}
		e^{-\frac{|y_2|}{4}}
		\mathbf{1}_{[-2,0]}(|b|^{\frac{3}{4}} y_1)\\
		&+|b|^{1+\frac{3}{4}(1+\beta_{1})}
		e^{-\frac{|y_2|}{4}}
		\mathbf{1}_{[-2,-1]}(|b|^{\frac{3}{4}} y_1).
		\end{aligned}
	\end{equation*}
	
	\item \emph{Scalar product with $Q$.} It holds
	\begin{equation*}
		\left|(\Psi_b,Q)+\sigma b\lambda^\theta (P,Q)\right|\lesssim
		|b|\left(|b|+\lambda^{2\theta}\right).
	\end{equation*}
	Here, the constant $\sigma$ is defined by 
	\begin{equation*}
		\sigma=\frac{\left(\Lambda P+6\partial_{y_{1}}(QP^{2}),Q\right)}
		{\left(\Lambda P+3\partial_{y_{1}}(QP^{2}),Q\right)}.
	\end{equation*}
\end{enumerate}
\end{proposition}

\begin{proof}
	Proof of (i). Expanding $W_{b,\lambda}=W+bP_{b}$, we find
	\begin{equation*}
		\int_{\R^{2}}W_{b,\lambda}^{2}\dd y=\int_{\R^{2}}W^{2}\dd y+2b\int_{\R^{2}}WP_{b}\dd y+b^{2}\int_{\R^{2}}P_{b}^{2}\dd y.
	\end{equation*}
	Using Lemma~\ref{le:Nonloca} and (ii) of Lemma~\ref{le:estW1}, we obtain 
	\begin{equation*}
	\left|b\int_{\R^{2}}WP_{b}\dd y\right|\lesssim |b|+|b|^{\frac{1}{4}}\lambda^{\theta}\quad \mbox{and}\quad 
b^{2}\int_{\R^{2}}P_{b}^{2}\dd y\lesssim |b|^{\frac{5}{4}}.
	\end{equation*}
	Using a similar argument as above, we also obtain 
    \begin{equation*}
        \|W_{b,\lambda}-W\|_{H^{1}}\lesssim |b|\|P_{b}\|_{H^{1}}\lesssim |b|^{\frac{5}{8}}.
    \end{equation*}
    Combining the above estimates, we complete the proof of (i).
	
	\smallskip
	Proof of (ii). Note that, these estimates follow directly from Lemma~\ref{le:Nonloca}, (ii) of Lemma~\ref{le:estW1} and the definition of $P_{b}$ and $W_{b,\lambda}$.
	
	\smallskip
	Proof of (iii). By an elementary computation and $\partial_{y_{1}}\mathcal{L}P=\Lambda Q$, we find
	\begin{equation}\label{equ:Psib}
		\begin{aligned}
		\Psi_{b}
		&=b\Gamma \Lambda P_{b}
		-b
		\left(\chi'_{b}\left(\mathcal{L}P-2\partial_{y_{1}}^{2}P\right)
		-3\chi''_{b}\partial_{y_{1}}P-\chi'''_{b}P
		\right)\\
		&+b(1-\chi_{b})\Lambda Q
		+b\partial_{y_{1}}\left(3(W^{2}-Q^{2})P_{b}+3bWP_{b}^{2}
		+b^{2}P_{b}^{3}\right).
		\end{aligned}
	\end{equation}
	Let $\beta=(\beta_{1},\beta_{2})\in \mathbb{N}^{2}$.
	From the decay property of $Q$ and $P$, we have 
	\begin{equation*}
		\begin{aligned}
				\left|
			b\Gamma\partial_{y}^{\beta}\Lambda P_{b}
			\right|&\lesssim 
			|b|\lambda^{\theta}
			\left(e^{-\frac{|y|}{4}}+
			e^{-\frac{|y_{2}|}{4}}
			\mathbf{1}_{[-2,0]}(|b|^{\frac{3}{4}}y_{1})\right),\\
				\left|b\partial_{y}^{\beta}
		\left((1-\chi_{b})\Lambda Q\right)
		\right|
		&\lesssim |b|
		e^{-\frac{|y|}{2}}
		\mathbf{1}_{[-2,-1]}(|b|^{\frac{3}{4}}y_{1})\lesssim b^{2}e^{-\frac{|y|}{4}}.
		\end{aligned}
	\end{equation*}
	Then, using again the decay property of $P$ and the definition of $\chi_{b}$, 
	\begin{equation*}
		\begin{aligned}
		\left|b\partial_{y}^{\beta}
		\left(\chi'_{b}
		\left(\mathcal{L}P-2\partial_{y_{1}}^{2}P\right)\right)\right|
&	\lesssim 	|b|^{1+\frac{3}{4}(1+\beta_{1})}
		e^{-\frac{|y_2|}{4}}
		\mathbf{1}_{[-2,-1]}(|b|^{\frac{3}{4}} y_1),\\
		\left|b\partial_{y}^{\beta}
		\left(
		\chi''_{b}\partial_{y_{1}}P\right)\right|+\left|b\partial_{y}^{\beta}
		\left(\chi'''_{b}P\right)\right|
		&\lesssim |b|^{1+\frac{3}{4}(1+\beta_{1})}
		e^{-\frac{|y_2|}{4}}
		\mathbf{1}_{[-2,-1]}(|b|^{\frac{3}{4}} y_1).
		\end{aligned}
	\end{equation*}
	Recall that, in Section~\ref{SS:Blowup}, we define
	\begin{equation*}
		W=Q+V\quad \mbox{and}\quad W^{2}-Q^{2}=2QV+V^{2}.
	\end{equation*}
	It follows from $\theta>\frac{8}{5}$ and Remark~\ref{re:decayZk} that
	\begin{equation*}
			\left|
			b\partial_{y_{2}}^{\beta_{2}}\partial_{y_{1}}^{1+\beta_{1}}
			(W^{2}-Q^{2})P_{b}
			\right|\lesssim |b|\lambda^{\theta}
			\left(
			e^{-\frac{|y|}{4}}
			+e^{-\frac{|y_{2}|}{4}}\mathbf{1}_{[-2,0]}(|b|^{\frac{3}{4}}y_{1})
			\right).
	\end{equation*}
	Last, using a similar argument as above, we find
	\begin{equation*}
		\begin{aligned}
			\left|b^{3}\partial_{y_{2}}^{\beta_{2}}\partial_{y_{1}}^{
			1+\beta_{1}}P_{b}^{3}
		\right|&\lesssim
		|b|^{3}
		\left(e^{-\frac{|y|}{4}}
		+|b|^{\frac{3}{4}(1+\beta_{1})}e^{-\frac{|y_{2}|}{4}}\mathbf{1}_{[-2,-1]}(|b|^{\frac{3}{4}}y_{1})
		\right),\\
		\left|b^{2}\partial_{y_{2}}^{\beta_{2}}\partial_{y_{1}}^{1+\beta_{1}}\left(WP_{b}^{2}\right)\right|&\lesssim 
		b^{2}e^{-\frac{|y|}{4}}+
		b^{2}\lambda^{\theta}
		\left(
		e^{-\frac{|y|}{4}}
		+e^{-\frac{|y_{2}|}{4}}\mathbf{1}_{[-2,0]}(|b|^{\frac{3}{4}}y_{1})
		\right).
	\end{aligned}
	\end{equation*}
	Combining the above estimates with~\eqref{equ:Psib}, we complete the proof of (iii).
	
	\smallskip
	Proof of (iv). First, using again the definition of $\chi_{b}$ and the decay property of $Q$, 
	\begin{equation*}
		\begin{aligned}
		\left|b
		\left(\chi'_{b}\left(\mathcal{L}P-2\partial_{y_{1}}^{2}P\right)
		-3\chi''_{b}\partial_{y_{1}}P-\chi'''_{b}P,Q
		\right)\right|
			&\lesssim b^{2},\\
				\left|b\left((1-\chi_{b})\Lambda Q,Q\right)\right|+	\left|b^{2}
			\left(\partial_{y_{1}}\left(WP_{b}^{2}\right),Q\right)\right|
			+\left|b^{3}\left(\partial_{y_{1}}P_{b}^{3},Q\right)\right|
			&\lesssim b^{2}.
		\end{aligned}
	\end{equation*}
	Recall that, in Section~\ref{SS:Blowup}, we define
	\begin{equation*}
		\Gamma(s)=c_{1}\lambda^{\theta}(s)+c_{2}\lambda^{2\theta}(s)+c_{3}\lambda^{3\theta}(s)\quad \mbox{and}\quad c_{1}=-\theta^{-1}.
	\end{equation*}
	It follows from the decay property of $Q$ that 
	\begin{equation*}
		\left|b\Gamma \left(\Lambda P_{b},Q\right)+\theta^{-1}b\lambda^{\theta}\left(\Lambda P,Q\right)\right|\lesssim |b|\left(|b|+\lambda^{2\theta}\right).
	\end{equation*}
	Recall also that, in Section~\ref{SS:Blowup}, we define
		\begin{equation*}
			V(s,y)=
			\sum_{k=1}^4\lambda^{k\theta}(s)X_k(y)\Theta(s,y_{1})
			\quad \mbox{and}\quad 
			X_{1}(y)=-\theta^{-1}P(y).
		\end{equation*}
		Therefore, using again the decay property of $Q$ again, we find 
		\begin{equation*}
			\left|b\left(
			3\partial_{y_{1}}\left(W^{2}-Q^{2}\right)P_{b},Q
			\right)
			+\theta^{-1}b\lambda^{\theta}
			\left(6\partial_{y_{1}}\left(QP^{2}\right),Q\right)
			\right|\lesssim |b|\left(|b|+\lambda^{2\theta}\right).
		\end{equation*}
		Combining the above estimates with Remark~\ref{re:deftheta}, we complete the proof of (iv).
		\end{proof}

\section{Modulation estimates}\label{S:Modula}
\subsection{Geometric decomposition}\label{SS:Geo}
In this subsection, we introduce the geometric decomposition of solutions that are even in $x_{2}$ and close to the refined blow-up profile. 
We recall the following result for solutions of~\eqref{CP} which are even in $x_2$.
\begin{proposition}\label{Prop:decomposition}
Let $0<T<T_0<1$ and $\frac{7}{9}\le s\le 1$. Assume that there exist $(\overline{\lambda}(t),\overline{x}_1(t),\overline{\varepsilon}(t))\in(0,\infty)\times\mathbb{R}\times H^{s}\left(\R^{2}\right)$ such that, for all $t\in[T,T_0]$, the solution $u(t)$ of \eqref{CP} satisfies
\begin{equation}\label{est:decomposition}
u(t,x)=\frac{1}{\overline{\lambda}(t)}\left[{Q}+\bar{\varepsilon}(t)\right]\left(\frac{x-(\bar{x}_1(t),0)}{\bar{\lambda}(t)}\right),\quad \mbox{with}\quad \|\bar{\varepsilon}(t)\|_{L^2}\le \kappa\le \kappa_{1},
\end{equation}
where $0<\kappa_{1}\ll1$ is small enough universal constant. Then there exist unique $C^1$ functions $(\lambda(t),x_1(t),b(t))\in(0,\infty)\times\mathbb{R}^2$ such that, for all $t\in [T,T_{0}]$, $\varepsilon(t)$ being defined by
\begin{equation}\label{equ:decom}
\varepsilon(t,y)=\lambda(t)u\left(t,\lambda(t)y+(x_1(t),0)\right)-W_{b(t),\lambda(t)}(y),
\end{equation}
it satisfies the orthogonality conditions
\begin{equation}\label{equ:ortho}
(\varepsilon(t),Q)=(\varepsilon(t), Q^3)=(\varepsilon(t),\partial_{y_1} Q)=0.
\end{equation}
In addition, we have
\begin{equation*}
\|\varepsilon(t)\|_{L^2}+|b(t)|+\bigg|1-\frac{\overline{\lambda}(t)}{\lambda(t)}\bigg|\lesssim\varrho(\kappa)\quad \mbox{and}\quad 
\|\varepsilon(t)\|_{H^{s}}\lesssim \varrho
\left(\|{\varepsilon}(0)\|_{H^{s}}\right),
\end{equation*}
where $\lim_{\alpha\rightarrow 0^+}\rho(\alpha)=0$.
\end{proposition}

\begin{proof}
Recall that,
the proof of the above decomposition proposition relies on a standard argument based on Proposition~\ref{prop:Spectral} and the implicit function Theorem. See more details in the proof of~\cite[Proposition 3.1]{CLY} for the case where the blow-up profile is $Q_{b}$. We mention here that, the same proof applies to the current case where the blow-up profile is $W_{b,\lambda}$ instead of $Q_{b}$.
\end{proof}

We now introduce the following change of variables:
\begin{equation*}
    y=\frac{x-(x_1(t),0)}{\lambda(t)}\quad 
    \mbox{and}\quad 
    s=S+\int_{T}^{t}\frac{\dd \tau}{\lambda^3(\tau)}\in I.
\end{equation*}
Here, we denote 
\begin{equation*}
    I=[S,s_0], \quad S=-\left(\frac{\theta}{(3-\theta) T}\right)^{\frac{\theta}{3-\theta}}\quad 
    \mbox{and}\quad 
    s_0=S+\int_{T}^{T_0}\frac{\dd \tau}{\lambda^3(\tau)}.
\end{equation*}
By choosing $0<T<T_0<1$ to be small enough, we have 
\begin{equation*}
    -\infty<S<s_{0}<0 \quad \mbox{with}\quad  |s_{0}|\gg 1.
\end{equation*}

\smallskip
Note that, from~\eqref{equ:rescal},~\eqref{equ:decom}, the definition of $\mathcal{E}(W)$ in Section~\ref{SS:Blowup} and the definition of $\Psi_{b}$ in Section~\ref{SS:Refined}, we directly have 
\begin{equation}\label{equ:e}
\begin{aligned}
	&\partial_{s}\varepsilon+\partial_{y_{1}}\left(\Delta \varepsilon-\varepsilon+
	\left(\left(W_{b,\lambda}+\varepsilon\right)^{3}-W_{b,\lambda}^{3}\right)
		\right)\\
	&=\frac{\lambda_{s}}{\lambda}\Lambda \varepsilon+\left(\frac{x_{1s}}{\lambda}-1\right)\partial_{y_{1}}\varepsilon-\mathcal{E}_{b}.
    \end{aligned}
\end{equation}
Here, we denote 
\begin{equation*}
\mathcal{E}_b=\mathcal{E}(W)-b\Lambda Q-\left(\frac{\lambda_{s}}{\lambda}+\Gamma\right)(b\Lambda P_b)-\left(\frac{x_{1s}}{\lambda}-1\right)\partial_{y_1}W_{b,\lambda}+\partial_{s}(bP_{b})+\Psi_b.
\end{equation*}
To simplify notation, we set 
\begin{equation*}
{\rm{Mod}}=\left(\frac{\lambda_s}{\lambda}+\Gamma+b\right)\Lambda Q+\left(\frac{x_{1s}}{\lambda}-1\right)\partial_{y_1}Q.
\end{equation*}
Note also that, using Proposition~\ref{prop:W}, we can rewrite 
\begin{align*}
\mathcal{E}_b=
-{\rm{Mod}}
-\Psi_M+\Psi_{N}+\Psi_W+\Psi_b.
\end{align*}
Here, we denote 
\begin{equation*}
	\begin{aligned}
		\Psi_{N}&=b(\Psi_\lambda+b\Lambda P_b)+b_s\left(\chi_{b}+\frac{3}{4}y_{1}\partial_{y_{1}}\chi_{b}\right)P,\\
		\Psi_{M}&=
		\left(\frac{\lambda_s}{\lambda}+\Gamma+b\right)
		\left(\Psi_\lambda+b\Lambda P_b\right)
	+	\left(\frac{x_{1s}}{\lambda}-1\right)
		\left(\partial_{y_1}V+b\partial_{y_1}P_b\right).
	\end{aligned}
\end{equation*}

To state the control of parameters $(\lambda,b,x_{1})$, we consider the weighted $H^{1}$ norm of $\varepsilon$ with specific weight norm. We follow the notation and presentation of~\cite[Section 3.1]{CLY}.
Consider the smooth even function $\zeta\in C^{\infty}(\R)$ with $\zeta\in (0,1]$ as follows,
\begin{equation*}
	\begin{aligned}
		\zeta(y_1)=
		\begin{cases}
			\exp({2y_{1}}),&\mbox{for}\ y_1\in(-\infty,-\frac{1}{6}),\\
			1,&\mbox{for}\ y_1\in(-\frac{1}{10},\frac{1}{10}),\\
			\exp({-2y_{1}}),&\mbox{for}\ y_1\in(\frac{1}{6},\infty),\\
		\end{cases}
		\quad \int_{\R}\zeta (y_{1})\dd y_{1}=1.
	\end{aligned}
\end{equation*}
In addition, we consider the smooth function $\vartheta\in C^{\infty}(\R)$ as follows,
\begin{equation*}
    \vartheta(y_{1})=\begin{cases}
			\frac{1}{2},&\mbox{for}\ y_1\in(-\infty,\frac{1}{2}),\\
			y_{1}^{7},&\mbox{for}\ y_1\in(1,\infty),\\
		\end{cases}
        \quad 
		\vartheta'(y_{1})\ge 0, \ \ \mbox{on}\ \R.
\end{equation*}
Let $B>1$ be a large enough universal constant to be chosen later. 
We consider two weight functions $\vartheta_{B}\in C^{\infty}(\R)$ and  $\varphi_{B}\in C^{\infty}(\R)$ with
\begin{equation*}
\vartheta_{B}(y_{1})=\vartheta\left(\frac{y_{1}}{B^{10}}\right)\quad \mbox{and}
\quad 
\varphi_{B}(y_{1})=\sqrt{2\psi_{B}(y_{1})}\vartheta_{B}(y_{1}),\quad \mbox{on}\ \ \R.
\end{equation*}
Here we set 
\begin{equation*}
	\lim_{y_1\rightarrow-\infty}\psi_B(y_1)=0 \ \ \mbox{and}\ \ 
	\psi_{B}'(y_1)=
	\begin{cases}
		\frac{1}{B}\zeta\left(\frac{y_{1}}{B}+\frac{1}{3}-\frac{1}{2}B^{-\frac{1}{3}}\right),&\text{ for }y_1<-\frac{1}{3}B,\\
		\frac{1}{B}\zeta\left(\frac{y_1}{B^{\frac{2}{3}}}+\frac{1}{3}B^{\frac{1}{3}}\right),&\text{ for }y_1\geq-\frac{1}{3}B.
	\end{cases}
\end{equation*}

\begin{lemma}\label{le:psiphi}
	For all large enough $B>1$, the following estimates hold.
	\begin{enumerate}
		\item We know that $\psi_B$ is strictly increasing and $\psi_B(y_1)\to\frac{1}{2}$ as $y_{1}\uparrow \infty$.
		
		\item For all $y_1\in (-\infty,-B)$, we have 
		\begin{equation*}
        \begin{aligned}
			\exp\left(\frac{y_{1}}{B}\right)&\le \sqrt{\psi_{B}(y_{1})}\le \sqrt{2}\exp\left(\frac{y_{1}}{B}\right),\\
             \exp\left(\frac{y_{1}}{B}\right)&\le \sqrt{2\varphi^{2}_{B}(y_{1})}\le \sqrt{2}\exp\left(\frac{y_{1}}{B}\right).
            \end{aligned}
		\end{equation*}
		
		\item For all $y_{1}\in \left(-\frac{B}{4},\frac{B}{4}\right)$, we have
		\begin{equation*}
			\begin{aligned}
			\psi_B'(y_1)+\left|\psi_{B}(y_{1})-\frac{1}{2}\right|
			&\lesssim \exp\left({-\frac{1}{6}B^{\frac{1}{3}}}\right),\\
		\varphi'_{B}(y_{1})+\left|\varphi_{B}(y_1)-\frac{1}{2}\right|
			&\lesssim \exp\left({-\frac{1}{6}B^{\frac{1}{3}}}\right).
			\end{aligned}
		\end{equation*}
		In addition, for all $y_1\in\mathbb{R}$, we have $\psi_B(y_1) \le \varphi_{B}(y_1)$.
	\end{enumerate}
\end{lemma}

\begin{proof}
The proof relies on some elementary computation based on the definition of $\psi_{B}$ and $\varphi_{B}$ (see \emph{e.g.}~\cite[Lemma 3.3]{CLY}), and we omit it.
\end{proof}

We now introduce the following technical lemma related to the pointwise estimates $\psi_{B}$ and $\varphi_{B}$ and their derivatives. 
The proof is based on the above Lemma and the definition of $\psi_{B}$ and $\varphi_{B}$.
We refer to~\cite[Lemma 3.4]{CLY} for the details of the proof.
\begin{lemma}\label{le:psiphi2}
	The following estimates hold.
\begin{enumerate}
    \item \emph{First-type estimates for the derivatives of $\psi_{B}$.}
    We have 
    \begin{equation*}
        B^{\frac{2}{3}}\left|\psi''_{B}\right|
        +B^{\frac{4}{3}}\left|\psi'''_{B}\right|\lesssim \psi'_{B},\quad \mbox{on}\ \R.
    \end{equation*}

    \item \emph{Second-type estimates for the derivatives of $\psi_{B}$.} 
    We have 
    \begin{equation*}
        \sqrt{B\psi'_{B}}\lesssim \psi_{B}+B\varphi'_{B}\lesssim \varphi_{B},\quad \mbox{on}\ \R.
    \end{equation*}

    \item \emph{Third-type estimates for the derivative of $\psi_{B}$.} 
    We have 
    \begin{equation*}
        |y_{1}|\psi'_{B}\lesssim \sqrt{\psi_{B}}\lesssim \psi_{B}+B\varphi'
        _{B},\quad \mbox{on}\ \R.
    \end{equation*}

    \item \emph{First-type estimates for the derivative of $\varphi_{B}$.}
    We have 
    \begin{equation*}
        \begin{aligned}
            \left|\varphi''_{B}\right|&\lesssim B^{-\frac{2}{3}}\varphi'_{B}+B^{-20}\psi_{B},\quad \mbox{on} \ \R,\\
             \left|\varphi'''_{B}\right|&\lesssim B^{-\frac{4}{3}}\varphi'_{B}+B^{-30}\psi_{B},\quad \mbox{on} \ \R,
        \end{aligned}
    \end{equation*}

    \item \emph{Second-type estimates for the derivative of $\varphi_{B}$.}
    We have 
    \begin{equation*}
        \varphi_{B}\lesssim \psi_{B}+B\varphi'_{B}+|y_{1}|\varphi'_{B}{\mathbf{1}}_{[B^{10},\infty)},\quad \mbox{on}\ \R.
    \end{equation*}
\end{enumerate}
\end{lemma}
From the definition of $W_{b,\lambda}$ in Section~\ref{SS:Refined}, we have the following pointwise estimate.
\begin{lemma}\label{le:psiphi3}
For any $\beta=(\beta_{1},\beta_{2})\in \mathbb{N}^{2}$ with $0\le \beta_{1}+\beta_{2}\le 2$, we have 
    \begin{equation*}
    \begin{aligned}
\left|\partial_{y}^{\beta}W_{b,\lambda}\right|\left(\psi'_{B}+\varphi'_{B}+\left|\psi_{B}-\varphi_{B}\right|\right) \lesssim
\left(B^{-30}+\lambda^{\theta}+|b|\right)
\left(\psi_{B}+
       B\varphi'_{B}
       \right).
        \end{aligned}
    \end{equation*}
\end{lemma}

\begin{proof}
    First, from Proposition~\ref{prop:Wb}, we directly have 
    \begin{equation*}
        \left|\partial_{y}^{\beta}W_{b,\lambda}\right|\lesssim e^{-\frac{|y|}{4}}+\lambda^{\theta}e^{-\frac{|y_{2}|}{4}}+|b|e^{-\frac{|y_{2}|}{4}}.
    \end{equation*}
    Second, from Lemma~\ref{le:psiphi}--\ref{le:psiphi2}, we also have 
    \begin{equation*}
    \begin{aligned}
        \psi'_{B}+\varphi'_{B}+\left|\psi_{B}-\varphi_{B}\right|
        &\lesssim B^{-30}\left(\psi_{B}+B\varphi'_{B}\right)\mathbf{1}_{\left[-\frac{B}{4},\frac{B}{4}\right]}\left(y_{1}\right)\\
        &+\left(\psi_{B}+B\varphi'_{B}\right)\mathbf{1}_{\left(-\infty,-\frac{B}{4}\right)}\left(y_{1}\right)
        +\left(\psi_{B}+B\varphi'_{B}\right)\mathbf{1}_{\left(\frac{B}{4},\infty\right)}\left(y_{1}\right).
        \end{aligned}
    \end{equation*}
    Combining the above two estimates, we complete the proof of Lemma~\ref{le:psiphi3}.
\end{proof}

We now consider the following weighted $H^{1}$ norm related to the remainder term $\varepsilon$:
\begin{equation*}
	\mathcal{N}_{B}(\varepsilon)=\left(\int_{\R^{2}}\left(|\nabla \varepsilon|^{2}\psi_{B}+\varepsilon^{2}\varphi_{B}\right)\dd y\right)^{\frac{1}{2}}.
\end{equation*}

Indeed, it is easy to check that 
\begin{equation}\label{est:locNb}
	\|\varepsilon\|_{L^{2}_{\rm{loc}}}\lesssim
	\left(\int_{\R^{2}}\varepsilon^{2}e^{-\frac{|y|}{B}}\dd y\right)^{\frac{1}{2}}
	\lesssim
	\left(\int_{\R^{2}}\varepsilon^{2}\varphi_{B}\dd y\right)^{\frac{1}{2}}
	\lesssim \mathcal{N}_{B}(\varepsilon).
\end{equation}

Let $0<\kappa\ll 1$ be a small enough universal constant. In what follows, we still assume the following scaling invariant bounds hold for all $s\in I$:
\begin{equation}\label{est:NBbound}
\lambda(s)+	|b(s)|+\mathcal{N}_{B}(s)+\|\varepsilon(s)\|_{L^{2}}<\kappa.
\end{equation}

We now deduce the estimates for parameters $(\lambda,b,x_{1})$ from the equation of $\varepsilon$.
\begin{lemma}\label{le:refined}
	For all $s\in I$, the following estimates hold.
	\begin{enumerate}
		
		\item \emph{Estimates for $(\lambda,x_{1})$.} It holds
		\begin{equation*}
		\left|\frac{\lambda_s}{\lambda}+\Gamma+b\right|+\left|\frac{x_{1s}}{\lambda}-1\right|\lesssim
		\lambda^{5\theta}+|b|\left(|b|+\lambda^{\theta}\right)
		+	\|\varepsilon\|_{L^{2}_{\rm{loc}}}.
		\end{equation*}
		
	\item \emph{Estimate for $b$.} It holds 
	\begin{equation*}
		\left|b_{s}- b\lambda^{\theta}\right|\lesssim 
		\lambda^{5\theta}
	+|b|\left(|b|+\lambda^{2\theta}\right)
	+	\|\varepsilon\|_{L^{2}_{\rm{loc}}}\left(	\|\varepsilon\|_{L^{2}_{\rm{loc}}}+
	|b|+\lambda^{\theta}
	\right).
	\end{equation*}
	\end{enumerate}
\end{lemma}

\begin{proof}
First, from the definition of $\mathcal{L}$ in~\eqref{def:L}, we rewrite 
\begin{equation}\label{equ:e2}
	\partial_{s}\varepsilon=\partial_{y_{1}}\mathcal{L}\varepsilon
	+\widetilde{{\rm{Mod}}}-\left(\Gamma+b\right)\Lambda \varepsilon-\partial_{y_{1}}R_{1}-\partial_{y_{1}}R_{2}+\Psi_{M}
	-\Psi_{N}-\Psi_{W}-\Psi_{b}.
\end{equation}
Here, we denote 
\begin{equation*}
	\begin{aligned}
		R_{1}&=\left(\left(W_{b,\lambda}+\varepsilon\right)^{3}-W_{b,\lambda}^{3}-3W_{b,\lambda}^{2}\varepsilon\right),\quad R_{2}=3(W_{b,\lambda}^{2}-Q	^{2})\varepsilon,\\
	\widetilde{{\rm{Mod}}}&=\left(\frac{\lambda_s}{\lambda}+\Gamma+b\right)\left(\Lambda Q+\Lambda \varepsilon\right)+\left(\frac{x_{1s}}{\lambda}-1\right)
	\left(\partial_{y_1}Q+\partial_{y_{1}}\varepsilon\right).
	\end{aligned}
\end{equation*}
Second, differentiating the orthogonality conditions $(\varepsilon,Q^{3})=(\varepsilon,\partial_{y_{1}}Q)=0$, and then using 
~\eqref{est:NBbound}--\eqref{equ:e2} and
Proposition~\ref{prop:W}--\ref{prop:Wb}, we find 
\begin{equation*}
	\begin{aligned}
	&\left(1+O\left(|b|+\lambda^{\theta}+\|\varepsilon\|_{L^{2}_{\rm{loc}}}\right)\right)
	\left(\left|\frac{\lambda_{s}}{\lambda}+\Gamma+b\right|+
	\left|\frac{x_{1s}}{\lambda}-1\right|\right)\\
	&\lesssim
	\lambda^{5\theta}+
	 |b|(|b|+\lambda^{\theta})+|b_{s}|
	 +\int_{\R^{2}}\left(|\varepsilon|+\varepsilon^{2}+|\varepsilon|^{3}\right)e^{-\frac{|y|}{2}}\dd y.
	\end{aligned}
\end{equation*}
Note that, from 2D Gagliardo-Nirenberg inequality and \eqref{est:NBbound},
\begin{equation*}
	\begin{aligned}
	&\int_{\R^{2}}\left(|\varepsilon|+\varepsilon^{2}+|\varepsilon|^{3}\right)e^{-\frac{|y|}{2}}\dd y\\
	&\lesssim 
	\big\|\varepsilon e^{-\frac{|y|}{4}}\big\|_{L^{2}}
	\left(
	\big\|e^{-\frac{|y|}{4}}\big\|_{L^{2}}
	+\big\|\varepsilon e^{-\frac{|y|}{8}}\big\|_{L^{4}}^{2}
	\right)\\
	&\lesssim \big\|\varepsilon e^{-\frac{|y|}{4}}\big\|_{L^{2}}
	\left(1+\big\|\varepsilon e^{-\frac{|y|}{8}}\big\|_{L^{2}}
	\big\|\nabla \big(\varepsilon e^{-\frac{|y|}{8}}\big)\big\|_{L^{2}}
	\right)\lesssim \|\varepsilon\|_{L^{2}_{\rm{loc}}}.
	\end{aligned}
\end{equation*}
Based on the above two estimates and~\eqref{est:NBbound}, we obtain 
\begin{equation}\label{est:Mod1}
		\left(\left|\frac{\lambda_{s}}{\lambda}+\Gamma+b\right|+
		\left|\frac{x_{1s}}{\lambda}-1\right|\right)\lesssim
		\lambda^{5\theta}+
		|b|(|b|+\lambda^{\theta})+|b_{s}|+\|\varepsilon\|_{L^{2}_{\rm{loc}}}.
\end{equation}
On the other hand, differentiating the orthogonality condition $(\varepsilon, Q)=0$, and then using $(\partial_{y_{1}}\mathcal{L}\varepsilon,Q)=-(\mathcal{L}\varepsilon,\partial_{y_{1}}Q)=0$ and \eqref{equ:e2}, we find 
\begin{equation*}
	\begin{aligned}
	\left|\left(\Psi_{N}+\Psi_{b},Q\right)\right|
	&\lesssim
\left|\left(\widetilde{{\rm{Mod}}},Q\right)\right|+\left|\Gamma+b\right|\left|\left(\Lambda \varepsilon,Q\right)\right|
+\left|(\Psi_{M},Q)\right|\\
&+\left|(R_{1},\partial_{y_{1}}Q)\right|
+\left|(R_{2},\partial_{y_{1}}Q)\right|
+\left|(\Psi_{W},Q)\right|.
	\end{aligned}
\end{equation*}

Note that, from Lemma~\ref{le:Nonloca} and Proposition~\ref{prop:W},
we find
\begin{equation*}
	\left(\Psi_{N},Q\right)=b_{s}(1+O(b^{2}))(P,Q)
	+b\lambda^{\theta}\left(\Lambda X_{1}-\theta X_{1},Q\right)
	+O\left(|b|\left(|b|+\lambda^{2\theta}\right)\right).
\end{equation*}
Using again Remark~\ref{re:deftheta}, the definition of $X_{1}$ in Section~\ref{SS:Blowup} and $c_{1}=-\theta^{-1}$, 
\begin{equation*}
	(\Lambda X_{1}-\theta X_{1},Q)=\theta^{-1}\left(\theta P-\Lambda P,Q\right)
	=\frac{\left(3\partial_{y_{1}}(QP^{2}),Q\right)\left(P,Q\right)}{\left(\Lambda P+3\partial_{y_{1}}(QP^{2}),Q\right)}.
\end{equation*}
It follows from (iv) of Proposition~\ref{prop:Wb} that 
\begin{equation*}
	\left(\Psi_{N}+\Psi_{b},Q\right)
	=\left(b_{s}(1+O(b^{2}))-b\lambda^{\theta}\right)(P,Q)+O\left(|b|\left(|b|+\lambda^{2\theta}\right)\right).
\end{equation*}
Then, using again 2D Gagliardo-Nirenberg inequality and \eqref{est:NBbound},
\begin{equation*}
\left|(R_{1},\partial_{y_{1}}Q)\right|+\left|(R_{2},\partial_{y_{1}}Q)\right|\lesssim 
\|\varepsilon\|_{L^{2}_{\rm{loc}}}
\left(\|\varepsilon\|_{L^{2}_{\rm{loc}}}
+|b|+\lambda^{\theta}\right).
\end{equation*}
Next, using again Proposition~\ref{prop:W} and the definition of $V$ in Section~\ref{SS:Blowup}, 
\begin{equation*}
	\left|(\Psi_{M},Q)\right|+
	\left|(\Psi_{W},Q)\right|\lesssim
	\lambda^{5\theta}
	+
	 \left(\left|\frac{\lambda_{s}}{\lambda}+\Gamma+b\right|
	 +\left|\frac{x_{1s}}{\lambda}-1\right|
	 \right)
	\left(|b|+\lambda^{\theta}\right).
\end{equation*}
Last, using the fact that $(\Lambda Q,Q)=(\partial_{y_{1}}Q,Q)=0$, we find
\begin{equation*}
		\left|\left(\widetilde{{\rm{Mod}}},Q\right)\right|
	+\left|\Gamma+b\right|\left|(\Lambda \varepsilon,Q)\right|
	\lesssim \left(\left|\frac{\lambda_{s}}{\lambda}+\Gamma+b\right|
	+\left|\Gamma+b\right|+\left|\frac{x_{1s}}{\lambda}-1\right|
	\right)\|\varepsilon\|_{L^{2}_{\rm{loc}}}.
\end{equation*}
Based on the above estimates, we obtain 
\begin{equation*}
	\begin{aligned}
	\left|b_{s}-b\lambda^{\theta}\right|
	&\lesssim 
\lambda^{5\theta}+
	\left(\left|\frac{\lambda_{s}}{\lambda}+\Gamma+b\right|
	+\left|\Gamma+b\right|+\left|\frac{x_{1s}}{\lambda}-1\right|
	\right)\|\varepsilon\|_{L^{2}_{\rm{loc}}}
	\\
	&+\left(\left|\frac{\lambda_{s}}{\lambda}+\Gamma+b\right|
	+\left|\frac{x_{1s}}{\lambda}-1\right|\right)
	\left(|b|+\lambda^{\theta}\right)
	+\|\varepsilon\|_{L^{2}_{\rm{loc}}}
	\left(\|\varepsilon\|_{L^{2}_{\rm{loc}}}
	+|b|+\lambda^{\theta}\right).
\end{aligned}
\end{equation*}
Combining the above estimate with~\eqref{est:Mod1} and $\|\varepsilon\|_{L^{2}}\ll 1$, we complete the proof.
\end{proof}

\subsection{Bootstrap assumption}\label{SS:Boot}
For any $n\in \mathbb{N}^{+}$ large enough, we set 
\begin{equation*}
	T_{n}=\frac{\theta}{(3-\theta)n^{\frac{3-\theta}{\theta}}}\quad \mbox{and}\quad S_{n}=-n.
\end{equation*}
Recall that, we define the rescaled time variable:
\begin{equation*}
	s=s(t)=S_{n}+\int_{T_{n}}^{t}\frac{\dd \tau}{\lambda_{n}^{3}(\tau)}.
\end{equation*}
Here, $\lambda_{n}(t)$ is a real-valued function that will be chosen later. From now on, any time-dependent function will be seen either as a function of $t$ or as a function of $s$.

We consider the solution $\phi_{n}$ of~\eqref{CP} with initial data 
\begin{equation*}
	\phi_{n}(T_{n})=\frac{1}{\lambda_{n}(T_{n})}W_{b_n,\lambda_n}
	\left(T_{n},\frac{x_{1}-x_{1n}(T_{n})}{\lambda_{n}(T_{n})},\frac{x_{2}}{\lambda_{n}(T_{n})}\right).
\end{equation*}
Here, we set 
\begin{equation}\label{equ:deflambdan}
	\lambda_{n}(T_{n})=G^{-1}(n),\quad x_{1n}(T_{n})=-\frac{\theta}{\theta-1}n^{1-\frac{1}{\theta}},
	\quad b_{n}(T_{n})=0,
\end{equation}
with 
\begin{equation*}
	G(\lambda)=\int_{\lambda}^{\lambda_{0}}
	\frac{\dd \tau}{\tau^{1+\theta}\left(\frac{1}{\theta}-c_2\tau^\theta-c_3\tau^{2\theta}\right)},\quad \mbox{for any}\ \lambda\in (0,\lambda_{0}].
\end{equation*}

\begin{remark}\label{re:G}
	The function $G(\lambda)$ is well-defined for any $\lambda\in (0,\lambda_{0}]$, where $\lambda_{0}$ is small enough. Moreover, this function is decreasing and one-to-one from $(0,\lambda_{0}]$ to $[0,\infty)$. From the Taylor expansion, we find 
	\begin{equation*}
		\frac{1}{
		\tau^{1+\theta}\left(\frac{1}{\theta}-c_2\tau^\theta-c_3\tau^{2\theta}\right)}=\theta \tau^{-1-\theta}+c_{2}\theta^{2}\tau^{-1}+O\left(\tau^{-1+\theta}\right).
	\end{equation*}
	From the definition of $G(\lambda)$, we deduce that 
	\begin{equation*}
		G(\lambda)=\lambda^{-\theta}-c_{2}\theta^{2}\log \lambda+O(1),\quad \mbox{as}\ \ \lambda\downarrow 0.
	\end{equation*}
	Based on the above identity and the Taylor expansion, for any $n\in \mathbb{N}^{+}$ large enough, 
	\begin{equation*}
		G^{-1}(n)=\left(\frac{1}{n+O(\log n)}\right)^{\frac{1}{\theta}}=n^{-\frac{1}{\theta}}
		+O\left(n^{-1-\frac{1}{\theta}}\log n \right).
	\end{equation*}
\end{remark}

\smallskip
For $t\ge T_{n}$, as long as $\phi_{n}(t)$ is well-defined in $H^{1}$ and 
remains close to the soliton $Q$ up to rescaling and translation, we decompose $\phi_{n}(t)$ as in Section~\ref{SS:Geo}
\begin{equation*}
	\phi_{n}(t,x)=\frac{1}{\lambda_{n}(t)}
	\left[
	W_{b_{n},\lambda_{n}}\left(t,\frac{x_{1}-x_{1n}(t)}{\lambda_{n}(t)},\frac{x_{2}}{\lambda_{n}(t)}\right)
	+\varepsilon_{n}\left(t,\frac{x_{1}-x_{1n}(t)}{\lambda_{n}(t)},\frac{x_{2}}{\lambda_{n}(t)}\right)
	\right],
\end{equation*}
where $\varepsilon_{n}$ satisfies~\eqref{equ:ortho}. At $t=T_{n}$, this decompose satisfies~\eqref{equ:deflambdan} and $\varepsilon_{n}(T_{n})=0$.

\smallskip
Let $C_{1}>1$ be a large constant to be chosen later.
We now introduce the following bootstrap estimates:
\begin{equation}\label{est:Boot}
	\left\{
	\begin{aligned}
		|b_{n}(s)|
		\le C_{1}|s|^{-4}&,\ \ \mathcal{N}_{B}(s)\leq |s|^{-\frac{17}{4}},\\
		\left|\lambda_n(s)-|s|^{-\frac{1}{\theta}}\right|
		\le& C_{1}|s|^{-1-\frac{1}{\theta}}\log |s|
        \\
		\left|x_{1n}(s)+\frac{\theta}{\theta-1}|s|^{1-\frac{1}{\theta}}\right|\le& C_{1}|s|^{-\frac{1}{\theta}}\log |s|
        \ \ \mbox{and}\ \
         \|\varepsilon_n(s)\|_{L^2}^2\le  C_1 |s|^{-1}.
	\end{aligned}
	\right.
\end{equation}
For $s_{0}\ll -1$ (independent with $n$), we define $S_{n}^{\star}\in [S_{n},s_{0}]$ by 
\begin{equation*}
	S_{n}^{\star}=\sup\left\{s\in [S_{n},s_{0}] \ \mbox{such that}~\eqref{est:NBbound} \ \mbox{and} \ \eqref{est:Boot}\ \mbox{hold on}\ [S_{n},S_{n}^{\star}]\right\}.
\end{equation*}
Note that, from~\eqref{equ:deflambdan}, Remark~\ref{re:G} and $\varepsilon_{n}(T_{n})=0$, we find $S_{n}^{\star}\in \left(S_{n},s_{0}\right]$. The following Proposition is the main part of the proof for Theorem~\ref{thm:main}.
\begin{proposition}\label{prop:Sn}
There exist $C_{1}>1$, $B>1$ and $s_{0}<-1$, independent of $n$, such that for all $n\in \mathbb{N}^{+}$ large enough, we have $S_{n}^\star=s_{0}$.
\end{proposition}

The rest of the article is organized as follows. First, in Section~\ref{S:EV}, we devote to the proof of Proposition~\ref{prop:Sn} based on the standard ODE argument and the energy-virial estimate. Then, in Section~\ref{S:Endproof}, we prove Theorem~\ref{thm:main} from Proposition~\ref{prop:Sn} via the control of the $H^{\frac{7}{9}}$ norm and the compactness argument. 

\smallskip
For simplicity of notation, we drop the $n$ index of the function $(W_{b_{n},\lambda_{n}},\lambda_{n},b_{n},x_{1n},\varepsilon_{n})$ in the following part of the article. In what follows, we work on the time interval $\mathcal{I}=[S_{n},S_{n}^{\star}]$ where the bootstrap estimates~\eqref{est:NBbound} and~\eqref{est:Boot} hold. In every step of the proof, the implied constants in $\lesssim$ and $O$ do not depend on the constant $C_{1}$ appearing in the bootstrap assumption~\eqref{est:Boot}.

\smallskip
Note that, from Lemma~\ref{le:refined} and the bootstrap estimates~\eqref{est:NBbound} and~\eqref{est:Boot},
\begin{equation}\label{est:Geomain}
	\begin{aligned}
		\left|b_{s}-b\lambda^{\theta}\right|&\lesssim
		|s|^{-1}\|\varepsilon\|_{L^{2}_{\rm{loc}}}+
		 |s|^{-5}+C_{1}^{10}|s|^{-6},\\
	\left|\frac{\lambda_s}{\lambda}+\Gamma+b\right|+\left|\frac{x_{1s}}{\lambda}-1\right|&\lesssim
	\|\varepsilon\|_{L^{2}_{\rm{loc}}}+
	 C_{1}|s|^{-5}+C_{1}^{10}|s|^{-6}.
	\end{aligned}
\end{equation}
Note also that, from ~\eqref{est:locNb},~\eqref{est:Boot} and the above estimate, 
\begin{equation}\label{est:Geo}
\begin{aligned}
		\left|b_{s}-b\lambda^{\theta}\right|&\lesssim |s|^{-5}+C_{1}^{10}|s|^{-6},\\
		\left|\frac{\lambda_s}{\lambda}+\Gamma+b\right|+\left|\frac{x_{1s}}{\lambda}-1\right|&\lesssim |s|^{-\frac{17}{4}}+C_{1}^{10}|s|^{-5}.
\end{aligned}
\end{equation}

\section{Energy-Virial Estimate}\label{S:EV}
To complete the proof of Proposition~\ref{prop:Sn}, we improve the bootstrap estimates~\eqref{est:Boot} using a variant of the energy-virial functional first introduced in~\cite[Section 4]{CLY}.

\subsection{Energy estimate}\label{SS:Energy}
In this subsection, we first introduce the weighted energy estimate for the remainder term $\varepsilon$. We denote
\begin{equation*}
	\mathcal{F}=\int_{\mathbb{R}^2}\left(|\nabla\varepsilon|^2\psi_B+\varepsilon^2\varphi_B-\frac{1}{2}\psi_{B}\left((W_{b,\lambda}+\varepsilon)^4-W_{b,\lambda}^4-4W_{b,\lambda}^3\varepsilon\right)\right)\dd y.
\end{equation*}

The following qualitative estimate of the time variation of $\mathcal{F}$ is true.
\begin{proposition}\label{prop:energy}
	There exist some universal constants $B>1$ large enough, $s_{0}<-1$ with $|s_{0}|$ large enough and $0<\kappa_{1}<B^{-1}$ small enough such that the following holds. Assume that for all $s\in [S_{n},S_{n}^{\star}]$, the solution $\phi(t)$ satisfies the bootstrap estimates~\eqref{est:NBbound} and~\eqref{est:Boot} with $0<\kappa<\kappa_{1}$. Then for all $s\in [S_{n},S_{n}^{\star}]$, we have 
\begin{equation}\label{est:dsF}
	\frac{\dd \mathcal{F}}{\dd s}+\frac{1}{8}\int_{\R^{2}}\left(|\nabla \varepsilon|^2+\varepsilon^2\right)\varphi'_{B}\dd y\le \frac{C_2}{B^{30}}
    \int_{\R^{2}}\left(|\nabla \varepsilon|^2+\varepsilon^2\right)\psi_{B}\dd y+\frac{C_3}{|s|^{10}}.
	\end{equation}
	Here, $C_{2}>1$ is a universal constant independent of $B$ and $C_{3}=C_{3}(B,C_{1})>1$ is a constant depending only on $B$ and $C_{1}$.
\end{proposition}

To complete the proof of Proposition~\ref{prop:energy}, we first recall the following 2D weighted Sobolev estimates introduced in~\cite[Section 4.1]{CLY}. The proof relies on a standard argument based on the Fundamental Theorem and the Cauchy-Schwarz inequality (see \emph{e.g.}~\cite[Lemma 4.3]{CLY}), and we omit it.

\begin{lemma}[\cite{CLY}]\label{le:2DSobolev}
	Let $\gamma:\R^{2}\to (0,\infty)$ be a $C^{2}$ function such that 
	\begin{equation*}
		\left\|\frac{\nabla \gamma}{\gamma}\right\|_{L^{\infty}}
		+\sum_{|\beta|=2}\left\|\frac{\partial_{y}^{\beta} \gamma}{\gamma}\right\|_{L^{\infty}}
		\lesssim 1.
	\end{equation*}
		Then, for all $f\in H^{1}(\R^{2})$, we have 
	\begin{equation*}
		\|f^{2}\sqrt{\gamma}\|_{L^{2}}
		\lesssim \|f\|_{L^{2}}\left(\int_{\R^{2}}\left(
		|\nabla f|^{2}+f^{2}
		\right)\gamma\dd y\right)^{\frac{1}{2}}.
	\end{equation*}
	In addition, for all $f\in H^{2}(\R^{2})$, we have 
	\begin{equation*}
		\begin{aligned}
		\|f^{2}\sqrt{\gamma}\|_{L^{\infty}}
		&
		\lesssim \|f\|_{L^{2}}\left(\int_{\R^{2}}\left|
		\frac{\partial^{2}f}{\partial y_{1}\partial {y_{2}}}
		\right|^{2}\gamma \dd y\right)^{\frac{1}{2}}+\left(\int_{\R^{2}}
		\left(|\nabla f|^{2}+f^{2}\right)\sqrt{\gamma}\dd y
		\right).
		\end{aligned}
	\end{equation*}
\end{lemma}

Note that, from the definition of $\psi_{B}$ and $\varphi_{B}$, we have 
\begin{equation}\label{est:pointpsiBphiB}
	\begin{aligned}
	\left\|\frac{\nabla \psi_{B}}{\psi_{B}}\right\|_{L^{\infty}}
    +\left\|\frac{\nabla \sqrt{\psi'_{B}}}{\sqrt{\psi'_{B}}}\right\|_{L^{\infty}}
	+	\left\|\frac{\nabla \varphi_{B}}{\varphi_{B}}\right\|_{L^{\infty}}
	&\lesssim 1,\\
	\sum_{|\beta|=2}	\left\|\frac{\partial_{y}^{\beta} \psi_{B}}{\psi_{B}}\right\|_{L^{\infty}}
    +\sum_{|\beta|=2}	\left\|\frac{\partial_{y}^{\beta} \sqrt{\psi'_{B}}}{\sqrt{\psi'_{B}}}\right\|_{L^{\infty}}
	+	\sum_{|\beta|=2}	\left\|\frac{\partial_{y}^{\beta} \varphi_{B}}{\varphi_{B}}\right\|_{L^{\infty}}
	&\lesssim 1.
\end{aligned}
\end{equation}

Second, we introduce the following estimates relate to $\left(\Psi_{M},\Psi_{N},\Psi_{W},\Psi_{b}\right)$.
\begin{lemma}\label{le:PsiMNWb}
	For all $s\in [S_{n},S_{n}^{\star}]$, the following estimates hold.
	\begin{enumerate}
		\item \emph{Pointwise estimates.} It holds 
		\begin{equation*}
\left|\Psi_{M}\right|+\left|\Psi_{N}\right|+\left|\Psi_{W}\right|+\left|\Psi_{b}\right|\lesssim |s|^{-1}.
		\end{equation*}
        \item \emph{Weighted $L^{2}$ estimate.} It holds 
        \begin{equation*}
            \int_{\R}\int_{B^{10}}^{\infty}y_{1}^{2}
            \left(
            \Psi_{M}^{2}+\Psi_{N}^{2}+\Psi_{W}^{2}+\Psi_{b}^{2}
            \right)\varphi'_{B}\dd y_{1}\dd y_{2}\lesssim C_{1}^{2}|s|^{-10}.
        \end{equation*}
        \item \emph{Weighted $H^{1}$ estimates.} It holds
        \begin{equation*}
		\mathcal{N}_{B}\left(\Psi_{M}\right)
		+\mathcal{N}_{B}\left(\Psi_{N}\right)
+\mathcal{N}_{B}\left(\Psi_{W}\right)+\mathcal{N}_{B}\left(\Psi_{b}\right)\le C_{4}|s|^{-5}.
	\end{equation*}
    Here, $C_{4}=C_{4}\left(B,C_{1}\right)>1$ is a constant depending only on $B$ and $C_{1}$.
		\end{enumerate}
\end{lemma}
\begin{proof}
    Proof of (i). These estimates follow directly from~\eqref{est:Boot}--\eqref{est:Geo}, Proposition~\ref{prop:W}--\ref{prop:Wb} and the definition of $\left(\Psi_{M},\Psi_{N},\Psi_{W},\Psi_{b}\right)$.

    \smallskip
    Proof of (ii). Note that, from~\eqref{est:Boot},~\eqref{est:Geo}, Lemma~\ref{le:Nonloca} and Proposition~\ref{prop:W}, 
    \begin{equation*}
        \left(|\Psi_{M}|+|\Psi_{N}|+|\Psi_{W}|+|\Psi_{b}|\right)
        {\textbf{1}}_{[B^{10},\infty)}(y_{1})\lesssim C_{1}|s|^{-5}e^{-\frac{|y|}{4}}.
    \end{equation*}
    Integrating the above estimate over $[B^{10},\infty)\times\R$, we complete the proof of (ii).

\smallskip
    Proof of (iii). Using again~\eqref{est:Boot},~\eqref{est:Geo}, Lemma~\ref{le:Nonloca} and Proposition~\ref{prop:W}, 
    \begin{equation*}
        \begin{aligned}
            \mathcal{N}_{B}\left(\Psi_{M}\right)
            &=\left(\int_{\R^{2}}\left(|\nabla \Psi_{M}|^{2}\psi_{B}+|\Psi_{M}|^{2}\varphi_{B}\right)\dd y\right)^{\frac{1}{2}}\\
            &\lesssim |s|^{-5}\left(\int_{\R^{2}}
            \left(e^{-\frac{|y|}{4}}+e^{-\frac{|y_{2}|}{2}}e^{\frac{y_{1}}{B}}{\textbf{1}}_{(-\infty,0)}(y_{1})\right)
            \dd y\right)^{\frac{1}{2}}
            \lesssim B^{\frac{1}{2}}|s|^{-5},\\
                \mathcal{N}_{B}\left(\Psi_{N}\right)
            &=\left(\int_{\R^{2}}\left(|\nabla \Psi_{N}|^{2}\psi_{B}+|\Psi_{N}|^{2}\varphi_{B}\right)\dd y\right)^{\frac{1}{2}}\\
            &\lesssim C_{1}|s|^{-5}\left(\int_{\R^{2}}
            \left(e^{-\frac{|y|}{4}}+e^{-\frac{|y_{2}|}{2}}e^{\frac{y_{1}}{B}}{\textbf{1}}_{(-\infty,0)}(y_{1})\right)
            \dd y\right)^{\frac{1}{2}}
            \lesssim C_{1}B^{\frac{1}{2}}|s|^{-5}.
        \end{aligned}
    \end{equation*}
   Based on a similar argument as above, we deduce that 
    \begin{equation*}
    \begin{aligned}
         \mathcal{N}_{B}\left(\Psi_{W}\right)
          &=\left(\int_{\R^{2}}\left(|\nabla \Psi_{W}|^{2}\psi_{B}+|\Psi_{W}|^{2}\varphi_{B}\right)\dd y\right)^{\frac{1}{2}}\\
            &\lesssim |s|^{-5}\left(\int_{\R^{2}}
            \left(e^{-\frac{|y|}{4}}+\langle y_{1}\rangle^{3} e^{-\frac{|y_{2}|}{2}}e^{\frac{y_{1}}{2B}}{\textbf{1}}_{(-\infty,0)}(y_{1})\right)
            \dd y\right)^{\frac{1}{2}}
            \lesssim B^{2}|s|^{-5},\\
             \mathcal{N}_{B}\left(\Psi_{b}\right)
            &=\left(\int_{\R^{2}}\left(|\nabla \Psi_{b}|^{2}\psi_{B}+|\Psi_{b}|^{2}\varphi_{B}\right)\dd y\right)^{\frac{1}{2}}\\
            &\lesssim C_{1}|s|^{-5}\left(\int_{\R^{2}}
            \left(e^{-\frac{|y|}{4}}+e^{-\frac{|y_{2}|}{2}}e^{\frac{y_{1}}{2B}}{\textbf{1}}_{(-\infty,0)}(y_{1})\right)
            \dd y\right)^{\frac{1}{2}}
            \lesssim C_{1}B^{\frac{1}{2}}|s|^{-5}.
         \end{aligned}
    \end{equation*}
    Combining the above estimates, we complete the proof of (iii).
  \end{proof}

We now give a complete proof of Proposition~\ref{prop:energy}.

\begin{proof}[Proof of Proposition~\ref{prop:energy}]
By an elementary computation, we decompose
\begin{equation}\label{equ:dsFij}
\frac{\dd \mathcal{F}}{\dd s}=\mathcal{J}_{1}+\mathcal{J}_{2}+\mathcal{J}_{3},
\end{equation}
where
\begin{equation*}
	\begin{aligned}
	\mathcal{J}_{1}&=
	-2\int_{\R^{2}}\psi_{B}\left(
	\left(W_{b,\lambda}+\varepsilon\right)^{3}-W_{b,\lambda}^{3}-3W_{b,\lambda}^{2}\varepsilon
	\right)\partial_{s}W_{b,\lambda}\dd y,
	\\
	\mathcal{J}_{2}&=
	2\frac{\lambda_s}{\lambda}\int_{\R^{2}}\Lambda\varepsilon
	\left(-\nabla\cdot(\psi_B\nabla\varepsilon)+\varphi_{B}\varepsilon-\psi_B\left((W_{b,\lambda}+\varepsilon)^3-W_{b,\lambda}^3\right)\right)\dd y,
	\\
	\mathcal{J}_{3}&=2\int_{\R^{2}}
	\left(\partial_{s}\varepsilon-\frac{\lambda_s}{\lambda}\Lambda\varepsilon\right)\left(-\nabla\cdot(\psi_B\nabla\varepsilon)+\varphi_{B}\varepsilon-\psi_B\left((W_{b,\lambda}+\varepsilon)^3-W_{b,\lambda}^3\right)\right)\dd y.
\end{aligned}
\end{equation*}

\textbf{Step 1.} Estimate on $\mathcal{J}_{1}$. We claim that, there exist some universal constants $B>1$ large enough, $0<\kappa_{1}<B^{-1}$ small enough and $|s_{0}|$ large enough (possible depending on $B$ and $C_{1}$) such that 
\begin{equation}\label{est:J1}
	\mathcal{J}_{1}\lesssim B^{-30}
	\int_{\R^{2}}\left(|\nabla \varepsilon|^2+\varepsilon^2\right)\psi_{B}\dd y.
\end{equation}
Indeed, from the definition of $W_{b,\lambda}$ in Section~\ref{SS:Refined}, we compute 
\begin{equation*}
	\begin{aligned}
	\partial_{s}W_{b,\lambda}=\frac{\lambda_{s}}{\lambda}\sum_{k=1}^{4}k\theta \lambda^{k\theta}X_{k}\Theta
	+\frac{8}{5}\frac{\lambda_{s}}{\lambda}\sum_{k=1}^{4}\lambda^{k\theta}X_{k}\left(y_{1}\partial_{y_{1}}\Theta\right)+b_{s}\left(
	\chi_{b}+\frac{3}{4}y_{1}\chi'_{b}
	\right)P.
	\end{aligned}
\end{equation*}
From~\eqref{est:Boot} and~\eqref{est:Geo}, for $|s_{0}|$ large enough (depending on $C_{1}$), we obtain 
\begin{equation*}
	\left|\partial_{s}W_{b,\lambda}\right|\lesssim \lambda^{\theta}\left(\lambda^{\theta}+|b|+|s|^{-\frac{17}{4}}\right)
	+|b|\lambda^{\theta}+|s|^{-5}\lesssim |s|^{-2}.
\end{equation*}
Then, using again~\eqref{est:Boot}, for $|s_{0}|$ large enough (depending on $C_{1}$), we find
\begin{equation*}
	|W_{b,\lambda}|\lesssim e^{-\frac{|y|}{2}}+\lambda^{\theta}+|b|\lesssim e^{-\frac{|y|}{2}}+|s|^{-1}.
\end{equation*}
Based on the above estimate,~\eqref{est:NBbound},~\eqref{est:pointpsiBphiB} and Lemma~\ref{le:2DSobolev}, we deduce that 
\begin{equation*}
	\begin{aligned}
		\mathcal{J}_{1}
		&\lesssim |s|^{-2}\int_{\R^{2}}\varepsilon^{2}\left(e^{-\frac{|y|}{2}}+|s|^{-1}+|\varepsilon|\right)
		\psi_{B}
		\dd y\\
		&\lesssim B^{-30}\int_{\R^{2}}\varepsilon^2\psi_{B}\dd y
		+B^{-30}\big\|\varepsilon\sqrt{\psi_{B}}\big\|_{L^{2}}
		\big\|\varepsilon^{2}\sqrt{\psi_{B}}\big\|_{L^{2}}
		\\
		&\lesssim B^{-30}\int_{\R^{2}}\varepsilon^2\psi_{B}\dd y+B^{-30}\int_{\R^{2}}\left(|\nabla \varepsilon|^{2}+ \varepsilon^2\right)\psi_{B}\dd y.
	\end{aligned}
\end{equation*}
We see that~\eqref{est:J1} follows from the above estimate.

\smallskip
\textbf{Step 2.} Estimate on $\mathcal{J}_{2}$. We claim that, 
there exist some universal constants $B>1$ large enough, $0<\kappa_{1}<B^{-1}$ small enough and $|s_{0}|$ large enough (possible depending on $B$ and $C_{1}$) such that 
\begin{equation}\label{est:J2}
	\mathcal{J}_{2}\lesssim 	B^{-30}\int_{\R^{2}}
    \left(|\nabla \varepsilon|^2+\varepsilon^{2}\right)\left(\psi_{B}+B\varphi'_{B}\right)\dd y+|s|^{-10}.
\end{equation}
Indeed, by integration by parts, we decompose
\begin{equation*}
	\mathcal{J}_{2}=\mathcal{J}_{2,1}+\mathcal{J}_{2,2}+\mathcal{J}_{2,3},
\end{equation*}
where
\begin{equation*}
	\begin{aligned}
		\mathcal{J}_{2,1}&=2\frac{\lambda_{s}}{\lambda}\int_{\R^{2}}\psi_{B}\Lambda W_{b,\lambda}\left(3W_{b,\lambda}\varepsilon^{2}+\varepsilon^{3}\right)\dd y,\\
			\mathcal{J}_{2,2}&=\frac{\lambda_{s}}{\lambda}\int_{\R^{2}}
		\left(2\psi_{B}-y_{1}\psi'_{B}\right)|\nabla \varepsilon|^{2}\dd y-\frac{\lambda_{s}}{\lambda}\int_{\R^{2}}y_{1}\varphi'_{B}\varepsilon^{2}\dd y,\\
		\mathcal{J}_{2,3}&=-\frac{1}{2}\frac{\lambda_{s}}{\lambda}\int_{\R^{2}}
		\left(2\psi_{B}-y_{1}\psi'_{B}\right)(6W^{2}_{b}\varepsilon^{2}+4W_{b,\lambda}\varepsilon^{3}+\varepsilon^{4})\dd y.
	\end{aligned}
\end{equation*}

\emph{Estimate on $\mathcal{J}_{2,1}$.} From~\eqref{est:NBbound},~\eqref{est:Boot},~\eqref{est:Geo},~\eqref{est:pointpsiBphiB} and Lemma~\ref{le:2DSobolev}, we find 
\begin{equation}\label{est:J21}
	\begin{aligned}
	\mathcal{J}_{2,1}
	&\lesssim 
	|s|^{-1}\int_{\R^{2}}\varepsilon^2\psi_{B}\dd y
	+|s|^{-1}\big\|\varepsilon\sqrt{\psi_{B}}\big\|_{L^{2}}
	\big\|\varepsilon^{2}\sqrt{\psi_{B}}\big\|_{L^{2}}
	\\
	&\lesssim B^{-30}\int_{\R^{2}}\varepsilon^2\psi_{B}\dd y+B^{-30}\int_{\R^{2}}\left(|\nabla \varepsilon|^{2}+ \varepsilon^2\right)\psi_{B}\dd y.
	\end{aligned}
\end{equation}

\emph{Estimate on $\mathcal{J}_{2,2}$.}
First, using~\eqref{est:Boot},~\eqref{est:Geo} and Lemma~\ref{le:psiphi2}, we deduce that 
\begin{equation*}
	\frac{\lambda_{s}}{\lambda}\int_{\R^{2}}
	\left(2\psi_{B}-y_{1}\psi'_{B}\right)|\nabla \varepsilon|^{2}\dd y\lesssim B^{-30}\int_{\R^{2}}|\nabla \varepsilon|^2
    \left(\psi_{B}+B\varphi'_{B}\right)\dd y.
\end{equation*}
Second, using again~\eqref{est:Boot} and~\eqref{est:Geo}, we have 
\begin{equation*}
	\frac{\lambda_{s}}{\lambda}=(\theta|s|)^{-1}+O(|s|^{-2})>0\Longrightarrow 
-\frac{\lambda_{s}}{\lambda}\int_{\R}\int_{0}^{\infty}y_{1}\varphi'_{B}\varepsilon^{2}\dd y_{1}\dd y_{2}\le 0.
\end{equation*}
Then, for any $B$ large enough and $y_{1}<0$, we have 
\begin{equation*}
	|y_{1}\varphi'_{B}|\lesssim \left|B\varphi'_{B}\right|^{\frac{99}{100}}
	\Longrightarrow
	-\frac{\lambda_{s}}{\lambda}\int_{\R}\int_{-\infty}^{0}y_{1}\varphi'_{B}\varepsilon^{2}\dd y_{1}\dd y_{2}\lesssim B^{-30}\int_{\R^{2}}\varepsilon^{2}\varphi'_{B}\dd y+|s|^{-10}.
\end{equation*}
Combining the above estimates, we obtain 
\begin{equation}\label{est:J22}
	\mathcal{J}_{2,2}\lesssim
	B^{-30}\int_{\R^{2}}\left(|\nabla \varepsilon|^2
    +\varepsilon^{2}\right)
    \left(\psi_{B}+B\varphi'_{B}\right)\dd y+|s|^{-10}.
\end{equation}

\emph{Estimate on $\mathcal{J}_{2,3}$.}
Using again~\eqref{est:Boot},~\eqref{est:Geo},~\eqref{est:pointpsiBphiB}, Lemma~\ref{le:psiphi2} and Lemma~\ref{le:2DSobolev},
\begin{equation}\label{est:J23}
	\begin{aligned}
	\mathcal{J}_{2,3}
	&\lesssim |s|^{-1}\int_{\R^{2}}\left(\varepsilon^{2}+\varepsilon^{4}\right)
	\left(\psi_{B}+\sqrt{\psi_{B}}\right)
	\dd y\\
	&\lesssim B^{-30}\int_{\R^{2}}
    \left(|\nabla \varepsilon|^2+\varepsilon^{2}\right)\left(\psi_{B}+B\varphi'_{B}\right)\dd y.
\end{aligned}
\end{equation}

We see that~\eqref{est:J2} follows from~\eqref{est:J21},~\eqref{est:J22} and~\eqref{est:J23}.

\smallskip
\textbf{Step 3.} Estimate on $\mathcal{J}_{3}$. We claim that, 
there exist some universal constants $B>1$ large enough, $0<\kappa_{1}<B^{-1}$ small enough and $|s_{0}|$ large enough (possible depending on $B$ and $C_{1}$) such that
\begin{equation}\label{est:J3}
\begin{aligned}
&\mathcal{J}_{3}+\frac{1}{4}\int_{\R^{2}}
\left(|\nabla \varepsilon|^{2}+\varepsilon^{2}\right)\varphi'_{B}\dd y\le
 \frac{C_{5}}{B^{30}}\int_{\R^{2}}\left(|\nabla \varepsilon|^{2}+\varepsilon^{2}\right)\psi_{B}\dd y+
    \frac{C_{6}}{|s|^{10}}.
\end{aligned}
\end{equation}
Here, $C_{5}>1$ is a universal constant independent of $B$ and $C_{6}=C_{6}(B,C_{1})>1$ is a constant depending only on $B$ and $C_{1}$.

Indeed, from the equation of $\varepsilon$ in~\eqref{equ:e}, we decompose
\begin{equation*}
\mathcal{J}_{3}=\mathcal{J}_{3,1}+\mathcal{J}_{3,2}+\mathcal{J}_{3,3},
\end{equation*}
where
\begin{equation*}
\begin{aligned}
\mathcal{J}_{3,1}&=-2\int_{\R^{2}}
\mathcal{E}_{b}\left(-\nabla\cdot(\psi_B\nabla\varepsilon)+\varphi_{B}\varepsilon-\psi_B\left((W_{b,\lambda}+\varepsilon)^3-W_{b,\lambda}^3\right)\right)\dd y,\\
\mathcal{J}_{3,2}&=-2\int_{\R^{2}}
\partial_{y_{1}}\left(\Delta \varepsilon-\varepsilon+(W_{b,\lambda}+\varepsilon)^{3}-W_{b,\lambda}^{3}\right)\\
&\quad \quad\quad \quad  \times
\left(-\nabla\cdot(\psi_B\nabla\varepsilon)+\varphi_{B}\varepsilon-\psi_B\left((W_{b,\lambda}+\varepsilon)^3-W_{b,\lambda}^3\right)\right)\dd y
,\\
\mathcal{J}_{3,3}&=2\left(\frac{x_{1s}}{\lambda}-1\right)\int_{\R^{2}}
\left(\partial_{y_1}\varepsilon\right)
\left(-\nabla\cdot(\psi_B\nabla\varepsilon)+\varphi_{B}\varepsilon-\psi_B\left((W_{b,\lambda}+\varepsilon)^3-W_{b,\lambda}^3\right)\right)\dd y.
\end{aligned}
\end{equation*}
\emph{Estimate on $\mathcal{J}_{3,1}$.} From the definition of $\mathcal{E}_{b}$ in Section~\ref{SS:Geo}, we decompose 
\begin{equation*}
	\mathcal{J}_{3,1}=\mathcal{J}_{3,1,1}+\mathcal{J}_{3,1,2}+\mathcal{J}_{3,1,3}+\mathcal{J}_{3,1,4},
\end{equation*}
where
\begin{equation*}
\begin{aligned}
	\mathcal{J}_{3,1,1}&=2\int_{\mathbb{R}^2}
	\left(\Psi_{M}-\Psi_{N}\right)
	\left(-\nabla\cdot(\psi_B\nabla\varepsilon)+\varphi_{B}\varepsilon-\psi_B\left((W_{b,\lambda}+\varepsilon)^3-W_{b,\lambda}^3\right)\right)\dd y,\\
    \mathcal{J}_{3,1,2}&=2\int_{\mathbb{R}^2}
\left(-\Psi_{W}-\Psi_{b}\right)
\left(-\nabla\cdot(\psi_B\nabla\varepsilon)+\varphi_{B}\varepsilon-\psi_B\left((W_{b,\lambda}+\varepsilon)^3-W_{b,\lambda}^3\right)\right)\dd y,\\
\mathcal{J}_{3,1,3}&=2\left(\frac{\lambda_{s}}{\lambda}+\Gamma+b\right)\int_{\mathbb{R}^2}\Lambda Q\left(-\nabla\cdot(\psi_B\nabla\varepsilon)+\varphi_{B}\varepsilon-\psi_B\left((W_{b,\lambda}+\varepsilon)^3-W_{b,\lambda}^3\right)\right)\dd y,\\
\mathcal{J}_{3,1,4}&=2\left(\frac{x_{1s}}{\lambda}-1\right)\int_{\mathbb{R}^2}\left(\partial_{y_{1}} Q\right)
\left(-\nabla\cdot(\psi_B\nabla\varepsilon)+\varphi_{B}\varepsilon-\psi_B\left((W_{b,\lambda}+\varepsilon)^3-W_{b,\lambda}^3\right)\right)\dd y.
\end{aligned}
\end{equation*}
First, from Lemma~\ref{le:psiphi}, Lemma~\ref{le:psiphi2},  
Lemma~\ref{le:2DSobolev} and Cauchy-Schwarz inequality, 
\begin{equation*}
    \begin{aligned}
        \mathcal{J}_{3,1,1}
        &\lesssim
        \mathcal{N}_{B}\left(\Psi_{M}-\Psi_{N}\right)
   \left( \int_{\R^{2}}\left(|\nabla \varepsilon|^{2}+\varepsilon^{2}\right)
   \left(\psi_{B}+B\varphi'_{B}\right)\dd y\right)^{\frac{1}{2}}\\
   &+\left(\int_{\R}\int_{B^{10}}^{\infty}y_{1}^{2}\left(\Psi_{M}^{2}+\Psi_{N}^{2}\right)\varphi'_{B}\dd y_{1}\dd y_{2}\right)^{\frac{1}{2}}\left(\int_{\R^{2}}\varepsilon^{2}\varphi'_{B}\dd y\right)^{\frac{1}{2}}\\
   &+\left(\|\Psi_{M}\|_{L^{\infty}}+\|\Psi_{N}\|_{L^{\infty}}\right)
\left(\int_{\R^{2}}\varepsilon^{2}\psi_{B}\dd y\right)^{\frac{1}{2}}
\left(\int_{\R^{2}}\varepsilon^{4}\psi_{B}\dd y\right)^{\frac{1}{2}},
    \end{aligned}
\end{equation*}
\begin{equation*}
    \begin{aligned}
        \mathcal{J}_{3,1,2}
        &\lesssim
        \mathcal{N}_{B}\left(\Psi_{W}+\Psi_{B}\right)
   \left( \int_{\R^{2}}\left(|\nabla \varepsilon|^{2}+\varepsilon^{2}\right)
   \left(\psi_{B}+B\varphi'_{B}\right)\dd y\right)^{\frac{1}{2}}\\
   &+\left(\int_{\R}\int_{B^{10}}^{\infty}y_{1}^{2}\left(\Psi_{W}^{2}+\Psi_{b}^{2}\right)\varphi'_{B}\dd y_{1}\dd y_{2}\right)^{\frac{1}{2}}\left(\int_{\R^{2}}\varepsilon^{2}\varphi'_{B}\dd y\right)^{\frac{1}{2}}\\
   &+\left(\|\Psi_{W}\|_{L^{\infty}}+\|\Psi_{b}\|_{L^{\infty}}\right)
\left(\int_{\R^{2}}\varepsilon^{2}\psi_{B}\dd y\right)^{\frac{1}{2}}
\left(\int_{\R^{2}}\varepsilon^{4}\psi_{B}\dd y\right)^{\frac{1}{2}}.
    \end{aligned}
\end{equation*}
It follows from Lemma~\ref{le:PsiMNWb} that 
\begin{equation*}
\begin{aligned}
\mathcal{J}_{3,1,1}&\lesssim B^{-30}\int_{\R^{2}}\left(|\nabla \varepsilon|^{2}+\varepsilon^{2}\right)\left(\psi_{B}+B\varphi'_{B}\right)\dd y
+\left(C_{1}^{2}+C_{4}^{2}\right)B^{30}|s|^{-10},\\
\mathcal{J}_{3,1,2}&\lesssim B^{-30}\int_{\R^{2}}\left(|\nabla \varepsilon|^{2}+\varepsilon^{2}\right)\left(\psi_{B}+B\varphi'_{B}\right)\dd y
+\left(C_{1}^{2}+C_{4}^{2}\right)B^{30}|s|^{-10}.
\end{aligned}
\end{equation*}

Second, by an elementary computation, we find 
\begin{equation*}
	\begin{aligned}
	&2\left(-\nabla\cdot(\psi_B\nabla\varepsilon)+\varphi_{B}\varepsilon-\psi_B\left((W_{b,\lambda}+\varepsilon)^3-W_{b,\lambda}^3\right)\right)\\
	&=\mathcal{L}\varepsilon+\left(2\psi_{B}-1\right)
	\left(-\Delta \varepsilon-3W_{b,\lambda}^{2}\varepsilon-3W_{b,\lambda}\varepsilon^{2}-\varepsilon^{3}\right)\\
	&+\left(2\varphi_{B}-1\right)\varepsilon-
	\left(3(W_{b,\lambda}^{2}-Q^{2})\varepsilon+3W_{b,\lambda}\varepsilon^{2}+\varepsilon^{3}\right)-2\psi'_{B}\partial_{y_{1}}\varepsilon.
	\end{aligned}
\end{equation*}
Therefore, from~\eqref{equ:ortho},~\eqref{est:Geomain}, Proposition~\ref{prop:Spectral} and Lemma~\ref{le:psiphi}, we deduce that 
\begin{equation*}
	\begin{aligned}
	\mathcal{J}_{3,1,3}+\mathcal{J}_{3,1,4}
	&\lesssim B^{-30}
	\|\varepsilon\|_{L^{2}_{\rm{loc}}}
	\left(\|\varepsilon\|_{L^{2}_{\rm{loc}}}+C_{1}|s|^{-5}+C_{1}^{10}|s|^{-6}\right)\\
&\lesssim B^{-30}\int_{\R^{2}}\left(|\nabla \varepsilon|^{2}+\varepsilon^{2}\right)\psi_{B}\dd y+C_{1}^{2}|s|^{-10}+C_{1}^{20}|s|^{-12}.
\end{aligned}
\end{equation*}
Combining the above estimates, we conclude that 
\begin{equation}\label{est:J31}
\begin{aligned}
    \mathcal{J}_{3,1}&\lesssim
    B^{-30}\int_{\R^{2}}\left(|\nabla \varepsilon|^{2}+\varepsilon^{2}\right)\left(\psi_{B}+B\varphi'_{B}\right)\dd y\\
    &+
    C_{1}^{20}|s|^{-10}+C_{1}^{2}
    B^{30}|s|^{-10}
    +C_{4}^{2}B^{30}|s|^{-10}.
    \end{aligned}
\end{equation}

\emph{Estimate on $\mathcal{J}_{3.2}$.} 
By integration by parts, we decompose 
\begin{equation*}
\mathcal{J}_{3,2}=\mathcal{J}_{3,2,1}+\mathcal{J}_{3,2,2}+\mathcal{J}_{3,2,3}+\mathcal{J}_{3,2,4},
\end{equation*}
where
\begin{equation*}
\begin{aligned}
\mathcal{J}_{3,2,1}&=-\int_{\R^{2}}|\nabla \varepsilon|^{2}\left(\varphi'_{B}+\psi'_{B}-
\psi'''_{B}\right)\dd y\\
&-2\int_{\R^{2}}(\partial_{y_{1}}\varepsilon)^{2}\varphi'_{B}\dd y
-\int_{\R^{2}}\varepsilon^{2}\left(\varphi'_{B}-\varphi'''_{B}\right)\dd y,\\
\mathcal{J}_{3,2,2}&=-\int_{\R^{2}}
\left(3(\partial^{2}_{y_{1}}\varepsilon)^{2}+4(\partial_{y_{1}}\partial_{y_{2}}\varepsilon)^{2}+(\partial^{2}_{y_{2}}\varepsilon)^{2}\right)
\psi'_{B}
\dd y,\quad \qquad \qquad \
\end{aligned}
\end{equation*}
\begin{equation*}
    \begin{aligned}
    \mathcal{J}_{3,2,3}
&=-2\int_{\R^{2}}(\partial_{y_{1}}W_{b,\lambda})\left(3W_{b,\lambda}\varepsilon^{2}+\varepsilon^{3}\right)
\left(\varphi_{B}-\psi_{B}\right)
\dd y\\
&+\frac{1}{2}\int_{\R^{2}}(6W^{2}_{b}\varepsilon^{2}+8W_{b,\lambda}\varepsilon^{3}+3\varepsilon^{4})
\left(\varphi'_{B}-\psi'_{B}\right)
\dd y\\
&+6\int_{\R^{2}}\left(\partial_{y_{1}}\varepsilon\right)
\left((\partial_{y_{1}}W_{b,\lambda})(2W_{b,\lambda}\varepsilon+\varepsilon^{2})+(\partial_{y_{1}}\varepsilon)(W_{b,\lambda}+\varepsilon)^{2}\right)
\psi'_{B}
\dd y,\\
\mathcal{J}_{3,2,4}&=-
\int_{\R^{2}}\left(\left(-\Delta \varepsilon+\varepsilon-(W_{b,\lambda}+\varepsilon)^{3}+W_{b,\lambda}^{3}\right)^{2}-
\left(-\Delta \varepsilon+\varepsilon\right)^{2}\right)
\psi'_{B}\dd y.
\end{aligned}
\end{equation*}
First, from Lemma~\ref{le:psiphi2}, we deduce that 
\begin{equation*}
\begin{aligned}
    \mathcal{J}_{3,2,1}
    +\frac{1}{2}\int_{\R^{2}}\left(|\nabla \varepsilon|^{2}+\varepsilon^{2}\right)\varphi'_{B}\dd y
    \lesssim B^{-30}\int_{\R^{2}}\left(|\nabla \varepsilon|^{2}+\varepsilon^{2}\right)\left(\psi_{B}+B\varphi'_{B}\right)\dd y.
    \end{aligned}
\end{equation*}

Second, from the definition of $\mathcal{J}_{3,2,2}$, we have
\begin{equation*}
    \mathcal{J}_{3,2,2}\le -\int_{\R^{2}}
\left((\partial^{2}_{y_{1}}\varepsilon)^{2}+(\partial_{y_{1}}\partial_{y_{2}}\varepsilon)^{2}+(\partial^{2}_{y_{2}}\varepsilon)^{2}\right)
\psi'_{B}
\dd y\le -\sum_{|\beta|=2}\int_{\R^{2}}\left|\partial_{y}^{\beta}\varepsilon\right|^{2}\psi'_{B}\dd y.
\end{equation*}

Then, using~\eqref{est:Boot},~\eqref{est:pointpsiBphiB}, Lemma~\ref{le:psiphi2}--\ref{le:psiphi3} and Lemma~\ref{le:2DSobolev}, we find 
\begin{equation*}
    \mathcal{J}_{3,2,3}\lesssim B^{-30}\int_{\R^{2}}\left(|\nabla \varepsilon|^{2}+\varepsilon^{2}\right)\left(\psi_{B}+B\varphi'_{B}\right)\dd y+\int_{\R^{2}}(\partial_{y_{1}} \varepsilon)^{2}\varepsilon^{2}\psi'_{B}\dd y.
\end{equation*}
Note that, from~\eqref{est:pointpsiBphiB} and Lemma~\ref{le:2DSobolev},
\begin{equation*}
    \begin{aligned}
        \int_{\R^{2}}(\partial_{y_{1}} \varepsilon)^{2}\varepsilon^{2}\psi'_{B}\dd y
        &\lesssim
        \left\|\varepsilon^{2}\sqrt{\psi'_{B}}\right\|_{L^{\infty}}
       \left( \int_{\R^{2}}|\nabla \varepsilon|^{2}\sqrt{\psi'_{B}}\dd y\right)\\
        &\lesssim \|\varepsilon\|_{L^{2}}\left(\int_{\R^{2}}|\nabla \varepsilon|^{2}\sqrt{\psi'_{B}}\dd y\right)
        \left(\int_{\R^{2}}\left|\frac{\partial^{2}\varepsilon}{\partial y_{1}\partial y_{2}}\right|^{2}\psi'_{B}\dd y\right)^{\frac{1}{2}}\\
        &+\left(\int_{\R^{2}}|\nabla \varepsilon|^{2}\sqrt{\psi'_{B}}\dd y\right)\left(\int_{\R^{2}}\left(|\nabla \varepsilon|^{2}+\varepsilon^{2}\right)\sqrt{\psi'_{B}}\dd y\right).
    \end{aligned}
\end{equation*}
Note also that, from Lemma~\ref{le:psiphi2} and integration by parts, 
\begin{equation*}
    \begin{aligned}
       \int_{\R^{2}}\varepsilon^{2}\sqrt{\psi'_{B}}\dd y&\lesssim
       \int_{\R^{2}}\varepsilon^{2}\left(B^{-\frac{1}{2}}\psi_{B}+B^{\frac{1}{2}}\varphi'_{B}\right)\dd y\lesssim \mathcal{N}_{B}^{2}\left(\varepsilon\right),\\
       \int_{\R^{2}}|\nabla \varepsilon|^{2}\sqrt{\psi'_{B}}\dd y&\lesssim
       \|\varepsilon\|_{L^{2}}
       \left(\int_{\R^{2}}|\Delta \varepsilon|^{2}\psi'_{B}\dd y\right)^{\frac{1}{2}}+\int_{\R^{2}}\varepsilon^{2}\sqrt{\psi'_{B}}\dd y.
    \end{aligned}
\end{equation*}
Based on the above estimates, we deduce that 
\begin{equation*}
    \mathcal{J}_{3,2,3}\lesssim B^{-30}\int_{\R^{2}}\left(|\nabla \varepsilon|^{2}+\varepsilon^{2}\right)\left(\psi_{B}+B\varphi'_{B}\right)\dd y+B^{-30}\sum_{|\beta|=2}\int_{\R^{2}}\left|\partial_{y}^{\beta}\varepsilon\right|^{2}\psi'_{B}\dd y.
\end{equation*}
Using a similar argument as above, we also deduce that 
\begin{equation*}
    \mathcal{J}_{3,2,4}\lesssim B^{-30}\int_{\R^{2}}\left(|\nabla \varepsilon|^{2}+\varepsilon^{2}\right)\left(\psi_{B}+B\varphi'_{B}\right)\dd y+B^{-30}\sum_{|\beta|=2}\int_{\R^{2}}\left|\partial_{y}^{\beta}\varepsilon\right|^{2}\psi'_{B}\dd y.
\end{equation*}
Here, we use the fact that 
\begin{equation*}
    \left\|\varepsilon^{2}\sqrt{\psi'_{B}}\right\|_{L^{\infty}}^{2}\lesssim
    \|\varepsilon\|_{L^{2}}^{2}\int_{\R^{2}}\left|\Delta \varepsilon\right|^{2}\psi'_{B}\dd y+\mathcal{N}_{B}^{2}(\varepsilon)\int_{\R^{2}}\varepsilon^{2}\sqrt{\psi'_{B}}\dd y.
\end{equation*}
Combining the above estimates, we conclude that 
\begin{equation}\label{est:J32}
    \mathcal{J}_{3,2}+\frac{1}{4}\int_{\R^{2}}\left(|\nabla \varepsilon|^{2}+\varepsilon^{2}\right)\varphi'_{B}\dd y
    \lesssim B^{-30}\int_{\R^{2}}\left(|\nabla \varepsilon|^{2}+\varepsilon^{2}\right)\psi_{B}\dd y.
\end{equation}

\emph{Estimate on $\mathcal{J}_{3,3}$.} 
By integration by parts, we compute 
\begin{equation*}
\begin{aligned}
   &2 \int_{\R^{2}}
\left(\partial_{y_1}\varepsilon\right)
\left(-\nabla\cdot(\psi_B\nabla\varepsilon)+\varphi_{B}\varepsilon-\psi_B\left((W_{b,\lambda}+\varepsilon)^3-W_{b,\lambda}^3\right)\right)\dd y\\
&=-\int_{\R^{2}}\left(|\nabla \varepsilon|^{2}\psi'_{B}+\varepsilon^{2}\varphi'_{B}
-3\varepsilon^{2}\partial_{y_{1}}(W^{2}_{b}\psi_{B})-2\varepsilon^{3}\partial_{y_{1}}(W_{b,\lambda}\psi_{B})-\frac{1}{2}\varepsilon^{4}\psi'_{B}\right)\dd y.
\end{aligned}
\end{equation*}
It follows from~\eqref{est:Geo},~\eqref{est:pointpsiBphiB} and Lemma~\ref{le:2DSobolev} that 
\begin{equation}\label{est:J33}
    \mathcal{J}_{3,3}\lesssim B^{-30}\int_{\R^{2}}\left(|\nabla \varepsilon|^{2}+\varepsilon^{2}\right)\left(\psi_{B}+B\varphi'_{B}\right)\dd y.
\end{equation}
We see that~\eqref{est:J3} follows from~\eqref{est:J31},~\eqref{est:J32} and~\eqref{est:J33}.

\smallskip
\textbf{Step 4.} Conclusion. Combining the estimates~\eqref{est:J1},~\eqref{est:J2} and~\eqref{est:J3}, we complete the proof of~\eqref{est:dsF} by taking $B$ large enough.
\end{proof}

\subsection{Virial estimate}\label{SS:Virial}
In this subsection, we introduce the virial estimate for the solution of~\eqref{CP}. As in the previous works~\cite{CLY,CLYMINI}, the Schr\"odinger operator appearing in the virial-type estimate does not yet have a suitable coercivity. Therefore, we should introduce a transform, which was first studied in Kenig-Martel~\cite{KenigMartel}, to perform the virial estimate after converting the original problem into a dual one. The following Schr\"odinger operator will appear in the estimate of the dual problem:
\begin{equation*}
\begin{aligned}
    \mathcal{H}f
    &=-\frac{1}{2}\Delta f-\partial_{y_{1}}^{2}f+\frac{1}{2}f-\frac{3}{2}\left(Q^{2}+2y_{1}Q\partial_{y_{1}}Q\right)f\\
    &+\frac{3}{(Q,Q)}\left((f,Q^{2}\partial_{y_{1}}Q)y_{1}Q+(f,y_{1}Q)Q^{2}\partial_{y_{1}}Q\right),\quad \mbox{for any}\ f\in H^{1}(\R^{2}).
    \end{aligned}
\end{equation*}
Recall that, the following coercivity of the Schr\"odinger operator $\mathcal{H}$ has been verified in~\cite[Section 16]{FHRY} via the numerical computation.
\begin{lemma}[\cite{FHRY}]\label{le:coerH}
    There exists $\mu_{2}>0$ such that, for all $f\in H^{1}(\R^{2})$,
    \begin{equation*}
        \left(\mathcal{H}f,f\right)\ge \mu_{2}\|f\|_{H^{1}}^{2}-\frac{1}{\mu_{2}}\left((f,Q)^{2}+(f,\partial_{y_{1}}Q)^{2}+(f,\partial_{y_{2}}Q)^{2}\right).
    \end{equation*}
\end{lemma}
To state the virial estimate for the solution of~\eqref{CP},  we should first consider the following transformed problem:
\begin{equation*}
    \eta=(1-\delta \Delta)^{-1}\mathcal{L}\varepsilon,\quad \mbox{for all}\ s\in [S_{n},S_{n}^{\star}].
\end{equation*}
Here, $0<\delta<1$ is a small enough constant (depending on $B$) to be chosen later.

\smallskip
To simplify notation, we denote
\begin{equation*}
R_{NL}=\left(W_{b,\lambda}+\varepsilon\right)^{3}-W_{b,\lambda}^{3}-3Q^{2}\varepsilon \quad \mbox{and}
\quad 
    {\rm{Mod}}_{\eta}=\frac{\lambda_s}{\lambda}\Lambda\varepsilon+\left(\frac{x_{1s}}{\lambda}-1\right)\partial_{y_{1}}\varepsilon.
\end{equation*}

\smallskip
By an elementary computation and equation of $\varepsilon$ in~\eqref{equ:e}, we have the following equation of $\eta$ and the identity related to $\left(\Lambda\varepsilon, \Lambda \eta\right)$.
\begin{lemma}\label{le:equeta}
The following identities hold.
\begin{enumerate}
    \item \emph{Equation of $\eta$.} It holds
    \begin{equation*}
\begin{aligned}
\partial_{s}\eta
&=\mathcal{L}\partial_{y_{1}}\eta-3\delta(1-\delta\Delta)^{-1}\left(\Delta(Q^{2})\partial_{y_{1}}\eta+4Q\nabla Q\cdot \nabla \partial_{y_{1}}\eta\right)\\
&+(1-\delta\Delta)^{-1}\mathcal{L}{\rm{Mod}}_{\eta}
-2\left(\frac{\lambda_{s}}{\lambda}+\Gamma+b\right)(1-\delta\Delta)^{-1}Q\\
&+(1-\delta\Delta)^{-1}{\mathcal{L}}(\Psi_{M}-\Psi_{N}-\Psi_{W}-\Psi_{b}-\partial_{y_{1}}R_{NL}).
\end{aligned}
\end{equation*}
In addition, the function $\eta$ satisfies the following orthogonality conditions
\begin{equation*}
\left(\eta,(1-\delta\Delta)Q\right)=\left(\eta,(1-\delta\Delta)\partial_{y_{1}}Q\right)=\left(\eta,(1-\delta\Delta)\partial_{y_{2}}Q\right)=0.
\end{equation*}

\item \emph{{Identity related to $\left(\Lambda \varepsilon,\Lambda \eta\right)$}.}
It holds 
\begin{equation*}
\begin{aligned}
    (1-\delta \Delta)^{-1}\mathcal{L}\Lambda \varepsilon
    &=\Lambda \eta 
    +3(1-\delta\Delta)^{-1}\left(\varepsilon y\cdot\nabla \left(Q^{2}\right)\right)\\
&+2\delta(1-\delta \Delta)^{-2}\Delta\mathcal{L}\varepsilon-2(1-\delta\Delta)^{-1}\Delta\varepsilon.
    \end{aligned}
\end{equation*}
\end{enumerate}
    \end{lemma}

\begin{proof}
Proof of (i). Note that,
\begin{equation*}
    \left[Q^{2},\left(1-\delta\Delta\right)\right]=\delta \Delta (Q^{2})+4\delta Q\nabla Q\cdot \nabla.
\end{equation*}
Based on the above identity, the equation of $\varepsilon$ in~\eqref{equ:e}, Proposition~\ref{prop:Spectral} and an elementary computation, we complete the proof of (i).

    \smallskip
Proof of (ii). From the definition of $\eta$, we deduce that 
\begin{equation*}
    \mathcal{L}\Lambda \varepsilon=(1-\delta\Delta)\Lambda \eta 
    +y\cdot \nabla \mathcal{L}\varepsilon-(1-\delta\Delta)(y\cdot\nabla \eta)-2\Delta\varepsilon
+3\varepsilon y\cdot\nabla \left(Q^{2}\right).
\end{equation*}
Using again the definition of $\eta$, we find 
\begin{equation*}
\begin{aligned}
  & y\cdot \nabla \mathcal{L}\varepsilon-(1-\delta\Delta)(y\cdot\nabla \eta)\\
  &=y\cdot \nabla \mathcal{L}\varepsilon-(1-\delta\Delta)y\cdot
  (\nabla(1-\delta\Delta)^{-1} \mathcal{L}\varepsilon)=2\delta(1-\delta\Delta)^{-1}\Delta\mathcal{L}\varepsilon.
    \end{aligned}
\end{equation*}
Combining the above two identities, we complete the proof of (ii).
\end{proof}

Recall that, we define the smooth cut-off function $\sigma\in [0,1]$ as follows:
\begin{equation*}
\sigma_{|(-1,1)}\equiv1\quad \mbox{and}\quad 
\sigma_{|(-\infty,-2)}\equiv \sigma_{|(2,\infty)}\equiv 0.
\end{equation*}

In addition, we define the following smooth functions $\psi_{0}\in (0,1]$
and $\psi_{1}\in (0,\infty)$:
\begin{equation*}
\psi_{0}(y_1)=
\begin{cases}
\exp\left({6y_1}\right),&\text{ for }y_1<-1,\\
\frac{1}{2},&\text{ for }y_1>-\frac{1}{2},
\end{cases}
\quad \mbox{with}\ \  \psi_0'(y_1)\ge 0,\ \ \mbox{on}\ \mathbb{R},
\end{equation*}
\begin{equation*}
\psi_{1}(y_1)=
\begin{cases}
\exp\left({10y_1}\right),&\text{ for }y_1<-1,\\
1+y_{1},&\text{ for }y_1>-\frac{1}{2},
\end{cases}
\quad \mbox{with}\ \  \psi_1'(y_1)> 0,\ \ \mbox{on}\ \mathbb{R}.
\end{equation*}
Let $B>1$ be a large enough universal constant to be chosen later. We consider the following two weight functions:
\begin{equation*}
    \psi_{0,B}(y_{1})=\psi_{0}\left(\frac{y_{1}}{B}\right)
    \quad \mbox{and}\quad 
    \psi_{1,B}(y_{1})=\psi_{1}\left(\frac{y_{1}}{B}\right).
\end{equation*}
We also define
\begin{equation*}
\rho_B(y_1)=
\begin{cases}
\sigma\left(\frac{y_1}{2B}\right)\int_0^{y_1}\frac{2}{B}\psi_{0,B}(\tau)\dd \tau,\quad\quad  \mbox{for}\ y_{1}\le 0,\\
\sigma\left(\frac{y_1}{10B^{10}}\right)\int_0^{y_1}\frac{2}{B}\psi_{0,B}(\tau)\dd \tau,\quad \mbox{for}\ y_{1}> 0.
\end{cases}
\end{equation*}
By the definition of the weight functions, we have the following pointwise estimate.

\begin{lemma}\label{le:psi0Bpsi1B}
The following estimates hold.
\begin{enumerate}
    \item \emph{Estimates on $\varphi_{B}$.} It holds
    \begin{equation*}
\begin{aligned}
|\varphi''_{B}|&\lesssim B^{-\frac{2}{3}}\varphi'_{B}+B^{-20}\psi_{0,B},\quad \mbox{on}\ \R,\\
|\varphi'''_{B}|&\lesssim B^{-\frac{4}{3}}\varphi'_{B}+B^{-30}\psi_{0,B},\quad \mbox{on}\ \R.
\end{aligned}
\end{equation*}
\item \emph{Estimate on $\rho_{B}$ and $\psi_{0,B}$.} It holds
\begin{equation*}
\begin{aligned}
 B|\rho'_{B}|+B^{2}|\rho''_{B}|+B^{3}|\rho'''_{B}|&\lesssim \psi_{0,B},\quad \mbox{on}\ \R\\
    B|\psi'_{0,B}|+B^{2}|\psi''_{0,B}|+B^{3}|\psi'''_{0,B}|&\lesssim \psi_{0,B},\quad \mbox{on}\ \R.
    \end{aligned}
\end{equation*}
\item \emph{First-type estimate on $\rho_{B}$.} It holds 
\begin{equation*}
    \rho_{B}\lesssim \min
    \left(B^{9}\psi_{0,B},\psi_{1,B}\sqrt{\psi_{0,B}}\right),\quad \mbox{on}\ \R.
\end{equation*}
\item \emph{Second-type estimate on $\rho_{B}$.} It holds 
\begin{equation*}
\begin{aligned}
\left|\rho'_{B}-\frac{2}{B}\psi_{0,B}\right|&\lesssim \min\left(B^{9}\varphi'_{B},\mathbf{1}_{\left(-\infty,-\frac{B}{2}\right]}(y_{1})+\mathbf{1}_{\left[B^{10},\infty\right)}(y_{1})\right), \ \ \mbox{on}\ \R,\\
\left|\rho_{B}-\frac{2y_{1}}{B}\psi_{0,B}\right|&\lesssim \left(\mathbf{1}_{\left(-\infty,-\frac{B}{2}\right]}(y_{1})+\mathbf{1}_{\left[B^{10},\infty\right)}(y_{1})\right)|y_{1}|,\quad \quad \quad \ \ \mbox{on}\ \R.
\end{aligned}
\end{equation*}
\end{enumerate}
\end{lemma}

\begin{proof}
    The proof is directly based on Lemma~\ref{le:psiphi}, Lemma~\ref{le:psiphi2} and the definition of $\rho_{B}$ and $\psi_{0,B}$, and we omit it.
\end{proof}

We denote
\begin{equation*}
\mathcal{P}(s)=\int_{\mathbb{R}^2}\eta^{2}(s,y)\rho_{B}(y_{1})\dd y,\quad \mbox{for any}\ s\in \left[S_{n},S_{n}^{\star}\right].
\end{equation*}

We now state the virial estimate of $\eta$. Let $B>1$ be a large enough constant to be chosen later and $\delta=B^{-3}$. Then the following qualitative estimate of the time variation of $\mathcal{P}$ is true.

\begin{proposition}\label{prop:virial}
	There exist some universal constants $B>1$ large enough, $s_{0}<-1$ with $|s_{0}|$ large enough and $0<\kappa_{1}<B^{-1}$ small enough such that the following holds. Assume that for all $s\in [S_{n},S_{n}^{\star}]$, the solution $\phi(t)$ satisfies the bootstrap estimates~\eqref{est:NBbound} and~\eqref{est:Boot} with $0<\kappa<\kappa_{1}$. Then for all $s\in [S_{n},S_{n}^{\star}]$, we have 
\begin{equation}\label{est:dsP}
    	\begin{aligned}
	&\frac{\dd \mathcal{P}}{\dd s}+\frac{\mu_{2}}{B}\int_{\R^{2}}\left(|\nabla \eta|^2+\eta^2\right)\psi_{0,B}\dd y\\
&\le \frac{C_{7}}{B^{8}}\int_{\R^{2}}\left(|\nabla \varepsilon|^2+\varepsilon^2\right)\left(\psi_{0,B}+B^{23}\varphi'_{B}\right)\dd y+\frac{C_{8}}{|s|^{10}}.
	\end{aligned}
	\end{equation}
    Here, $C_{7}>1$ is a universal constant independent of $B$ and $C_{8}=C_{8}(B,C_{1})>1$ is a constant depending only on $B$ and $C_{1}$.
\end{proposition}

To complete the proof of Proposition~\ref{prop:virial}, we first recall the following 2D weighted Sobolev estimates introduced in~\cite[Section 4.2]{CLY}. The proof relies on a standard argument based on the Fourier transform and the Cauchy-Schwarz inequality (see \emph{e.g.}~\cite[Lemma 4.11]{CLY}), and we omit it.

\begin{lemma}[\cite{CLY}]\label{le:2DSobolevvirial}
	Let $\gamma:\R^{2}\to (0,\infty)$ be a $C^{2}$ function such that 
	\begin{equation*}
		\left\|\frac{\nabla \gamma}{\gamma}\right\|_{L^{\infty}}
		+\sum_{|\beta|=2}\left\|\frac{\partial_{y}^{\beta} \gamma}{\gamma}\right\|_{L^{\infty}}
		\lesssim 1.
	\end{equation*}
		Then, for all $f\in H^{2}(\R^{2})$ and $\ell\in \left\{0,1,2\right\}$, we have
	\begin{equation*}
	\sum_{|\beta|=\ell}\left\|\gamma (1-\delta\Delta)^{-1}\partial_{y}^{\beta}f\right\|_{L^{2}}\lesssim \delta^{-\frac{\ell}{2}}\left\| f\gamma\right\|_{L^{2}}.
	\end{equation*}
\end{lemma}
Note that, the functions $\psi_{0,B}$ and $\psi_{1,B}$ satisfy
\begin{equation}\label{est:pointpsi0Bpsi1B}
	\begin{aligned}
	\left\|\frac{\nabla \psi_{1,B}}{\psi_{1,B}}\right\|_{L^{\infty}}
  +\sum_{|\beta|=2}	\left\|\frac{\partial_{y}^{\beta} \psi_{1,B}}{\psi_{1,B}}\right\|_{L^{\infty}}
	&\lesssim 1,\\
    \left\|\frac{\nabla \sqrt{\psi_{0,B}}}{\sqrt{\psi_{0,B}}}\right\|_{L^{\infty}}
  +\sum_{|\beta|=2}	\left\|\frac{\partial_{y}^{\beta} \sqrt{\psi_{0,B}}}{\sqrt{\psi_{0,B}}}\right\|_{L^{\infty}}
	&\lesssim 1.
\end{aligned}
\end{equation}
Second, we introduce the following relations between the weighted norm of $(\varepsilon,\eta)$.
\begin{lemma}\label{le:viresteeta}
    Let $B>1$ be a large enough constant and $0<\delta<1$ be a small enough constant. Then the following estimates hold.
    \begin{enumerate}
    \item {\emph{First-type estimates.}} It holds
        \begin{equation*}
        \begin{aligned}
        \int_{\R^{2}}\left(|\nabla \varepsilon|^{2}+\varepsilon^{2}\right)\psi_{0,B}\dd y&\lesssim 
        \int_{\R^{2}}\left(\delta^{2}|\nabla \eta|^{2}+\eta^{2}\right)\psi_{0,B}\dd y,\\
        \int_{\R^{2}}\left(\delta|\nabla \eta|^{2}+\eta^{2}\right)\psi_{0,B}\dd y&\lesssim 
        \int_{\R^{2}}\left(\delta^{-1}|\nabla \varepsilon|^{2}+\varepsilon^{2}\right)\psi_{0,B}\dd y,\\
        \int_{\R^{2}}\left(\delta|\nabla \eta|^{2}+\eta^{2}\right)\psi_{1,B}\dd y&\lesssim 
        \int_{\R^{2}}\left(\delta^{-1}|\nabla \varepsilon|^{2}+\varepsilon^{2}\right)\psi_{1,B}\dd y.
        \end{aligned}
    \end{equation*}
    Here, the implied constants in $\lesssim$ are independent of $B$ and $\delta$.

    \item  \emph{Second-type estimate.} It holds
    \begin{equation*}
    \begin{aligned}
        \int_{\R^{2}}\left(\delta |\nabla \eta|^{2}+\eta^{2}\right)\varphi'_{B}\dd y&\lesssim 
        \int_{\R^{2}}\left(
        \delta^{-1}|\nabla \varepsilon|^{2}+\varepsilon^{2}
        \right)\varphi'_{B}\dd y\\
        &+B^{-20}\int_{\R^{2}}\left(
        |\nabla \varepsilon|^{2}+\varepsilon^{2}+\eta^{2}
        \right)\psi_{0,B}\dd y.
        \end{aligned}
    \end{equation*}
    Here, the implied constants in $\lesssim$ are independent of $B$ and $\delta$.
    \end{enumerate}
\end{lemma}
\begin{proof}
    The proof relies on a standard argument based on integration by parts and the coercivity of $\mathcal{L}$ (see \emph{e.g.}~\cite[Lemma 4.7 and 4.8]{CLY}), and we omit it.
\end{proof}

We now give a complete proof of Proposition~\ref{prop:virial}.
\begin{proof}[Proof of Proposition~\ref{prop:virial}]
By Lemma~\ref{le:equeta}, we decompose
\begin{equation}\label{equ:dsP}
\frac{\dd \mathcal{P}}{\dd s}
=\mathcal{J}_{4}+\mathcal{J}_{5}+\mathcal{J}_{6}+\mathcal{J}_{7},
\end{equation}
where 
\begin{equation*}
    \begin{aligned}
           \mathcal{J}_{4}&=
        2\int_{\R^{2}}\left((1-\delta\Delta)^{-1}\mathcal{L}
        \left({\rm{Mod}}_{\eta}
        -\partial_{y_{1}}R_{NL}
        \right)
        \right)
        \eta \rho_{B}\dd y,
        \\
\mathcal{J}_{5}&=2\int_{\R^{2}}\left((1-\delta\Delta)^{-1}{\mathcal{L}}(\Psi_{M}-\Psi_{N}-\Psi_{W}-\Psi_{b})\right)\eta \rho_{B}\dd y,\\
      \mathcal{J}_{6}&=-6\delta\int_{\R^{2}}
        \left((1-\delta\Delta)^{-1}\left(\Delta(Q^{2})\partial_{y_{1}}\eta+4Q\nabla Q\cdot \nabla \partial_{y_{1}}\eta\right)\right)\eta \rho_{B}\dd y,\\
            \mathcal{J}_{7}&=2\int_{\R^{2}}
        \left(\mathcal{L}\partial_{y_{1}}\eta\right)\eta \rho_{B}\dd y
-4\left(\frac{\lambda_{s}}{\lambda}+\Gamma+b\right)
\int_{\R^{2}}
\left((1-\delta\Delta)^{-1}Q\right)\eta \rho_{B}\dd y.
    \end{aligned}
\end{equation*}

\textbf{Step 1.}
Estimate on $\mathcal{J}_{4}$.
We claim that, there exist some universal constants $B>1$ large enough, $0<\kappa_{1}<B^{-1}$ small enough and $|s_{0}|$ large enough (possible depending on $B$ and $C_{1}$) such that
\begin{equation}\label{est:J4}
\begin{aligned}
\mathcal{J}_{4}&\lesssim
 B^{-10}
        \int_{\R^{2}}\left(|\nabla \eta|^{2}+\eta^{2}\right)
        \psi_{0,B}\dd y\\
        &+B^{-10}
        \int_{\R^{2}}\left(|\nabla \varepsilon|^{2}+\varepsilon^{2}\right)
        \left(\psi_{0,B}+B\varphi'_{B}\right)\dd y.
\end{aligned}
\end{equation}
Indeed, from Lemma~\ref{le:equeta} and integration by parts, we decompose
\begin{equation*}
    \mathcal{J}_{4}=\mathcal{J}_{4,1}+\mathcal{J}_{4,2}+\mathcal{J}_{4,3}+\mathcal{J}_{4,4},
\end{equation*}
where
\begin{equation*}
\begin{aligned}
&
\mathcal{J}_{4,1}=2\int_{\R^{2}}\left((1-\delta\Delta)^{-1}\left(-\Delta+1-3Q^{2}\right)\partial_{y_{1}}R_{NL}\right)\eta \rho_{B}\dd y,
\\
   &\mathcal{J}_{4,2}=2\left(\frac{x_{1s}}{\lambda}-1\right)\int_{\R^{2}}
\left(\left(1-\delta\Delta\right)^{-1}\left(-\Delta+1-3Q^{2}\right)\partial_{y_{1}}\varepsilon\right)\eta \rho_{B}\dd y,\\
        &\mathcal{J}_{4,3}=-\frac{\lambda_{s}}{\lambda}\int_{\R^{2}}\eta ^{2}\rho'_{B}\dd y+6\frac{\lambda_{s}}{\lambda}\int_{\R^{2}}\left((1-\delta\Delta)^{-1}\left(\varepsilon y\cdot\nabla \left(Q^{2}\right)\right)\right)\eta \rho_{B}\dd y,\\
&\mathcal{J}_{4,4}=4\delta\frac{\lambda_{s}}{\lambda}\int_{\R^{2}}\left((1-\delta\Delta)^{-2}\Delta\mathcal{L}\varepsilon\right)\eta \rho_{B}\dd y
-4\frac{\lambda_{s}}{\lambda}\int_{\R^{2}}\left((1-\delta\Delta)^{-1}\Delta \varepsilon\right)\eta \rho_{B}\dd y.
        \end{aligned}
\end{equation*}
\emph{Estimate on $\mathcal{J}_{4,1}$.} 
Note that, from the definition of $R_{NL}$, we directly have 
\begin{equation*}
    \left|R_{NL}\right|
    \lesssim |W_{b,\lambda}^{2}-Q^{2}||\varepsilon|+|W_{b,\lambda}|\varepsilon^{2}+|\varepsilon|^{3}
    \lesssim \left(\lambda^{\theta}+|b|\right)|\varepsilon|+\varepsilon^{2}+|\varepsilon|^{3}.
\end{equation*}
It follows from~\eqref{est:pointpsi0Bpsi1B}, Lemma~\ref{le:2DSobolevvirial} and 2D Sobolev estimate that
\begin{equation*}
\begin{aligned}
    &\left\|\sqrt{\psi_{0,B}}\left((1-\delta\Delta)^{-1}\left(-\Delta+1-3Q^{2}\right)R_{NL}\right)\right\|_{L^{2}}\\
    & \lesssim B^{3}\left(
\left(\lambda^{\theta}+|b|\right)\left\|\varepsilon\sqrt{\psi_{0,B}}\right\|_{L^{2}}+\left\|\varepsilon\psi^{\frac{1}{4}}_{0,B}\right\|^{2}_{H^{1}}+\left\|\varepsilon\psi_{0,B}^{\frac{1}{6}}\right\|^{3}_{H^{1}}
    \right).
    \end{aligned}
    \end{equation*}
Based on the above estimate and~\eqref{est:Boot}, we find
\begin{equation*}
\begin{aligned}
   & \left\|\sqrt{\psi_{0,B}}\left((1-\delta\Delta)^{-1}\left(-\Delta+1-3Q^{2}\right)R_{NL}\right)\right\|_{L^{2}}\\
   &\lesssim B^{-30}\left(\int_{\R^{2}}\left(|\nabla \varepsilon|^{2}+\varepsilon^{2}\right)\left(\psi_{0,B}+B\varphi'_{B}\right)\dd y\right)^{\frac{1}{2}}.
    \end{aligned}
\end{equation*}
Here, we use the fact that 
\begin{equation*}
    \begin{aligned}
        \psi_{0,B}^{\frac{1}{2}}&\lesssim \psi_{B}\lesssim \varphi_{B}\quad 
        \mbox{and}\quad  \psi_{0,B}^{\frac{1}{2}}\lesssim \psi_{0,B}+B\varphi'_{B},\\
          \psi_{0,B}^{\frac{1}{3}}&\lesssim \psi_{B}\lesssim \varphi_{B}\quad 
        \mbox{and}\quad  \psi_{0,B}^{\frac{1}{3}}\lesssim \psi_{0,B}+B\varphi'_{B}.
    \end{aligned}
\end{equation*}
Combining the above estimate with Lemma~\ref{le:psi0Bpsi1B}, we obtain
\begin{equation}\label{est:J41}
    \begin{aligned}
        \mathcal{J}_{4,1}&\lesssim B^{-10}
        \int_{\R^{2}}\left(|\nabla \eta|^{2}+\eta^{2}\right)
        \psi_{0,B}\dd y\\
        &+B^{-10}
        \int_{\R^{2}}\left(|\nabla \varepsilon|^{2}+\varepsilon^{2}\right)
        \left(\psi_{0,B}+B\varphi'_{B}\right)\dd y.
    \end{aligned}
\end{equation}
\emph{Estimate on $\mathcal{J}_{4,2}-\mathcal{J}_{4,4}$.} 
Based on~\eqref{est:Geo} and a similar argument as above, we directly have 
\begin{equation}\label{est:J42J44}
\begin{aligned}
    \mathcal{J}_{4,2}+\mathcal{J}_{4,3}+\mathcal{J}_{4,4}
    &\lesssim B^{-10}
        \int_{\R^{2}}\left(|\nabla \eta|^{2}+\eta^{2}\right)
        \psi_{0,B}\dd y\\
        &+B^{-10}
        \int_{\R^{2}}\left(|\nabla \varepsilon|^{2}+\varepsilon^{2}\right)
        \left(\psi_{0,B}+B\varphi'_{B}\right)\dd y.
        \end{aligned}
\end{equation}
We see that~\eqref{est:J4} follows from~\eqref{est:J41} and~\eqref{est:J42J44}.

\smallskip
\textbf{Step 2}. Estimate on $\mathcal{J}_{5}$. 
We claim that 
\begin{equation}\label{est:J5}
    \mathcal{J}_{5}\le \frac{C_{9}}{B^{10}}\int_{\R^{2}}\left(|\nabla \eta|^{2}+\eta^{2}\right)\psi_{0,B}\dd y+\frac{C_{10}}{|s|^{10}}.
\end{equation}
Here, $C_{9}>1$ is a universal constant independent of $B$ and $C_{10}=C_{10}(B,C_{1})>1$ is a constant depending only on $B$ and $C_{1}$.

\smallskip
Indeed, from~\eqref{est:Boot}--\eqref{est:Geo} and Proposition~\ref{prop:W}--\ref{prop:Wb}, we deduce that
\begin{equation*}
\begin{aligned}
&\left\|\Psi_{M}\sqrt{\psi_{0,B}}\right\|_{L^{2}}
    +\left\|\Psi_{N}\sqrt{\psi_{0,B}}\right\|_{L^{2}}
+\left\|\Psi_{W}\sqrt{\psi_{0,B}}\right\|_{L^{2}}
    +\left\|\Psi_{b}\sqrt{\psi_{0,B}}\right\|_{L^{2}}
\\
&\lesssim C_{1}|s|^{-5}\left(\int_{\R^{2}}
            \left(e^{-\frac{|y|}{4}}+\langle y_{1}\rangle^{3} e^{-\frac{|y_{2}|}{2}}e^{\frac{6y_{1}}{B}}{\textbf{1}}_{(-\infty,0)}(y_{1})\right)
            \dd y\right)^{\frac{1}{2}}
            \lesssim C_{1}B^{2}|s|^{-5}.
            \end{aligned}
\end{equation*}
Using Lemma~\ref{le:psi0Bpsi1B}, Lemma~\ref{le:2DSobolevvirial} and the Cauchy-Schwarz inequality, we have 
\begin{equation*}
\begin{aligned}
    \mathcal{J}_{5}
    &\lesssim B^{12}\left( \left\|\Psi_{M}\sqrt{\psi_{0,B}}\right\|_{L^{2}}
    +\left\|\Psi_{N}\sqrt{\psi_{0,B}}\right\|_{L^{2}}
    \right)\left\|\eta\sqrt{\psi_{0,B}}\right\|_{L^{2}}\\
    &+B^{12}\left( \left\|\Psi_{W}\sqrt{\psi_{0,B}}\right\|_{L^{2}}
    +\left\|\Psi_{b}\sqrt{\psi_{0,B}}\right\|_{L^{2}}
    \right)\left\|\eta\sqrt{\psi_{0,B}}\right\|_{L^{2}}.
    \end{aligned}
\end{equation*}
Combining the above two estimates, we complete the proof of~\eqref{est:J5}.

\smallskip
\textbf{Step 3.} Estimate on $\mathcal{J}_{6}$. We claim that 
\begin{equation}\label{est:J6}
    \mathcal{J}_{6}\le \frac{C_{11}}{B^{\frac{3}{2}}}\int_{\R^{2}}\left(|\nabla \eta|^{2}+\eta^{2}\right)\psi_{0,B}\dd y.
\end{equation}
Here, $C_{11}>1$ is a universal constant independent of $B$.

\smallskip
Indeed, using an elementary computation, we find
\begin{equation*}
\begin{aligned}
    \mathcal{J}_{6}
    &=-24\delta \int_{\R^{2}}
    \left(
    (1-\gamma\Delta)^{-1}
    \left(\partial_{y_{1}}\left(Q\nabla Q\cdot \nabla\eta \right)\right)
    \right)
   \eta \rho_{B}\dd y
    \\
    &+6\delta\int_{\R^{2}}
        \left(
        (1-\delta\Delta)^{-1}\left(
       4\left( \partial_{y_{1}}\left(Q\nabla Q\right)\right)\cdot \nabla\eta
        -\left(\Delta(Q^{2})\right)\partial_{y_{1}}\eta\right)\right)\eta \rho_{B}\dd y.
        \end{aligned}
\end{equation*}
From~\eqref{est:pointpsi0Bpsi1B} and Lemma~\ref{le:2DSobolevvirial}, we deduce that 
\begin{equation*}
\begin{aligned}
\left\|\psi_{1,B}
        (1-\delta\Delta)^{-1}\left(
       \left(\Delta(Q^{2})\right)\partial_{y_{1}}\eta\right)\right\|_{L^{2}}
       &\lesssim \left\|\left|\nabla \eta\right|\sqrt{\psi_{0,B}}\right\|_{L^{2}},
       \\
    \delta^{\frac{1}{2}}\left\|\psi_{1,B} (1-\delta\Delta)^{-1}
    \partial_{y_{1}}
    \left(Q\nabla Q\cdot \nabla\eta \right)\right\|_{L^{2}}
    &\lesssim \left\||\nabla \eta|\sqrt{\psi_{0,B}}\right\|_{L^{2}},
    \\
    \left\|\psi_{1,B}
        (1-\delta\Delta)^{-1}
        \left(\left(
       \partial_{y_{1}}\left(Q\nabla Q\right)\right)\cdot \nabla\eta
       \right)
       \right\|_{L^{2}}&\lesssim \left\||\nabla \eta|\sqrt{\psi_{0,B}}\right\|_{L^{2}}.
    \end{aligned}
\end{equation*}
It follows from Lemma~\ref{le:psi0Bpsi1B} and the Cauchy-Schwarz inequality that 
\begin{equation*}
    \mathcal{J}_{6}\lesssim \delta \left(1+\delta^{-\frac{1}{2}}\right)\int_{\R^{2}}\left(|\nabla \eta|^{2}+\eta^{2}\right)\psi_{0,B}\dd y.
\end{equation*}
Combining the above estimate with $\delta=B^{-3}$, we complete the proof of~\eqref{est:J6}.

\smallskip
\textbf{Step 4.} Estimate on $\mathcal{J}_{7}$. We claim that, there exist some universal constants $B>1$ large enough, $0<\kappa_{1}<B^{-1}$ small enough and $|s_{0}|$ large enough (possible depending on $B$ and $C_{1}$) such that
\begin{equation}\label{est:J7}
\begin{aligned}
    \mathcal{J}_{7}&\le -\frac{2\mu_{2}}{B}\int_{\R^{2}}\left(
    |\nabla \eta|^{2}+\eta^{2}
    \right)\psi_{0,B}\dd y\\
    &+\frac{C_{12}}{B^{8}}\int_{\R^{2}}\left(
    |\nabla \varepsilon|^{2}+\varepsilon^{2}
    \right)\left(\psi_{0,B}+B^{23}\varphi'_{B}\right)\dd y+\frac{C_{13}}{|s|^{10}}.
    \end{aligned}
\end{equation}
Here, $C_{12}>1$ is a universal constant independent of $B$ and $C_{13}=C_{13}(B,C_{1})>1$ is a constant depending only on $B$ and $C_{1}$.

\smallskip
Indeed, from integration by parts, we find 
\begin{equation*}
\begin{aligned}
   2 \int_{\R^{2}}\left(\mathcal{L}\partial_{y_{1}}\eta\right)\eta \rho_{B}\dd y&=-\int_{\R^{2}}|\nabla\eta|^{2}\rho'_{B}\dd y
-2\int_{\R^{2}}\left(\partial_{y_{1}}\eta\right)^{2}\rho'_{B}\dd y
-\int_{\R^{2}}\eta^{2}\rho'_{B}\dd y\\
&+3\int_{\R^{2}}Q^{2}\eta^{2}\rho'_{B}\dd y
+6\int_{\R^{2}}\left(Q\partial_{y_{1}}Q\right)\eta^{2}\rho_{B}\dd y
+\int_{\R^{2}}\eta^{2}\rho'''_{B}\dd y.
\end{aligned}
\end{equation*}
From the above identity, we rewrite the term $\mathcal{J}_{7}$ by 
\begin{equation*}
\mathcal{J}_{7}=\mathcal{J}_{7,1}+\mathcal{J}_{7,2}+\mathcal{J}_{7,3},
\end{equation*}
where
\begin{equation*}
\begin{aligned}
\mathcal{J}_{7,1}&=
-\frac{2}{B}\int_{\R^{2}}\left(|\nabla\eta|^{2}+2\left(\partial_{y_{1}}\eta\right)^{2}\right)\psi_{0,B}\dd y
-\frac{2}{B}\int_{\R^{2}}\eta^{2}\psi_{0,B}\dd y\\
&+\frac{6}{B}\int_{\R^{2}}\left(Q^{2}+2y_{1}Q\partial_{y_{1}}Q\right)\eta^{2}\psi_{0,B}\dd y
-\frac{12(\eta,y_1Q)}{B(Q,Q)}
(\eta,Q^2\partial_{y_{1}}Q),\qquad \qquad \quad 
\end{aligned}
\end{equation*}
\begin{equation*}
\begin{aligned}
\mathcal{J}_{7,2}=\frac{12}{B}\frac{(\eta,Q^2\partial_{y_{1}}Q)}{(Q,Q)}(\eta,y_1Q)-4\left(\frac{\lambda_s}{\lambda}+\Gamma+b\right) \left((1-\delta\Delta)^{-1}Q,\eta\rho_{B}\right),\quad \quad \ \
\end{aligned}
\end{equation*}
\begin{equation*}
\begin{aligned}
\mathcal{J}_{7,3}&=
-\int_{\R^{2}}\left(|\nabla \eta|^{2}+2\left(\partial_{y_{1}}\eta\right)^{2}+\eta^{2}\right)\left(\rho'_{B}-\frac{2}{B}\psi_{0,B}\right)\dd y+\int_{\R^{2}}\eta^{2}\rho'''_{B}\dd y
\\
&+3\int_{\R^{2}}Q^{2}\eta^{2}\left(\rho'_{B}-\frac{2}{B}\psi_{0,B}\right)\dd y
+6\int_{\R^{2}}\left(Q\partial_{y_{1}}Q\right)\eta^{2}\left(\rho_{B}-\frac{2y_{1}}{B}\psi_{0,B}\right)\dd y.
\end{aligned}
\end{equation*}

\emph{Estimate on $\mathcal{J}_{7,1}$.}
By an elementary computation, we rewrite the term $\mathcal{J}_{7,1}$ by 
\begin{equation*}
\begin{aligned}
\mathcal{J}_{7,1}
&=-\frac{4}{B}\left(\mathcal{H}
\left(\eta \sqrt{\psi_{0,B}}\right),\eta \sqrt{\psi_{0,B}}\right)
-\frac{3}{B}\int_{\R^{2}}\eta^{2}\left(\psi''_{0,B}
-\frac{(\psi'_{0,B})^{2}}{2\psi_{0,B}}
\right)\dd y\\
&-\frac{12}{B(Q,Q)}\left(\left(\eta,y_{1}Q\right)\left(\eta,Q^{2}\partial_{y_{1}}Q\right)
-2\left(\eta\sqrt{\psi_{0,B}},y_{1}Q\right)\left(\eta\sqrt{\psi_{0,B}},Q^{2}\partial_{y_{1}}Q\right)
\right).
\end{aligned}
\end{equation*}
First, from Lemma~\ref{le:coerH}, we deduce that 
\begin{equation*}
    \begin{aligned}
        \left(\mathcal{H}
\left(\eta \sqrt{\psi_{0,B}}\right),\eta \sqrt{\psi_{0,B}}\right)
&\ge \mu_{2}\left\|\eta \sqrt{\psi_{0,B}}\right\|_{H^{1}}^{2}-\frac{1}{\mu_{2}}
\left(\eta \sqrt{\psi_{0,B}},Q\right)^{2}\\
&-\frac{1}{\mu_{2}}
\left(\eta \sqrt{\psi_{0,B}},\partial_{y_{1}}Q\right)^{2}
-\frac{1}{\mu_{2}}
\left(\eta \sqrt{\psi_{0,B}},\partial_{y_{2}}Q\right)^{2}.
    \end{aligned}
\end{equation*}
Note that, from Lemma~\ref{le:psi0Bpsi1B} and integration by parts, we have
\begin{equation*}
    \left\|\eta \sqrt{\psi_{0,B}}\right\|_{H^{1}}^{2}
    =\left(
    1+O\left(B^{-1}\right)
    \right)
    \int_{\R^{2}}\left(|\nabla \eta|^{2}+\eta^{2}\right)\psi_{0,B}\dd y.
\end{equation*}
Note also that, from (i) of Lemma~\ref{le:equeta} and $\delta=B^{-3}$, we have 
\begin{equation*}
 \left(\eta \sqrt{\psi_{0,B}},Q\right)^{2}
 + \left(\eta \sqrt{\psi_{0,B}},\partial_{y_{1}}Q\right)^{2}
 + \left(\eta \sqrt{\psi_{0,B}},\partial_{y_{2}}Q\right)^{2}\lesssim B^{-6}\int_{\R^{2}}\eta^{2}\psi_{0,B}\dd y.
\end{equation*}
Second, using again Lemma~\ref{le:psi0Bpsi1B}, we deduce that 
\begin{equation*}
   \left| \int_{\R^{2}}\eta^{2}\left(\psi''_{0,B}
-\frac{(\psi'_{0,B})^{2}}{2\psi_{0,B}}
\right)\dd y\right|\lesssim B^{-2}\int_{\R^{2}}\eta^{2}\psi_{0,B}\dd y.
\end{equation*}
Third, using the definition of $\psi_{0,B}$, we deduce that 
\begin{equation*}
\begin{aligned}
   & \left|\left(\eta,y_{1}Q\right)\left(\eta,Q^{2}\partial_{y_{1}}Q\right)
-2\left(\eta\sqrt{\psi_{0,B}},y_{1}Q\right)\left(\eta\sqrt{\psi_{0,B}},Q^{2}\partial_{y_{1}}Q\right)
\right|\\
&\lesssim B^{-2}\left(\int_{\R^{2}}\eta^{2}\psi_{0,B}\dd y\right)^{\frac{1}{2}}
\left(\int_{\R^{2}}\eta^{2}e^{-\frac{|y|}{10}}\dd y\right)^{\frac{1}{2}}
\lesssim B^{-2}\int_{\R^{2}}\eta^{2}\psi_{0,B}\dd y.
\end{aligned}
\end{equation*}
Combining the above estimates, for $B$ large enough, we obtain 
\begin{equation}\label{est:J71}
    \mathcal{J}_{7,1}\le -\frac{3\mu_{2}}{B}\int_{\R^{2}}\left(
    |\nabla \eta|^{2}+\eta^{2}
    \right)\psi_{0,B}\dd y.
\end{equation}

\emph{Estimate on $\mathcal{J}_{7,2}$.}
Note that
\begin{equation*}
    \left(1-\delta\Delta\right)^{-1}Q=Q+\delta\left(1-\delta\Delta\right)^{-1}\Delta Q.
\end{equation*}
Therefore, by an elementary computation, we decompose
\begin{equation*}
\mathcal{J}_{7,2}=\mathcal{J}_{7,2,1}+\mathcal{J}_{7,2,2}+\mathcal{J}_{7,2,3},
\end{equation*}
where
\begin{equation*}
    \begin{aligned}
\mathcal{J}_{7,2,1}&=12\frac{\left(\eta,Q^{2}\partial_{y_{1}}Q\right)}{(Q,Q)}\int_{\R^{2}}\eta Q\left(\frac{y_{1}}{B}-\rho_{B}\right)\dd y,\\
        \mathcal{J}_{7,2,2}&=-4\delta \left(
        \frac{\lambda_{s}}{\lambda}+\Gamma+b\right)
        \int_{\R^{2}}\left(
        \left(1-\delta\Delta\right)^{-1}\Delta Q\right)\eta \rho_{B}
        \dd y
        ,\\
        \mathcal{J}_{7,2,3}&=4
        \left(
        3\frac{(\eta,Q^{2}\partial_{y_{1}}Q)}{(Q,Q)}
        -\left(\frac{\lambda_{s}}{\lambda}+\Gamma+b\right)
        \right)\int_{\R^{2}}\eta Q\rho_{B}\dd y.
    \end{aligned}
\end{equation*}
First, from $\rho_{B}(y_{1})=\frac{y_{1}}{B}$ on $\left[-\frac{B}{2},\frac{B}{2}\right]$, we find 
\begin{equation*}
    \mathcal{J}_{7,2,1}\lesssim B^{-3}
    \left(\int_{\R^{2}}|\eta|e^{-\frac{|y|}{10}}\dd y\right)^{2}\lesssim B^{-3}\int_{\R^{2}}\eta^{2}\psi_{0,B}\dd y.
\end{equation*}
Second, using again~\eqref{est:pointpsi0Bpsi1B}, Lemma~\ref{le:psi0Bpsi1B} and Lemma~\ref{le:2DSobolevvirial},
\begin{equation*}
\begin{aligned}
&\left|  \int_{\R^{2}}\left(
        \left(1-\delta\Delta\right)^{-1}\Delta Q\right)\eta \rho_{B}
        \dd y
        \right|\\
        &\lesssim \int_{\R^{2}}\left|\psi_{1,B}\left(
        \left(1-\delta\Delta\right)^{-1}\Delta Q\right)\right|\left|\eta \sqrt{\psi_{0,B}}\right|\dd y\lesssim \left(\int_{\R^{2}}\eta^{2}\psi_{0,B}\dd y\right)^{\frac{1}{2}}.
        \end{aligned}
\end{equation*}
Based on the above estimate, \eqref{est:Geomain} and Lemma~\ref{le:viresteeta}, we find 
\begin{equation*}
    \mathcal{J}_{7,2,2}\lesssim B^{-3}\int_{\R^{2}}\left(|\nabla \eta|^{2}+\eta^{2}\right)\psi_{0,B}\dd y+C_{1}^{20}B^{-3}|s|^{-10}.
\end{equation*}
Third, from (i) of Lemma~\ref{le:equeta}, we compute 
\begin{equation*}
\begin{aligned}
&\left(
\mathcal{L}\partial_{y_{1}}\eta,
\left(1-\delta \Delta\right)Q
\right)-2\left(\frac{\lambda_{s}}{\lambda}+\Gamma+b\right)(Q,Q)\\
&=3\delta\left(\Delta (Q^{2})\partial_{y_{1}}\eta+4Q\nabla Q\cdot \nabla \partial_{y_{1}}\eta,Q\right)\\
&-\left(
{{\rm{Mod}}_{\eta}+\Psi_{M}}-\Psi_{N}
-\Psi_{W}-\Psi_{b}-\partial_{y_{1}}R_{NL}
,\mathcal{L}Q
\right).
\end{aligned}
\end{equation*}
On the other hand, from the fact that $-\Delta Q+Q-Q^{3}=0$, we have 
\begin{equation*}
    \left(
\mathcal{L}\partial_{y_{1}}\eta,
\left(1-\delta \Delta\right)Q
\right)=6\left(\eta, Q^{2}\partial_{y_{1}}Q\right)+\delta\left(\eta,\partial_{y_{1}}\mathcal{L}\Delta Q\right).
\end{equation*}
It follows from (iii) of Lemma~\ref{le:PsiMNWb} and Lemma~\ref{le:viresteeta} that 
\begin{equation*}
\begin{aligned}
    & \left|
        3\frac{(\eta,Q^{2}\partial_{y_{1}}Q)}{(Q,Q)}
        -\left(\frac{\lambda_{s}}{\lambda}+\Gamma+b\right)
        \right|\\
        &\lesssim B^{-3}\left(\int_{\R^{2}}\left(|\nabla \eta|^{2}+\eta^{2}\right)\psi_{0,B}\dd y\right)^{\frac{1}{2}}
        +C_{4}|s|^{-5}.
        \end{aligned}
\end{equation*}
Based on the above estimate and the definition of $\mathcal{J}_{7,2,3}$, we find 
\begin{equation*}
    \mathcal{J}_{7,2,3}\lesssim B^{-3}\int_{\R^{2}}\left(|\nabla \eta|^{2}+\eta^{2}\right)\psi_{0,B}\dd y+
        +C^{2}_{4}B^{3}|s|^{-10}.
\end{equation*}
Combining the above estimates, we obtain 
\begin{equation}\label{est:J72}
    \mathcal{J}_{7,2}\lesssim B^{-3}
    \int_{\R^{2}}\left(|\nabla \eta|^{2}+\eta^{2}\right)\psi_{0,B}\dd y
        +\left(C_{1}^{20}B^{-3}+C^{2}_{4}B^{3}\right)|s|^{-10}.
\end{equation}

\emph{Estimate on $\mathcal{J}_{7,3}$.}
Recall that, from Lemma~\ref{le:psi0Bpsi1B}, we have
\begin{equation*}
    \left|\rho'_{B}-\frac{2}{B}\psi_{0,B}\right|\lesssim B^{9}\varphi'_{B},\quad \mbox{on}\ \R.
\end{equation*}
It follows from (ii) of Lemma~\ref{le:viresteeta} that 
\begin{equation*}
\begin{aligned}
   &\left| \int_{\R^{2}}\left(|\nabla \eta|^{2}+2\left(\partial_{y_{1}}\eta\right)^{2}+\eta^{2}\right)\left(\rho'_{B}-\frac{2}{B}\psi_{0,B}\right)\dd y\right|\\
   &\lesssim B^{15}\int_{\R^{2}}\left(|\nabla \varepsilon|^{2}+\varepsilon^{2}\right)\varphi'_{B}\dd y
   +B^{-8}\int_{\R^{2}}
   \left(|\nabla \varepsilon|^{2}+\varepsilon^{2}+\eta^{2}\right)\psi_{0,B}\dd y.
   \end{aligned}
\end{equation*}
Then, using again Lemma~\ref{le:psi0Bpsi1B} and the exponential decay of $Q$, we have
\begin{equation*}
\begin{aligned}
&\left|\int_{\R^{2}}\eta^{2}\rho'''_{B}\dd y\right|
+\left|\int_{\R^{2}}\eta^{2}Q^{2}\left(\rho'_{B}-\frac{2}{B}\psi_{0,B}\right)\dd y\right|\\
&+\left|\int_{\R^{2}}\eta^{2}\left(Q\partial_{y_{1}}Q\right)\left(\rho_{B}-\frac{2y_{1}}{B}\psi_{0,B}\right)\dd y\right|\lesssim B^{-3}\int_{\R^{2}}\eta^{2}\psi_{0,B}\dd y.
\end{aligned}
\end{equation*}
Combining the above estimates, we obtain 
\begin{equation}\label{est:J73}
\begin{aligned}
    \mathcal{J}_{7,3}&\lesssim B^{-3}\int_{\R^{2}}\left(
    |\nabla \eta|^{2}+\eta^{2}
    \right)\psi_{0,B}\dd y\\
    &+B^{-8}\int_{\R^{2}}\left(|\nabla \varepsilon|^{2}+\varepsilon^{2}\right)\left(\psi_{0,B}+B^{23}\varphi'_{B}\right)\dd y.
    \end{aligned}
\end{equation}
We see that~\eqref{est:J7} follows from~\eqref{est:J71},~\eqref{est:J72} and~\eqref{est:J73}.

\smallskip
\textbf{Step 4.} Conclusion.  Combining the estimates~\eqref{est:J5},~\eqref{est:J6} and~\eqref{est:J7}, we complete the proof of~\eqref{est:dsP} by taking $B$ large enough.
\end{proof}

\subsection{Proof of Proposition~\ref{prop:Sn}}\label{SS:proofSn}
In this subsection, we will give a complete proof of Proposition~\ref{prop:Sn} via the energy method.
We first introduce the following energy-virial Lyapunov functional $\mathcal{W}$ which is based on the combination of the energy and virial quantities defined in Section~\ref{SS:Energy} and Section~\ref{SS:Virial}. Denote 
\begin{equation*}
\mathcal{W}=\mathcal{F}+\frac{\mathcal{P}}{B^{20}},\quad \mbox{for all}\ s\in \left[S_{n},S_{n}^{\star}\right].
\end{equation*}

\begin{proposition}\label{Prop:Mon}
There exist some universal constants $B>1$ large enough, $s_{0}<-1$ with $|s_{0}|$ large enough and $0<\kappa_{1}<B^{-1}$ small enough such that the following holds. Assume that for all $s\in [S_{n},S_{n}^{\star}]$, the solution $\phi(t)$ satisfies the bootstrap estimates~\eqref{est:NBbound} and~\eqref{est:Boot} with $0<\kappa<\kappa_{1}$. Then for all $s\in [S_{n},S_{n}^{\star}]$, the following estimates hold.
\begin{enumerate}
	\item \emph{Coercivity.} It holds
	\begin{equation*}
	\mathcal{N}^{2}_{B}\lesssim\mathcal{W}\lesssim \mathcal{N}_{B}^{2}.
	\end{equation*}
\item {\emph{Monotonicity formula}.}
It holds
\begin{equation*}
  \frac{\dd \mathcal{W}}{\dd s}+\frac{\mu_{3}}{2B^{27}}\int_{\R^{2}}\left(|\nabla \varepsilon|^{2}+\varepsilon^{2}\right)\left(\varphi'_{B}+\psi_{0,B}\right)\dd y\le \frac{C_{14}}{|s|^{10}}.
\end{equation*}
Here, $\mu_{3}>0$ is a universal constant independent of $B$ and
$C_{14}=C_{14}(B,C_{1})>1$ is a constant dependent only on $B$ and $C_{1}$.
\end{enumerate}
\end{proposition}

\begin{proof}
    Proof of (i). First, from (i) of Lemma~\ref{le:viresteeta}, we find 
    \begin{equation}\label{est:P1}
      \left|  \mathcal{P}\right|\lesssim B^{12}\int_{\R^{2}}\left(
        |\nabla \varepsilon|^{2}+\varepsilon^{2}
        \right)\psi_{0,B}\dd y\lesssim B^{12}\mathcal{N}_{B}^{2}
        \Longrightarrow \mathcal{W}\lesssim \mathcal{N}_{B}^{2}.
    \end{equation}
    Second, from the definition of $\mathcal{L}$ and $\mathcal{F}$, we decompose
    \begin{equation*}
    \begin{aligned}
\mathcal{F}&=\left(\mathcal{L}\left(\varepsilon\sqrt{\psi_{B}}\right),\varepsilon\sqrt{\psi_{B}}\right)+\frac{1}{4}\int_{\R^{2}}\varepsilon^{2}\left(\frac{(\psi'_{B})^{2}}{\psi_{B}}-2\psi''_{B}\right)\dd y
        -\frac{1}{2}\int_{\R^{2}}\varepsilon^{4}\psi_{B}\dd y
        \\
&+\int_{\R^{2}}\varepsilon^{2}(\varphi_{B}-\psi_{B})\dd y-3\int_{\R^{2}}\left(W^{2}_{b}-Q^{2}\right)\varepsilon^{2}\psi_{B}\dd y-2\int_{\R^{2}}W_{b,\lambda}\varepsilon^{3}\psi_{B}\dd y.
\end{aligned}
\end{equation*}
Using Proposition~\ref{prop:Spectral}, Lemma~\ref{le:psiphi} and~\eqref{equ:ortho}, we deduce that 
\begin{equation}\label{est:Lvarepslion}
\left(\mathcal{L}\left(\varepsilon\sqrt{\psi_{B}}\right),\varepsilon\sqrt{\psi_{B}}\right)+\int_{\R^{2}}\varepsilon^{2}(\varphi_{B}-\psi_{B})\dd y\ge \frac{1}{2}\min(1,\mu_{1})\mathcal{N}_{B}^{2}.
\end{equation}
Then, from~\eqref{est:Boot}, Lemma~\ref{le:psiphi}, Lemma~\ref{le:2DSobolev}  and the 2D Sobolev embedding, we find 
\begin{equation}\label{est:nonlinear}
\begin{aligned}
  \left|\int_{\R^{2}}\left(W^{2}_{b}-Q^{2}\right)\varepsilon^{2}\psi_{B}\dd y\right|&\lesssim B^{-10}\mathcal{N}_{B}^{2},\\
    \left|\int_{\R^{2}}\varepsilon^{2}\left(\frac{(\psi'_{B})^{2}}{\psi_{B}}-2\psi''_{B}\right)\dd y\right|&\lesssim B^{-10}\mathcal{N}_{B}^{2},\\
    \left|\int_{\R^{2}}W_{b,\lambda}\varepsilon^{3}\psi_{B}\dd y\right|+
    \left|\int_{\R^{2}}\varepsilon^{4}\psi_{B}\dd y\right|&\lesssim B^{-10}\mathcal{N}_{B}^{2}.
    \end{aligned}
\end{equation}
Combining the estimates~\eqref{est:P1}--\eqref{est:nonlinear}, we complete the proof of (i).

\smallskip
Proof of (ii). From Propostion~\ref{prop:energy} and Propostion~\ref{prop:virial}, we dedeuce that 
\begin{equation*}
    \begin{aligned}
        &\frac{\dd \mathcal{W}}{\dd s}+\frac{1}{8}\int_{\R^{2}}\left(|\nabla\varepsilon|^{2}+\varepsilon^{2}\right)\varphi'_{B}\dd y+\frac{\mu_{2}}{B^{21}}\int_{\R^{2}}\left(|\nabla\eta|^{2}+\eta^{2}\right)\psi_{0,B}\dd y\\
        &\le \frac{C_{15}}{B^{5}}\int_{\R^{2}}\left(|\nabla\varepsilon|^{2}+\varepsilon^{2}\right)\varphi'_{B}\dd y+\frac{C_{16}}{B^{28}}\int_{\R^{2}}\left(|\nabla\varepsilon|^{2}+\varepsilon^{2}\right)\psi_{B}\dd y+\frac{C_{14}}{|s|^{10}}.
    \end{aligned}
\end{equation*}
Here, $C_{15}>1$ and $C_{16}>1$ are some universal constants independent of $B$ and $C_{14}=C_{14}(B,C_{1})>1$ is a constant depending only on $B$ and $C_{1}$.
Based on the above estimate and Lemma~\ref{le:viresteeta}, there exists a universal constant $\mu_{3}>0$ (independent of $B$) such that 
\begin{equation*}
    \begin{aligned}
         &\frac{\dd \mathcal{W}}{\dd s}+\frac{1}{9}\int_{\R^{2}}\left(|\nabla\varepsilon|^{2}+\varepsilon^{2}\right)\varphi'_{B}\dd y\\
         &+\frac{\mu_{3}}{B^{27}}\int_{\R^{2}}\left(|\nabla\varepsilon|^{2}+\varepsilon^{2}\right)\psi_{0,B}\dd y\le\frac{C_{16}}{B^{28}}\int_{\R^{2}}\left(|\nabla\varepsilon|^{2}+\varepsilon^{2}\right)\psi_{B}\dd y+\frac{C_{14}}{|s|^{10}}.
    \end{aligned}
\end{equation*}
Therefore, from $\psi_{B}\lesssim \psi_{0,B}+B\varphi'_{B}$, we obtain
\begin{equation*}
    \frac{\dd \mathcal{W}}{\dd s}+\int_{\R^{2}}\left(|\nabla \varepsilon|^{2}+\varepsilon^{2}\right)\left(\frac{1}{10}\varphi'_{B}+\frac{\mu_{3}}{2B^{27}}\psi_{0,B}\right)\dd y\le \frac{C_{14}}{|s|^{10}},
\end{equation*}
which completes the proof of (ii).
\end{proof}

We now prove Proposition~\ref{prop:Sn} by improving all the estimates in~\eqref{est:NBbound} and~\eqref{est:Boot}.

\begin{proof}[Proof of Proposition~\ref{prop:Sn}.]
We consider the solution $\omega(s,y)$ with the initial data at $s=S_{n}$ as defined in Section~\ref{SS:Geo}. Recall that, for the initial data, we have 
\begin{equation}\label{est:Sn1}
    \omega(S_{n})=W(S_{n})\Longrightarrow \varepsilon(S_{n})=0\ \ \mbox{and}\ \ 
    b(S_{n})=0.
\end{equation}
\textbf{Step 1.} Closing the estimate in $\varepsilon$. First, from~\eqref{equ:ortho},~\eqref{est:Sn1} and conservation of mass, for any $s\in [S_{n},S_{n}^{\star}]$, we find 
\begin{equation*}
    \|\varepsilon(s)\|_{L^{2}}^{2}=\|W(S_{n})\|_{L^{2}}^{2}-\|W_{b,\lambda}(s)\|_{L^{2}}^{2}-2\int_{\R^{2}}\left(W_{b,\lambda}(s)-Q\right)\varepsilon(s)\dd y.
\end{equation*}
Based on~\eqref{est:Boot}, Lemma~\ref{le:estW1}, Proposition~\ref{prop:Wb} and the above identity, we obtain 
\begin{equation*}
    \|\varepsilon(s)\|_{L^{2}}^{2}\lesssim \lambda^{\theta}(S_{n})+\lambda^{\theta}(s)+|b(s)|\lesssim |s|^{-1}+C_{1}^{\theta}|s|^{-\theta-1}\log^{\theta}|s|.
    \end{equation*}
    This strictly improves the $L^{2}$ estimate on $\varepsilon$ in~\eqref{est:Boot} for taking $C_{1}$ large enough.

    \smallskip
    Second, integrating the monotonicity formula in Proposition~\ref{Prop:Mon} on $[S_{n},s]$ for any $s\in [S_{n},S_{n}^{\star}]$ and then using (i) of Proposition~\ref{Prop:Mon}, we obtain 
    \begin{equation*}
        \mathcal{N}_{B}^{2}(s)\lesssim \mathcal{W}(s)\lesssim C_{14}\int_{S_{n}}^{s}|\tau|^{-10}\dd \tau\lesssim C_{14}|s|^{-9}
        \Longrightarrow 
        \mathcal{N}_{B}(s)\lesssim C^{\frac{1}{2}}_{14}|s|^{-\frac{9}{2}}.
    \end{equation*}
This strictly improves the $\mathcal{N}_{B}$ estimate on $\varepsilon$ in~\eqref{est:Boot} for taking $|s_{0}|$ large enough.

\smallskip
\textbf{Step 2.} Closing the estimate in $b$. Note that, from~\eqref{est:Boot} and~\eqref{est:Geo}, we find 
\begin{equation*}
\begin{aligned}
    \left|\frac{\dd }{\dd s}(sb)\right|&\lesssim 
    |s|\left|b_{s}-b\lambda^{\theta}\right|+|s||b|\left|\lambda^{\theta}-|s|^{-1}\right|\\
    &\lesssim  |s|^{-4}+C_{1}^{10}|s|^{-5}+C_{1}^{2}|s|^{-5}\log |s|.
    \end{aligned}
\end{equation*}
This strictly improves the estimate on $b$ in~\eqref{est:Boot} for taking $C_{1}$ and $|s_{0}|$ large enough.

\smallskip
\textbf{Step 3.} Closing the estimate in $\lambda$. Note that, from~\eqref{est:Boot} and the definition of $G(\lambda)$ in Section~\ref{SS:Boot}, for $|s_{0}|$ large enough, we find 
\begin{equation*}
    \frac{\dd G(\lambda(s))}{\dd s}=\lambda_{s}G'(\lambda)=-1+O\left(C_{1}|s|^{-3}\right)=-1+O(|s|^{-2}).
\end{equation*}
Therefore, from Remark~\ref{re:G} and $G(\lambda(S_{n}))=-S_{n}$, we obtain 
\begin{equation*}
    \left|G(\lambda(s))-|s|\right|\lesssim |s|^{-1}\Longrightarrow
    \lambda(s)=|s|^{-\frac{1}{\theta}}+O\left(|s|^{-1-\frac{1}{\theta}}\log |s|\right).
\end{equation*}
This strictly improves the estimate on $\lambda$ in~\eqref{est:Boot} for taking $C_{1}$ large enough.

\smallskip
\textbf{Step 4.} Closing the estimate in $x_{1}$. Note that, using again~\eqref{est:Boot}, ~\eqref{est:Geo} and the estimate in Step 3, we find  
\begin{equation*}
    \left|x_{1s}-|s|^{-\frac{1}{\theta}}\right|\lesssim
    |s|^{-\left(\frac{17}{4}+\frac{1}{\theta}\right)}+|s|^{-1-\frac{1}{\theta}}\log |s|+C_{1}^{10}|s|^{-\left(5+\frac{1}{\theta}\right)}.
\end{equation*}
Integrating the above estimate on $[S_{n},s]$ for any $s\in [S_{n},S_{n}^{\star}]$ and then using $x_{1}(S_{n})=-\frac{\theta}{\theta-1}|S_{n}|^{1-\frac{1}{\theta}}$, for $|s_{0}|$ large enough, we find 
\begin{equation*}
\begin{aligned}
    x_{1}(s)&=-\frac{\theta}{\theta-1}|S_{n}|^{1-\frac{1}{\theta}}+\int_{S_{n}}^{s}|\tau|^{-\frac{1}{\theta}}\dd \tau\\
    &+O\left(|s|^{-\frac{1}{\theta}}\log |s|\right)=-\frac{\theta}{\theta-1}|s|^{1-\frac{1}{\theta}}+O\left(|s|^{-\frac{1}{\theta}}\log |s|\right).
    \end{aligned}
\end{equation*}
This strictly improves the estimate on $x_{1}$ in~\eqref{est:Boot} for taking $C_{1}$ large enough.

\smallskip
\textbf{Step 5.} Conclusion. Note that, the estimates in~\eqref{est:NBbound} are standard consequences of the estimates in~\eqref{est:Boot}, and thus, we have strictly improved the estimates in~\eqref{est:NBbound} and~\eqref{est:Boot}. As a consequence of improving all the estimates in the
 bootstrap assumption~\eqref{est:NBbound} and~\eqref{est:Boot}, for any initial data $\phi_{n}$ satisfying~\eqref{equ:deflambdan}, we conclude that $S_{n}^{\star}=s_{0}$ for $|s_{0}|$ large enough. The proof of Proposition~\ref{prop:Sn} is complete.
\end{proof}

\section{Proof of Theorem~\ref{thm:main}}\label{S:Endproof}

From Proposition \ref{prop:Sn}, we know that there exists a sequence of solutions $(\phi_n)_{n\in \mathbb{N}}$ of~\eqref{CP} that satisfying the bootstrap estimates \eqref{est:Boot}. From now on, the implied constants in $\lesssim$ and $O$ can depend on the large constant $C_{1}$ in~\eqref{est:Boot}. Recall that, from~\eqref{est:Boot} and the definition of $s$, we have 
\begin{equation*}
    \frac{\dd t}{\dd s}=\lambda^3_n(s)=|s|^{-\frac{3}{\theta}}+O\left(|s|^{-1-\frac{3}{\theta}}\log |s|\right),\quad \mbox{for all}\  s\in[S_n,s_0].
\end{equation*}
Therefore, from the choice of initial time of $(S_{n},T_{n})=\left(-n,\theta/[(3-\theta)n^{\frac{3-\theta}{\theta}}]\right)$, 
\begin{equation*}
\begin{aligned}
    t(s)&=T_n+\int_{S_n}^s\left(|\tau|^{-\frac{3}{\theta}}+O\left(\frac{\log |\tau|}{|\tau|^{1+\frac{3}{\theta}}}\right)\right)\dd \tau\\
    &=\frac{\theta}{(3-\theta)|s|^{\frac{3-\theta}{\theta}}}+O\left(|s|^{-\frac{3}{\theta}}\log|s|\right).
    \end{aligned}
\end{equation*}
Then, we conclude that there exists $T_0>0$ independent of $n$ such that 
\begin{equation}\label{est:uniform}
\left\{
	\begin{aligned}
		&|b_{n}(t)|
		\lesssim t^{\frac{4\theta}{3-\theta}},\quad 
        \|\varepsilon_n(t)\|_{L^2}^2\lesssim t^{\frac{\theta}{3-\theta}},
        \\
		&\left|\lambda_n(t)-\left(\frac{3-\theta}{\theta}\right)^{\frac{1}{3-\theta}}t^{\frac{1}{3-\theta}}\right|
		\lesssim t^{\frac{1+\theta}{3-\theta}}\left|\log t\right|,\\
		&\left|x_{1n}(t)+\frac{\theta}{\theta-1}\left(\frac{\theta}{3-\theta}\right)^{\frac{\theta-1}{3-\theta}}t^{\frac{1-\theta}{3-\theta}}\right|\lesssim t^{\frac{1}{3-\theta}}\left|\log t\right|,
	\end{aligned}
	\right.\quad \mbox{for all}\ t\in [T_{n},T_{0}].
\end{equation}
Here, the implied constants in \eqref{est:uniform} are independent of $n$.

\smallskip
Using \eqref{est:Boot} and \eqref{est:Geo}, for all $t\in [T_{n},T_{0}]$, we deduce that 
\begin{equation}\label{est:t}
\left\{   \begin{aligned}
    &\left|\frac{\lambda_{ns}}{\lambda_{n}}+\Gamma_{n}+b_{n}\right|+\left|\frac{x_{1ns}}{\lambda_{n}}-1\right|\lesssim\lambda_{n}^{\frac{17}{4}\theta},\\
    &\left|b_{ns}-b_{n}\lambda_{n}^{\theta}\right|\lesssim \lambda_{n}^{5\theta},\ \ |b_{n}|\lesssim \lambda_{n}^{4\theta} \ \ \mbox{and}\ \  |x_{1n}(t)|\lesssim \lambda_{n}^{1-\theta}.
    \end{aligned} \right.
\end{equation}
Using again \eqref{est:Boot} and \eqref{est:Geo}, for all $t\in [T_{n},T_{0}]$, we deduce that 
\begin{equation}\label{est:t2}
      \int_{\mathbb{R}^2}
    \left(
    |\nabla \varepsilon_{n}(t)|^2+
    \varepsilon_{n}^{2}(t)
    \right)e^{-\frac{|y|}{10}}\dd y\lesssim \mathcal{N}^{2}_{B}(\varepsilon_{n})\lesssim\lambda_{n}^{8\theta}.
\end{equation}
In addition, for any $\rho>\theta-3$, we have 
\begin{equation}\label{est:lambda}
    \int_{T_n}^T\lambda_{n}^\rho(t)\dd t\lesssim \int_{T_n}^Tt^{\frac{\rho}{3-\theta}}\dd t\lesssim T^{1+\frac{\rho}{3-\theta}}\ll 1,\quad \mbox{for all}\ \ T\in [T_{n},T_{0}]. 
\end{equation}

\subsection{Uniform estimate of the $H^{\frac{7}{9}}$ norm.}\label{SS:Uniform}
In this subsection, we are devoted to providing a uniform estimate (with respect to $n$) for $\|\phi_n(t)\|_{H^{\frac{7}{9}}}$. Denote $U(t)=e^{-t\partial_{x_1}\Delta}$ by the semigroup relate to the 2D ZK equations. In addition, for all $T>0$, and $(p,q)\in[1,+\infty]^{2}$, we denote by
\begin{equation*}
\|f\|_{L_{x_1}^pL_{x_2T}^q}=\left\|\left\|f\right\|_{L_{x_2,t}^q(\mathbb{R}\times [0,T])}\right\|_{L_{x_1}^p(\mathbb{R})}\quad 
\mbox{and}\quad 
\|f\|_{L_T^qL_x^p}=\left\|\left\|f\right\|_{L_{x}^p(\mathbb{R}^2)}\right\|_{L_{t}^q([0,T])}.
\end{equation*}

\smallskip
We first recall the following linear estimates and smoothing effects related to $U(t)$.
\begin{lemma}\label{L1}
Let $0\le T\leq 1$. Then the following estimates hold. 
\begin{enumerate}
\item For all $\phi_0\in\mathcal{S}(\mathbb{R}^2)$, we have
\begin{align}
\|U(t)\phi_0\|_{L^2(\mathbb{R}^2)}&=\|\phi_0\|_{L^2(\mathbb{R}^2)},\label{41}\\
\|U(t)\phi_0\|_{L_T^{\frac{9}{4}}L_x^\infty}&\lesssim \min\left\{\|\phi_0\|_{L^2},\|D_{x_1}^{-\frac{2}{9}}\phi_0\|_{L^2}\right\}.\label{410}
\end{align}
\item  For all $\phi_0\in\mathcal{S}(\mathbb{R}^2)$ and $r>\frac{3}{4}$, we have
\begin{align}
\|U(t)\phi_0\|_{L^2_{x_1}L^\infty_{x_2T}}&\lesssim \|\phi_0\|_{H^{r}(\mathbb{R}^2)}\label{46},\\
\|\partial_{x_1} U(t)\phi_0\|_{L^\infty_{x_1}L^2_{x_2T}}&\lesssim \|\phi_0\|_{L^2(\mathbb{R}^2)}.\label{42}
\end{align}
\item For all regular and localized function $g$ on $[0,T]\times\mathbb{R}^2$, we have
\begin{align}
    \left\|\int_0^tU(t-\tau)g(\tau)\dd \tau\right\|_{L_{T}^{\infty}L^2_{x}}&\lesssim \|g\|_{L_{T}^1L_{x}^2},\quad \label{47}\\
\left\|\partial_{x_1}\int_0^t U(t-\tau)g(\tau)\dd \tau\right\|_{L_{T}^{\infty}L^2_{x}}&\lesssim \|g\|_{L_{x_1}^1L_{x_2T}^2} \label{44}.
\end{align}
\item For all regular and localized function $g$ on $[0,T]\times\mathbb{R}^2$, we have
\begin{equation}\label{412}
\begin{aligned}
    \left\|\int_0^t U(t-\tau)g(\tau)\dd \tau\right\|_{L_T^{\frac{9}{4}}L^\infty_{x}}&\lesssim \|g\|_{L_T^{1}L_x^2},\\
     \left\|\int_0^t U(t-\tau)g(\tau)\dd \tau\right\|_{L_T^{\frac{9}{4}}L^\infty_{x}}&\lesssim \|D_{x_1}^{-\frac{2}{9}}g\|_{L_T^1L_x^2},
\end{aligned}
\end{equation}
\begin{equation}\label{419}
\begin{aligned}
\left\|\partial_{x_{1}}\int_0^t U(t-\tau)g(\tau)\dd \tau\right\|_{L_T^{\frac{9}{4}}L^\infty_{x}}&\lesssim
\|\langle x_1\rangle g\|_{L_T^2L_x^2},\\
\left\|\partial_{x_{1}}\int_0^t U(t-\tau)g(\tau)\dd \tau\right\|_{L_T^{\frac{9}{4}}L^\infty_{x}}&\lesssim
\|\langle x_1\rangle D_{x_1}^{-\frac{2}{9}}g\|_{L_T^2L_x^2}.
\end{aligned}
\end{equation}
\item For all regular and localized function $g$ on $[0,T]\times\mathbb{R}^2$, we have
\begin{align}
    \left\|\int_0^tU(t-\tau)g(\tau)\dd \tau\right\|_{L^2_{x_1}L^\infty_{x_2T}}
    &\lesssim \|g\|_{L_T^1H_x^{\frac{7}{9}}},\label{43}\\
    \left\|\partial_{x_{1}}\int_0^tU(t-\tau)g(\tau)\dd \tau\right\|_{L^\infty_{x_1}L^2_{x_2T}}
    &\lesssim\|g\|_{L_T^1L_x^{2}},\label{417}\\
    \left\|
    \partial_{x_1}^2
    \int_0^tU(t-\tau)g(\tau)\dd \tau\right\|_{L^\infty_{x_1}L^2_{x_2T}}&\lesssim\|g\|_{L_{x_1}^1L_{x_2T}^2}.\label{413}
    \end{align}
    In addition, the following smoothing effect holds
    \begin{equation}\label{45}
         \left\|\partial_{x_1}\int_0^tU(t-\tau)g(\tau)\dd \tau\right\|_{L^2_{x_1}L^\infty_{x_2T}}\lesssim \left\|(1+D_{x_1}^{\frac{7}{9}}+D_{x_2}^{\frac{7}{9}})g\right\|_{L_{x_1}^1L_{x_2T}^2}.
    \end{equation}
\end{enumerate}
\end{lemma}
\begin{proof}
    For the sake of completeness and the readers’ convenience, the details of the proof for Lemma~\ref{L1} are given in Appendix \ref{App:linear}.
\end{proof}

\smallskip
Second, we recall the following fractional Leibniz's rule.
\begin{lemma}\label{le:Leibniz}
The following fractional Leibniz's rules hold.
\begin{enumerate}
  \item Let $(\alpha,p)\in (0,1)\times (1,\infty)$. For all $(f,g)\in\mathcal{S}(\mathbb{R})\times \mathcal{S}(\R)$, we have 
  \begin{equation}\label{est:1Dfrac}
       \|D_{x}^\alpha(fg)-fD_{x}^\alpha g\|_{L_{x}^p(\mathbb{R})}\lesssim \|g\|_{L_{x}^\infty(\mathbb{R})}\|D_{x}^\alpha f\|_{L^p_{x}(\mathbb{R})}.
  \end{equation}
  \item Let $(\alpha,p)\in (0,1)\times (1,\infty)$. For all $(f,g)\in\mathcal{S}(\mathbb{R}^2)\times \mathcal{S}(\R^{2})$, we have 
\begin{equation}\label{49}
\begin{aligned}
\|D_{x_1}^{\alpha} (fg)-fD_{x_1}^\alpha g\|_{L^p(\mathbb{R}^2)}&\lesssim \|g\|_{L^\infty(\mathbb{R}^2)}\|D_{x_1}^{\alpha} f\|_{L^p(\mathbb{R}^2)},\\
\|D_{x_2}^{\alpha} (fg)-fD_{x_2}^\alpha g\|_{L^p(\mathbb{R}^2)}&\lesssim \|g\|_{L^\infty(\mathbb{R}^2)}\|D_{x_2}^{\alpha} f\|_{L^p(\mathbb{R}^2)}.
\end{aligned}
\end{equation}
\end{enumerate}
\end{lemma}
\begin{proof}
Recall that, from~\cite[Theorem A.12]{KPV}, we directly have~\eqref{est:1Dfrac}. Then, using again~\cite[Theorem A.12]{KPV}, we also have
\begin{equation*}
\begin{aligned}
    \|D_{x_1}^\alpha(fg)-gD_{x_1}^\alpha f-fD_{x_1}^\alpha g\|_{L_{x_1}^p(\mathbb{R})}&\lesssim \|g\|_{L_{x_1}^\infty(\mathbb{R})}\|D_{x_1}^\alpha f\|_{L^p_{x_1}(\mathbb{R})},\quad \mbox{for all}\ x_{2}\in \R,\\
    \|D_{x_2}^\alpha(fg)-gD_{x_2}^\alpha f-fD_{x_2}^\alpha g\|_{L_{x_2}^p(\mathbb{R})}&\lesssim \|g\|_{L_{x_2}^\infty(\mathbb{R})}\|D_{x_2}^\alpha f\|_{L^p_{x_2}(\mathbb{R})},\quad \mbox{for all}\ x_{1}\in \R.
    \end{aligned}
\end{equation*}
Taking $L_{x_{2}}^{p}$ (or $L_{x_{1}}^{p}$) norm for the above estimates, we complete the proof.
\end{proof}

Third, we introduce the following estimates related to the space-time mixed norm.
\begin{lemma}\label{le:holder}
    For any regular functions $f$ and $g$ on $[0,T]\times \R^{2}$, we have 
\begin{equation*}
\begin{aligned}
\left\|fg\right\|_{L_{T}^{2}W_{x}^{1,\frac{18}{11}}}&\lesssim \|f\|_{L_{T}^{4}H_{x}^{1}}^{2}+\|g\|_{L_{T}^{4}H_{x}^{1}}^{2},\\
    \left\|fg\right\|_{L_{T}^{2}L^{2}_{x_{2}}L^{\frac{18}{13}}_{x_{1}}}&\lesssim \|f\|_{L_{T}^{4}H_{x}^{1}}^{2}+\|g\|_{L_{T}^{4}H_{x}^{1}}^{2}.
    \end{aligned}
\end{equation*}
\end{lemma}
\begin{proof}
    First, from the H\"older inequality and the 2D Sobolev embedding, we have 
    \begin{equation*}
        \begin{aligned}
           \left\|fg\right\|_{W_{x}^{1,\frac{18}{11}}}&\lesssim \|f\|_{L_{x}^{2}}\|g\|_{L_{x}^{9}}+
           \|\nabla f\|_{L_{x}^{2}}\|g\|_{L_{x}^{9}}\\
           &+\|f\|_{L_{x}^{9}}\|\nabla g\|_{L_{x}^{2}}\lesssim \|f\|_{H_{x}^{1}}^{2}+\|g\|_{H_{x}^{1}}^{2},
        \end{aligned}
    \end{equation*}
    which directly completes the proof for the first estimate.

    \smallskip
    Second, from the Hardy-Littlewood-Sobolev estimate and the Minkowski inequality, for any regular function $F$, we have
\begin{equation*}
\begin{aligned}
    \left\|D_{x_{1}}^{-\frac{1}{36}}D_{x_{2}}^{-\frac{1}{4}}F\right\|_{L_{x_{2}}^{4}L_{x_{1}}^{\frac{36}{17}}}
    &\lesssim 
     \left\|\frac{1}{|x_{2}|^{\frac{3}{4}}}\star_{2}\left(\frac{1}{|x_{1}|^{\frac{35}{36}}}\star_{1}F\right)\right\|_{L_{x_{2}}^{4}L_{x_{1}}^{\frac{36}{17}}}\\
     &\lesssim  \left\|\frac{1}{|x_{2}|^{\frac{3}{4}}}\star_{2}\left\|\left(\frac{1}{|x_{1}|^{\frac{35}{36}}}\star_{1}F\right)\right\|_{L_{x_{1}}^{\frac{36}{17}}}\right\|_{L_{x_{2}}^{4}}\lesssim \|F\|_{L_{x}^{2}}.
     \end{aligned}
\end{equation*}
It follows from the H\"older inequality that 
\begin{equation*}
\begin{aligned}
     \left\|fg\right\|_{L_{T}^{2}L^{2}_{x_{2}}L^{\frac{18}{13}}_{x_{1}}}
     &\lesssim \|f\|_{L_{T}^{4}L_{x}^{4}}\|g\|_{L_{T}^{4}L_{x_{2}}^{4}L_{x_{1}}^{\frac{36}{17}}}\\
     &\lesssim \|f\|_{L_{T}^{4}L_{x}^{4}}\|D_{x_{2}}^{\frac{1}{4}}D_{x_{1}}^{\frac{1}{36}}g\|_{L_{T}^{4}L_{x}^{2}}\lesssim \|f\|_{L_{T}^{4}H_{x}^{1}}^{2}+\|g\|_{L_{T}^{4}H_{x}^{1}}^{2},
     \end{aligned}
\end{equation*}
which directly completes the proof for the second estimate.
\end{proof}

From now on, for any $t\in [T_{n},T_{0}]$, we set 
\begin{equation*}
   X(t,x_1,x_2)=\frac{1}{\lambda_{n}(t)}Q\left(\frac{x_1-x_{1n}(t)}{\lambda_{n}(t)},\frac{x_2}{\lambda_{n}(t)}\right).
\end{equation*}
In addition, for any $t\in [T_{n},T_{0}]$, we also set 
\begin{equation*}
    w_n(t,x_1,x_2)=\frac{1}{\lambda_n(t)}\varepsilon_n\left(t,\frac{x_1-x_{1n}(t)}{\lambda_n(t)},\frac{x_2}{\lambda_n(t)}\right)\quad \mbox{and}\quad  Q_{S,n}=\phi_n-w_{n}.
\end{equation*}
We have the following uniform $H^{\frac{7}{9}}$ control for $w_{n}$.
\begin{proposition}\label{Prop:bdd}
Let $0<T_0<1$ be small enough $($independent of $n$$)$.
Then there exists a small universal constant $a=a(\theta)>0$, such that for all $n\in \mathbb{N}^{+}$ large enough, 
\begin{equation*}
    \|w_n(T)\|_{H^{\frac{7}{9}}}\le T^a,\quad \mbox{for any}\ T\in [T_{n},T_{0}].
\end{equation*}
\end{proposition}
\begin{proof}
We first introduce the following bootstrap estimate:
\begin{equation}\label{48}
    \begin{aligned}
\vertiii{w_n}_T\coloneqq&\sup_{t\in[T_n,T]}\|w_n(t)\|_{H^{\frac{7}{9}}}+\|w_n\|_{L^2_{x_1}L^\infty_{x_2T}}+\|w_n\|_{L^{\frac{9}{4}}_TL_x^\infty}\\
&+\|\partial_{x_1}w_n\|_{L^{\frac{9}{4}}_TL_x^\infty}+\sum_{j=1,2}\|
\partial_{x_1}
D_{x_j}^{\frac{7}{9}}w_n\|_{L^\infty_{x_1}L^2_{x_2T}}
\le 2T^a. 
    \end{aligned}
    \end{equation}
For $0<T_{0}\ll 1$ (independent with $n$), we define $T_{n}^{\star}\in [T_{n},T_{0}]$ by 
\begin{equation*}
    T_{n}^{\star}=\sup\left\{T\in [T_{n},T_{0}] \ \mbox{such that}~\eqref{48}  \ \mbox{holds on}\ [T_{n},T_{n}^{\star}]\right\}.
\end{equation*}
From $\varepsilon_{n}(T_{n})=w_{n}(T_{n})=0$, we directly have $T_{n}^{\star}>T_{n}$. In what follows, we will show that $T_{n}^{\star}=T_{0}$ for any large enough $n\in \mathbb{N}^{+}$.
For simplicity of notation, we drop the $n$ index of the function $(\lambda_{n},b_{n},x_{1n},\varepsilon_{n},w_{n})$ in the following part of the proof.

\smallskip
We denote 
\begin{equation*}
    \mathcal{E}_{w}=\partial_tQ_S+\partial_{x_1}(\Delta Q_S+Q_S^3)=\frac{1}{\lambda^4(t)}\mathcal{E}_b\left(t,\frac{x_1-x_{1}(t)}{\lambda(t)},\frac{x_2}{\lambda(t)}\right).
\end{equation*}
By an elementary computation, we decompose
\begin{align*}
w=\mathcal{K}_1+\mathcal{K}_2+\mathcal{K}_3+\mathcal{K}_4,
\end{align*}
where
\begin{equation*}
    \begin{aligned}
&\mathcal{K}_1=-\int_{T_n}^TU(t-\tau){\mathcal{E}}_w(\tau)\dd \tau,
\quad\mathcal{K}_2=-3\partial_{x_1}\int_{T_n}^TU(t-\tau)\left(wQ^{2}_{S}\right)(\tau)\dd \tau,\\
&\mathcal{K}_3=-3\partial_{x_1}\int_{T_n}^TU(t-\tau)\left(w^{2}Q_S\right)(\tau)\dd \tau,\quad\mathcal{K}_4=-\partial_{x_1}\int_{T_n}^TU(t-\tau)w^3(\tau)\dd \tau.
\end{aligned}
\end{equation*}

{\bf Step 1.} Estimate of $\mathcal{K}_1$. From the decomposition of $\mathcal{E}_{b}$ in Section~\ref{SS:Geo}, 
\begin{equation*}
\begin{aligned}
    \left\|\mathcal{E}_{\omega}\right\|_{H_{x}^{\frac{7}{9}}}
    &\lesssim \lambda^{-\frac{34}{9}}\left(\left|\frac{\lambda_s}{\lambda}+\Gamma+b\right|+\left|\frac{x_{1s}}{\lambda}-1\right|\right)\\
    &+\lambda^{-\frac{34}{9}}\left(\|\Psi_M\|_{H_{y}^{1}}+\|\Psi_N\|_{H_{y}^{1}}+\|\Psi_W\|_{H_{y}^1}+\|\Psi_b\|_{H_{y}^{1}}\right).
    \end{aligned}
\end{equation*}
It follows from~\eqref{est:t}, Proposition~\ref{prop:W}, Proposition~\ref{prop:Wb} that 
\begin{equation*}
     \left\|\mathcal{E}_{\omega}\right\|_{H_{x}^{\frac{7}{9}}}
    \lesssim \lambda^{-\frac{34}{9}}\left(\lambda^{\frac{17}{4}\theta}+|b|\lambda^{\theta-\frac{4}{5}}+|b|^{\frac{11}{8}}+\lambda^{\theta}|b|^{\frac{5}{8}}+\lambda^{\theta+\frac{4}{5}}+\lambda^{5\theta}\right)\lesssim \lambda^{\theta-\frac{134}{45}}.
\end{equation*}
Here, we use the fact that 
\begin{equation*}
 \begin{aligned}
    \left\|\Psi_{W}\right\|_{H_{y}^{1}}+\left\|\Psi_{b}\right\|_{H_{y}^{1}}
      &\lesssim \lambda^{5\theta}+\lambda^{\theta+\frac{4}{5}}+b|^{\frac{5}{8}}\lambda^{\theta}+|b|^{\frac{11}{8}},\\
     \left\|\Psi_{M}\right\|_{H_{y}^{1}}+\left\|\Psi_{N}\right\|_{H_{y}^{1}}
     &\lesssim \lambda^{\frac{17}{4}\theta}+|b|\lambda^{\theta+\frac{4}{5}}+|b|^{\frac{13}{8}}+|b|^{\frac{5}{8}}\lambda^{\theta}.
 \end{aligned}
\end{equation*}
Therefore, from~\eqref{est:lambda}, \eqref{47}, \eqref{412}, \eqref{43} and \eqref{417}, we have
\begin{equation}\label{est:K1}
\vertiii{\mathcal{K}_1}_T\lesssim 
 \left\|\mathcal{E}_{\omega}\right\|
_{L_{T}^1H_{x}^{\frac{7}{9}}}\lesssim\int_{T_n}^{T}
\lambda^{\theta-\frac{134}{45}}
\dd t\lesssim T^{\frac{1}{45(3-\theta)}}.
\end{equation}

\smallskip
{\textbf{Step 2.}} Estimate of $\mathcal{K}_2$. We claim that 
\begin{equation}\label{est:K2}
    \vertiii{\mathcal{K}_{2}}_{T}\lesssim 
     T^{\frac{5\theta-1}{2(3-\theta)}}+T^{\frac{63\theta-5}{18(3-\theta)}}+
    T^{a+\frac{\theta-1}{3-\theta}}+T^{a+\frac{9\theta-13}{18(3-\theta)}}+T^{a+\frac{45\theta-4}{30(3-\theta)}}.
\end{equation}

Indeed, using again~\eqref{44}, \eqref{419}, \eqref{413} and \eqref{45},
we have
\begin{equation*}
    \begin{aligned}
        \vertiii{\mathcal{K}_{2}}_{T}&
        \lesssim \sum_{k=0,1}\sum_{j=1,2}\big\|D_{x_j}^{\frac{7k}{9}}\left(w{Q}^{2}_{S}\right)\big\|_{L_{x_1}^1L_{x_2T}^2}\\
     &+\sum_{k=0,1}\big\|\langle x_1\rangle D_{x_1}^{\frac{7k}{9}}(wQ^{2}_{S})\big\|_{L_{T}^2L_{x}^2}.
    \end{aligned}
\end{equation*}
It follows from the AM-GM inequality that 
\begin{equation*}
\begin{aligned}
        \vertiii{\mathcal{K}_{2}}_{T}&
        \lesssim \sum_{k=0,1}\sum_{j=1,2}\big\|
        \langle x_{1}\rangle
        D_{x_j}^{\frac{7k}{9}}\left(wX(X-2Q_{S})\right)\big\|_{L_{T}^{2}L_{x}^2}\\
     &+\sum_{k=0,1}\sum_{j=1,2}\big\|\langle x_1\rangle D_{x_j}^{\frac{7k}{9}}(w(X-Q_{S})^{2})\big\|_{L_{T}^2L_{x}^2}\coloneqq \mathcal{K}_{2,1}+\mathcal{K}_{2,2}.
    \end{aligned}
    \end{equation*}
\emph{Estimate on $\mathcal{K}_{2,1}$.} By an elementary computation, we find 
\begin{equation*}
\begin{aligned}
    \mathcal{K}_{2,1}&
    \lesssim 
    \|
    \langle x_{1}\rangle
    wX(X-2Q_S)\|_{L_T^2H_x^1}\\
    &+\sum_{j=1,2}\big\|[x_1,D_{x_j}^{\frac{7}{9}}](wX(X-2Q_S))\big\|_{L^2_TL_x^2}
    \coloneqq\mathcal{K}_{2,1,1}+\mathcal{K}_{2,1,2}.
    \end{aligned}
\end{equation*}
Recall that, from the definition of space-time variable $(s,y)$ and~\eqref{est:t}, we have 
\begin{equation}\label{est:x1791}
    x_1=\lambda(t)y_1+x_1(t) \quad \mbox{and}\quad |x_1(t)|\lesssim \lambda^{1-\theta}(t).
\end{equation}
Based on the above estimate and \eqref{est:t2}, we find 
\begin{equation*}
\begin{aligned}
     &\|
    \langle x_{1}\rangle
    wX(X-2Q_S)\|_{H_x^1}^{2}\\
    &\lesssim\lambda^{-6}\int_{\R^{2}}\left(1+y_{1}^{2}\right)
    \left(|\nabla \varepsilon|^{2}+\varepsilon^{2}\right)e^{-\frac{|y|}{10}}\dd y\\
    &+\lambda^{-6}x_{1}^{2}\int_{\R^{2}}
    \left(|\nabla \varepsilon|^{2}+\varepsilon^{2}\right)e^{-\frac{|y|}{10}}\dd y\lesssim \lambda^{6\theta-4},
    \end{aligned}
\end{equation*}
which implies 
\begin{equation*}
    \mathcal{K}_{2,1,1}\lesssim \|
    \langle x_{1}\rangle
    wX(X-2Q_S)\|_{L_{T}^{2}H_x^1}\lesssim 
    \left(\int_{T_{n}}^{T}\lambda^{6\theta-4}(t)\dd t\right)^{\frac{1}{2}}\lesssim T^{\frac{5\theta-1}{2(3-\theta)}}.
\end{equation*}
Then, for any regular function $f$, it is easy to check that
\begin{equation}\label{est:D-29}
    [x_1,D_{x_1}^{\frac{7}{9}}]f=\frac{7i}{9}\mathcal{F}^{-1}_{\xi\to x}\left(\frac{\xi_{1}}{|\xi_{1}|^{\frac{11}{9}}}\widehat{f}(\xi)\right)
    \quad \mbox{and}\quad 
    [x_1,D_{x_2}^{\frac{7}{9}}]f=0.
\end{equation}
It follows from 1D Sobolev embedding that%
\footnote{Recall that, the Sobolev embedding $L^{\frac{18}{13}}(\R)
\hookrightarrow \dot{H}^{-\frac{2}{9}}(\mathbb{R})
$ holds true.} 
\begin{equation*}
    \begin{aligned}
       \mathcal{K}_{2,1,2}&\lesssim
       \big \|D_{x_1}^{-\frac{2}{9}}(wX(X-2Q_S))\big\|_{L_T^2L_x^2} \lesssim\big\|wX(X-2Q_S)\big\|_{L_T^2L_{x_2}^2L_{x_1}^{\frac{18}{13}}}.
    \end{aligned}
\end{equation*}
Based on the above estimate~\eqref{est:t2}, we find 
\begin{equation*}
    \big\|wX(X-2Q_S)\big\|^{2}_{L_{x_2}^2L_{x_1}^{\frac{18}{13}}}\lesssim \lambda^{-\frac{32}{9}}\int_{\R^{2}}\varepsilon^{2}e^{-\frac{|y|}{10}}\dd y\lesssim \lambda^{8\theta-\frac{32}{9}},
\end{equation*}
which implies 
\begin{equation*}
    \mathcal{K}_{2,1,2}\lesssim
    \big\|wX(X-2Q_S)\big\|_{L_T^2L_{x_2}^2L_{x_1}^{\frac{18}{13}}}\lesssim T^{\frac{63\theta-5}{18(3-\theta)}}.
\end{equation*}
Combining the above estimates, we obtain 
\begin{equation}\label{est:K21}
    \mathcal{K}_{2,1}\lesssim \mathcal{K}_{2,1,1}+\mathcal{K}_{2,1,2}\lesssim
    T^{\frac{5\theta-1}{2(3-\theta)}}+T^{\frac{63\theta-5}{18(3-\theta)}}.
\end{equation}

\emph{Estimate on $\mathcal{K}_{2,2}$.} 
We decompose 
\begin{equation*}
\begin{aligned}
    \mathcal{K}_{2,2}&\lesssim
    \|w(Q_S-X)^2\|_{L_T^2H_x^{\frac{7}{9}}}+\|x_1w(Q_S-X)^2\|_{L_T^2H_x^{\frac{7}{9}}}\\
    &+\sum_{j=1,2}\big\|[x_1,D_{x_j}^{\frac{7}{9}}](w(Q_S-X)^2)\big\|_{L_T^2L_x^2}\coloneqq\mathcal{K}_{2,2,1}+\mathcal{K}_{2,2,2}+\mathcal{K}_{2,2,3}.
    \end{aligned}
\end{equation*}
Note that, from Lemma~\ref{le:Leibniz}, we deduce that 
\begin{equation*}
\begin{aligned}
    &\left\|w(Q_{S}-X)^{2}\right\|_{H_{x}^{\frac{7}{9}}}+\|x_1w(Q_S-X)^2\|_{H_x^{\frac{7}{9}}}\\
    &\lesssim 
    \|w\|_{H_{x}^{\frac{7}{9}}}\|\langle x_{1}\rangle (Q_{S}-X)^{2}\|_{L_{x}^{\infty}}+\|w\|_{L_{x}^{9}}\big\|\langle x_{1}\rangle(Q_{S}-X)^{2}\big\|_{W_{x}^{1,\frac{18}{7}}}.
    \end{aligned}
\end{equation*}
It follows from~\eqref{48} and the H\"older inequality that 
\begin{equation*}
    \mathcal{K}_{2,2,1}+\mathcal{K}_{2,2,2}
    \lesssim T^{a}\left(
    \|\langle x_{1}\rangle (Q_{S}-X)^{2}\|_{L_{T}^{\infty}L_{x}^{\infty}}
    +\big\|\langle x_{1}\rangle(Q_{S}-X)^{2}\big\|_{L_{T}^{2}W_{x}^{1,\frac{18}{7}}}
    \right).
\end{equation*}
Using~\eqref{est:t}, we check that 
\begin{equation}\label{estLinftyL18}
\begin{aligned}
\|\langle x_{1}\rangle ({Q}_S-X)^{2}\|_{L_T^\infty L_x^\infty}&\lesssim 
\lambda^{-2}(T)x_{1}(T)\|W_{b,\lambda}-Q\|_{L_{T}^{\infty}L_{y}^{\infty}}\lesssim 
T^{\frac{\theta-1}{3-\theta}}.
\end{aligned}
\end{equation}
Moreover, we have
\begin{equation*}
\begin{aligned}
   & \big\|\langle x_{1}\rangle(Q_{S}-X)^{2}\big\|_{W_{x}^{1,\frac{18}{7}}}^{2}
   \\
   &\lesssim \lambda^{-\frac{40}{9}}\left(\int_{\R^{2}}\left(1+\lambda^{\frac{18}{7}}y_{1}^{\frac{18}{7}}+x_{1}^{\frac{18}{7}}\right)\left(|\nabla V|^{\frac{36}{7}}+b^{\frac{36}{7}}|\nabla P_{b}|^{\frac{36}{7}}\right)\dd y
    \right)^{\frac{7}{9}}\\
   &+\lambda^{-\frac{22}{9}}\left(
    \int_{\R^{2}}\left(1+\lambda^{\frac{18}{7}}y_{1}^{\frac{18}{7}}+x_{1}^{\frac{18}{7}}\right)\left(V^{\frac{36}{7}}+b^{\frac{36}{7}}P_{b}^{\frac{36}{7}}\right)\dd y
    \right)^{\frac{7}{9}}\lesssim \lambda^{2\theta-\frac{40}{9}}.
    \end{aligned}
\end{equation*}
It follows from~\eqref{est:lambda} that 
\begin{equation*}
    \big\|\langle x_{1}\rangle(Q_{S}-X)^{2}\big\|_{L_{T}^{2}W_{x}^{1,\frac{18}{7}}}\lesssim \left(\int_{T_{n}}^{T}\lambda^{2\theta-\frac{40}{9}}(t)\dd t \right)^{\frac{1}{2}}\lesssim T^{\frac{9\theta-13}{18(3-\theta)}}.
\end{equation*}
Based on the above estimates, we deduce that 
\begin{equation}\label{est:K221K222}
    \mathcal{K}_{2,2,1}+\mathcal{K}_{2,2,2}\lesssim T^{a+\frac{\theta-1}{3-\theta}}+T^{a+\frac{9\theta-13}{18(3-\theta)}}.
\end{equation}
Then, using again~\eqref{48},~\eqref{est:D-29} and the Minkowski inequality, we directly have 
\begin{equation*}
\begin{aligned}
    \mathcal{K}_{2,2,3}\lesssim \left\|w(Q_S-X)^2\right\|_{L_{x_1}^{\frac{18}{13}}L_{x_2T}^2}&\lesssim
    \|w\|_{L_{x_1}^2L_{x_2T}^{\infty}}\|Q_S-X\|_{L_{x_1}^6L_{x_2T}^{6}}\|Q_S-X\|_{L_{x_1}^{18}L_{x_2T}^3}\\
    &\lesssim T^{a}\|Q_S-X\|_{L_{T}^6L_{x}^{6}}\|Q_S-X\|_{L_{T}^{3}L_{x_{1}}^{18}L^{3}_{x_{2}}}.
\end{aligned}
\end{equation*}
Recall that, using again~\eqref{est:t}, we find 
\begin{equation*}
\begin{aligned}
    \|Q_S-X\|^6_{L^6_TL_{x}^{6}}
    &\lesssim 
    \int_{T_n}^T \lambda^{-4}(t)\|V\|_{L^{6}_{y}}^6\dd t\\
    &+\int_{T_n}^T \lambda^{-4}(t)\|bP_{b}\|_{L_{y}^{6}}^6\dd t\lesssim T^{\frac{25\theta-13}{5(3-\theta)}},
    \end{aligned}
\end{equation*}
\begin{equation*}
\begin{aligned}
    \|Q_S-X\|^3_{L^3_TL_{x_1}^{18}L_{x_2}^3}
    &\lesssim 
    \int_{T_n}^T \lambda^{-\frac{11}{6}}(t)\|V\|_{L_{y_1}^{18}L_{y_2}^3}^3\dd t\\
    &+\int_{T_n}^T \lambda^{-\frac{11}{6}}(t)\|bP_{b}\|_{L_{y_1}^{18}L_{y_2}^3}^3\dd t\lesssim T^{\frac{20\theta+9}{10(3-\theta)}},
    \end{aligned}
\end{equation*}
which implies
\begin{equation*}
    \mathcal{K}_{2,2,3}\lesssim T^{a}\|Q_S-X\|_{L_{T}^6L_{x}^{6}}\|Q_S-X\|_{L_{T}^{3}L_{x_{1}}^{18}L^{3}_{x_{2}}}
    \lesssim T^{a+\frac{45\theta-4}{30(3-\theta)}}.
\end{equation*}
Combining the above estimate with~\eqref{est:K221K222}, we obtain
\begin{equation}\label{est:K22}
    \mathcal{K}_{2,2}\lesssim 
    \mathcal{K}_{2,2,1}+\mathcal{K}_{2,2,2}+\mathcal{K}_{2,2,3}
    \lesssim
    T^{a+\frac{\theta-1}{3-\theta}}+T^{a+\frac{9\theta-13}{18(3-\theta)}}+T^{a+\frac{45\theta-4}{30(3-\theta)}}.
\end{equation}
We see that~\eqref{est:K2} follows from~\eqref{est:K21} and~\eqref{est:K22}.

\smallskip
{\textbf{Step 3}.} Estimate of $\mathcal{K}_3$. 
We claim that 
\begin{equation}\label{est:K3}
\begin{aligned}
    \vertiii{\mathcal{K}_{3}}_{T}
    &\lesssim 
    T^{\frac{11\theta+1}{4(3-\theta)}}+T^{2a+\frac{\theta-1}{3-\theta}}+T^{2a+\frac{15\theta-23}{20(3-\theta)}}\\
    &+T^{2a+\frac{17-5\theta}{120(3-\theta)}}+T^{2a+\frac{20\theta+9}{30(3-\theta)}}.
    \end{aligned}
\end{equation}
Indeed, using again~\eqref{44}, \eqref{419}, \eqref{413} and \eqref{45},
 we have
 \begin{equation*}
 \begin{aligned}
     \vertiii{\mathcal{K}_3}_T&\lesssim
     \sum_{k=0,1}\sum_{j=1,2}\big\|D_{x_j}^{\frac{7k}{9}}\left(w^{2}{Q}_{S}\right)\big\|_{L_{x_1}^1L_{x_2T}^2}\\
     &+\sum_{k=0,1}\big\|\langle x_1\rangle D_{x_1}^{\frac{7k}{9}}(w^{2}Q_{S})\big\|_{L_{T}^2L_{x}^2}.
     \end{aligned}
 \end{equation*}
 It follows from the AM-GM inequality that 
 \begin{align*}
\vertiii{\mathcal{K}_3}_T
\lesssim &\sum_{k=0,1}\sum_{j=1,2}\big\|\langle x_1\rangle D_{x_j}^{\frac{7k}{9}}\big(w^2X\big)\big\|_{L_{T}^2L_x^2}+\big\|D_{x_2}^{\frac{7}{9}}\big(w^2\big(Q_S-X\big)\big)\big\|_{L_{x_1}^1L_{x_2T}^2}\\
&
+\sum_{k=0,1}\big\|\langle x_1\rangle D_{x_1}^{\frac{7k}{9}}
\big(w^2\big(Q_S-X\big)\big)\big\|_{L_{T}^2L_{x}^2}\coloneqq\mathcal{K}_{3,1}+\mathcal{K}_{3,2}+\mathcal{K}_{3,3}.
\end{align*}
\emph{Estimate on $\mathcal{K}_{3,1}$.} First, from the Sobolev embedding, we have%
\footnote{Recall that, the Sobolev embedding $W^{1,\frac{18}{11}}(\mathbb{R}^2)\hookrightarrow H^{\frac{7}{9}}(\mathbb{R}^2)$ holds true.}
\begin{equation}\label{51}
\begin{aligned}
\mathcal{K}_{3,1}&\lesssim
\big\|w^2X\big\|_{L_T^2H_x^{\frac{7}{9}}}+
\big\|x_1w^2X\big\|_{L_T^2H_x^{\frac{7}{9}}}+
\sum_{j=1}^2\big\|[x_1,D_{x_j}^{\frac{7}{9}}]\big(w^2X\big)\big\|_{L_T^2L_x^2}\\
&\lesssim 
\big\|w^2X\big\|_{L_T^2W_x^{1,\frac{18}{11}}}+
\big\|x_1w^2X\big\|_{L_T^2W_x^{1,\frac{18}{11}}}
+\sum_{j=1}^2\big\|[x_1,D_{x_j}^{\frac{7}{9}}](w^2X)\big\|_{L_T^2L_x^2}.
\end{aligned}
\end{equation}
From Lemma~\ref{le:holder}, we directly have 
\begin{equation*}
    \big\|w^2X\big\|_{L_T^2W_x^{1,\frac{18}{11}}}+
\big\|x_1w^2X\big\|_{L_T^2W_x^{1,\frac{18}{11}}}\lesssim
\big\|wX^{\frac{1}{2}}\big\|_{L_T^4H_{x}^{1}}^{2}
+\big\|x_{1}wX^{\frac{1}{2}}\big\|_{L_T^4H_{x}^{1}}^{2}.
\end{equation*}
Recall that, for any regular function $f$, it is easy to check that
\begin{equation*}
    [x_1,D_{x_1}^{\frac{7}{9}}]f=\frac{7i}{9}\mathcal{F}^{-1}_{\xi\to x}\left(\frac{\xi_{1}}{|\xi_{1}|^{\frac{11}{9}}}\widehat{f}(\xi)\right)
    \quad \mbox{and}\quad 
    [x_1,D_{x_2}^{\frac{7}{9}}]f=0.
\end{equation*}
It follows from 1D Sobolev embedding and Lemma~\ref{le:holder} that%
\footnote{Recall again that, the Sobolev embedding $L^{\frac{18}{13}}(\R)
\hookrightarrow \dot{H}^{-\frac{2}{9}}(\mathbb{R})
$ holds true.} 
\begin{equation}\label{54}
    \begin{aligned}
        \sum_{j=1}^2\big\|[x_1,D_{x_j}^{\frac{7}{9}}](w^2X)\big\|_{L_T^2L_x^2}
      &\lesssim
      \big\|D_{x_1}^{-{\frac{2}{9}}}(w^2X)\big\|_{L_T^{2}L_x^2}\\
&\lesssim\big\|w^2X\|_{L_T^2L_{x_2}^2L_{x_1}^{\frac{18}{13}}}\lesssim \big\| wX^{\frac{1}{2}}\big\|_{L_T^4H_x^1}^2.
    \end{aligned}
\end{equation}
Combining the above estimates, we deduce that 
\begin{equation}\label{est:R11}
    \mathcal{K}_{3,1}\lesssim \big\|wX^{\frac{1}{2}}\big\|_{L_T^4H_{x}^{1}}^{2}
    +\big\|x_{1}wX^{\frac{1}{2}}\big\|_{L_T^4H_{x}^{1}}^{2}.
\end{equation}
Recall again that, from the definition of space-time variable $(s,y)$ and~\eqref{est:t}, we have 
\begin{equation}\label{est:x179}
    x_1=\lambda(t)y_1+x_1(t) \quad \mbox{and}\quad |x_1(t)|\lesssim \lambda^{1-\theta}(t).
\end{equation}
It follows from \eqref{est:t} that 
\begin{equation*}
\begin{aligned}
    \big\|wW^{\frac{1}{2}}\big\|^{2}_{H_{x}^{1}}
    &\lesssim \lambda^{-3}(t)\int_{\R^{2}}\left(1+x^{2}_{1}(t)\right)(|\nabla \varepsilon(t)|^2+\varepsilon^{2}(t))e^{-\frac{|y|}{10}}\dd y\\
    &\lesssim \lambda^{-3}(t)\left(\lambda^{8\theta}(t)+\lambda^{6\theta+2}(t)\right)\lesssim 
    \lambda^{8\theta-3}(t)+
    \lambda^{6\theta-1}(t),
    \end{aligned}
\end{equation*}
\begin{equation*}
\begin{aligned}
    \big\|x_{1}wW^{\frac{1}{2}}\big\|^{2}_{H_{x}^{1}}
    &\lesssim \lambda^{-3}(t)\int_{\R^{2}}\left(1+x^{2}_{1}(t)\right)(|\nabla \varepsilon(t)|^2+\varepsilon^{2}(t))e^{-\frac{|y|}{10}}\dd y\\
    &\lesssim \lambda^{-3}(t)\left(\lambda^{8\theta}(t)+\lambda^{6\theta+2}(t)\right)\lesssim 
    \lambda^{8\theta-3}(t)+
    \lambda^{6\theta-1}(t). \qquad 
    \end{aligned}
\end{equation*}
Combining the above estimates with~\eqref{est:lambda} and~\eqref{est:R11}, we obtain 
\begin{equation}\label{est:R1}
    \mathcal{K}_{3,1}\lesssim \left(\int_{T_{n}}^{T}
    \left(\lambda^{16\theta-6}(t)+
    \lambda^{12\theta-2}(t)\right)\dd t\right)^{\frac{1}{4}}\lesssim 
    T^{\frac{11\theta+1}{4(3-\theta)}}.
\end{equation}

\emph{Estimate on $\mathcal{K}_{3,2}$.} From~\eqref{est:1Dfrac}, for all $(x_{1},T)\in \mathbb{R}\times[T_{n},T_{n}^{\star}]$, we have 
\begin{equation*}
\begin{aligned}
    \left\|D_{x_2}^{\frac{7}{9}}(w^2(Q_S-X))\right\|_{L_{x_2}^2}
    &\lesssim \|Q_{S}-X\|_{L_{x_{2}}^{\infty}}\|w\|_{L_{x_{2}}^{\infty}}\|D_{x_{2}}^{\frac{7}{9}}w\|_{L_{x_{2}}^{2}}\\
    &+\|w\|_{L_{x_{2}}^{\infty}}\|w\|_{L_{x_{2}}^{4}}\|D_{x_{2}}^{\frac{7}{9}}(Q_{S}-X)\|_{L_{x_{2}}^{4}},
    \end{aligned}
\end{equation*}
which implies 
\begin{equation*}
    \begin{aligned}
        \left\|D_{x_2}^{\frac{7}{9}}(w^2(Q_S-X))\right\|_{L_{x_2T}^2}
         &\lesssim \|Q_{S}-X\|_{L_{x_{2}T}^{\infty}}\|w\|_{L_{x_{2}T}^{\infty}}\|D_{x_{2}}^{\frac{7}{9}}w\|_{L_{x_{2}T}^{2}}\\
    &+\|w\|_{L_{x_{2}T}^{\infty}}\|w\|_{L_{x_{2}T}^{4}}\|D_{x_{2}}^{\frac{7}{9}}(Q_{S}-X)\|_{L_{x_{2}T}^{4}}.
    \end{aligned}
\end{equation*}
Taking $L_{x_1}^1$ norm on both sides of the above inequality and then using~\eqref{48},
\begin{equation*}
\begin{aligned}
    \mathcal{K}_{3,2}&\lesssim \|Q_{S}-X\|_{L_{T}^{\infty}L_{x}^{\infty}}\|w\|_{L_{x_{1}}^{2}L_{x_{2}T}^{\infty}}\|D_{x_{2}}^{\frac{7}{9}}w\|_{L_{T}^{2}L_{x}^{2}}\\
&+\|w\|_{L_{x_{1}}^{2}L_{x_{2}T}^{\infty}}\|w\|_{L_{T}^{4}H_{x}^{1}}\|D_{x_{2}}^{\frac{7}{9}}(Q_{S}-X)\|_{L_{T}^{4}L_{x}^{4}}\\
&\lesssim T^{2a}
\|Q_{S}-X\|_{L_{T}^{\infty}L_{x}^{\infty}}+T^{2a}\|(Q_{S}-X)\|_{L_{T}^{4}W_{x}^{1,4}}.
    \end{aligned}
\end{equation*}
Using again~\eqref{est:t}, we check that 
\begin{equation}\label{53}
\begin{aligned}
\|{Q}_S-X\|_{L_T^\infty L_x^\infty}&\lesssim 
\lambda^{-1}(T)\|W_{b,\lambda}-Q\|_{L_{T}^{\infty}L_{y}^{\infty}}\lesssim 
T^{\frac{\theta-1}{3-\theta}},\\
\|{Q}_S-X\|^{4}_{L_T^4 W_x^{1,4}}&\lesssim\int_{T_n}^T\lambda^{-6}(t)\|W_{b,\lambda}-Q\|^4_{W_y^{1,4}}\dd t\lesssim T^{\frac{15\theta-23}{5(3-\theta)}}.
\end{aligned}
\end{equation}
Here, we use the fact that 
\begin{equation*}
    \|W_{b,\lambda}-Q\|_{L_{T}^{\infty}L_{y}^{\infty}}\lesssim \lambda^{\theta}+|b|\quad \mbox{and}\quad \|W_{b,\lambda}-Q\|^4_{W_y^{1,4}}\lesssim \lambda^{4\theta-\frac{8}{5}}+|b|^{\frac{13}{4}}.
\end{equation*}
Combining the above estimates, we obtain
\begin{equation}\label{est:R2}
    \mathcal{K}_{3,2}\lesssim T^{2a+\frac{\theta-1}{3-\theta}}+T^{2a+\frac{15\theta-23}{20(3-\theta)}}.
\end{equation}

\emph{Estimate on $\mathcal{K}_{3,3}$.}
We decompose
\begin{align*}
\mathcal{K}_{3,3}&\lesssim \|(1+|x_1|)w^2(Q_S-X)\|_{L_T^2L_x^{2}}+\left\|\partial_{x_1}\left(x_1w^2(Q_S-X)\right)\right\|_{L_T^2L_x^{2}}\\
&\quad+\big\|[x_1,D_{x_1}^{\frac{7}{9}}](w^2(Q_S-X))\big\|_{L_T^2L_x^2}\coloneqq \mathcal{K}_{3,3,1}+\mathcal{K}_{3,3,2}+\mathcal{K}_{3,3,3}.
\end{align*}
First, using the H\"older inequality and~\eqref{48}, 
\begin{equation*}
\begin{aligned}
\mathcal{K}_{3,3,1}&\lesssim \|w\|^2_{L_T^8L_x^8}\|\langle x_{1}\rangle(Q_S-X)\|_{L_T^4L_x^4}\lesssim T^{2a}\|\langle x_{1}\rangle(Q_S-X)\|_{L_T^4L_x^4}.
\end{aligned}
\end{equation*}
It follows from~\eqref{est:x179} and $\theta>\frac{8}{5}$ that 
\begin{align*}
&\| \langle x_{1}\rangle (Q_S-X)\|_{L_T^4L_{x}^4}^4\\
&\lesssim \int_{T_n}^T{\lambda^{-2}(t)}\left(\int_{\mathbb{R}^2}(1+x^{4}_1(t)+\lambda^{4}(t)y_{1}^{4})(V^{4}+b^{4}P_{b}^{4})\dd y\right)\dd t\\
&\lesssim \int_{T_n}^{T}{\lambda^{-2}(t)}
\left(\lambda^{\frac{12}{5}}(t)+\lambda(t)^{4-4\theta}|b(t)|^{\frac{13}{4}}+\lambda^{4}(t)|b(t)|^{\frac{1}{4}}\right)\dd t\lesssim T^{1+\frac{2}{5(3-\theta)}}.
\end{align*}
Combining the above estimates, we obtain 
\begin{equation*}
    \mathcal{K}_{3,3,1} \lesssim T^{2a}\|(1+|x_1|)(Q_S-X)\|_{L_T^4L_x^4}\lesssim T^{2a+\frac{17-5\theta}{20(3-\theta)}}.
\end{equation*}

Second, using again the H\"older inequality and~\eqref{48},
\begin{equation*}
    \begin{aligned}
        \mathcal{K}_{3,3,2}&\lesssim\|w\|^2_{L_T^\infty L_x^9}\left\|\partial_{x_1}(x_1(Q_S-X))\right\|_{L_T^2L_x^{\frac{18}{5}}}\\
        &+ \|\partial_{x_{1}}w\|_{L_T^{\frac{9}{4}}L_x^\infty}\|w\|_{L_T^\infty L_x^{\frac{9}{4}}}\|x_1(Q_S-X)\|_{L_T^{18}L_x^{18}}\\
&\lesssim T^{2a}
\|\partial_{x_1}(x_1(Q_S-X))\|_{L_T^2L_x^{\frac{18}{5}}}+ T^{2a}\|x_1(Q_S-X)\|_{L_T^{18}L_x^{18}}.
    \end{aligned}
\end{equation*}
Using again~\eqref{est:x179} and $\frac{8}{5}<\theta<\frac{19}{9}$, we have
\begin{equation*}
\begin{aligned}
    &\|\partial_{x_1}(x_1(Q_S-X))\|^{2}_{L_T^2L_x^{\frac{18}{5}}}\\
    &\lesssim\int_{T_n}^T {\lambda^{-\frac{8}{9}}(t)}\left(\int_{\mathbb{R}^2}\big(V^{\frac{18}{5}}+b^{\frac{18}{5}}P_{b}^{\frac{18}{5}}\big)\dd y\right)^{\frac{5}{9}}\dd t\\
    &+\int_{T_n}^T{\lambda^{-\frac{26}{9}}(t)}\left(\int_{\mathbb{R}^2}
    \big(x^{\frac{18}{5}}_1(t)+\lambda^{\frac{18}{5}}(t)y_1^{\frac{18}{5}}\big)
    ((\partial_{x_{1}}V)^{\frac{18}{5}}
    \dd y\right)^{\frac{5}{9}}\dd t\\
    &+\int_{T_n}^T{\lambda^{-\frac{26}{9}}(t)}b^{2}(t)\left(\int_{\mathbb{R}^2}
    \big(x^{\frac{18}{5}}_1(t)+\lambda^{\frac{18}{5}}(t)y_1^{\frac{18}{5}}\big)
    (\partial_{x_{1}}P_{b})^{\frac{18}{5}})
    \dd y\right)^{\frac{5}{9}}\dd t,
    \end{aligned}
\end{equation*}
which implies 
\begin{equation*}
    \|\partial_{x_1}(x_1(Q_S-X))\|^{2}_{L_T^2L_x^{\frac{18}{5}}}
    \lesssim \int_{T_{n}}^{T}\left(\lambda^{2\theta-\frac{16}{9}}(t)+\lambda^{-\frac{8}{9}}(t)\right)\dd t\lesssim T^{\frac{19-9\theta}{9(3-\theta)}}.
\end{equation*}
Similarly, we have 
\begin{equation*}
\begin{aligned}
\| x_{1}(Q_S-X)\|_{L_T^{18}L_{x}^{18}}^{18}&\lesssim \int_{T_n}^T{\lambda^{-16}(t)}\left(\int_{\mathbb{R}^2}(x^{18}_1(t)+\lambda^{18}(t)y_{1}^{18})(V^{18}+b^{18}P_{b}^{18})\dd y\right)\dd t\\
&\lesssim \int_{T_n}^{T}
\left(\lambda^{\frac{2}{5}}(t)+\lambda^{2-18\theta}(t)|b(t)|^{\frac{69}{4}}+\lambda^{2}(t)|b(t)|^{\frac{15}{4}}\right)\dd t\lesssim T^{\frac{17-5\theta}{5(3-\theta)}}.
\end{aligned}
\end{equation*}
Combining the above estimates, we obtain 
\begin{equation*}
    \mathcal{K}_{3,3,2}\lesssim T^{2a}
\|\partial_{x_1}(x_1(Q_S-X))\|_{L_T^2L_x^{\frac{18}{5}}}+ T^{2a}\|x_1(Q_S-X)\|_{L_T^{18}L_x^{18}}\lesssim T^{2a+\frac{17-5\theta}{120(3-\theta)}}.
\end{equation*}

Using again the Minkowski inequality, the H\"older inequality and~\eqref{48}, we find 
\begin{equation*}
\begin{aligned}
    \mathcal{K}_{3,3,3}
    &\lesssim \|w\|_{L_{T}^6L_{x}^{6}}\|w\|_{L_{x_1}^2L_{x_2T}^{\infty}}\|Q_S-X\|_{L_{x_1}^{18}L_{x_2T}^3}\\
    &\lesssim \vertiii{w}_{T}^{2}\|Q_S-X\|_{L_{T}^{3}L_{x_1}^{18}L_{x_2}^3}
    \lesssim T^{2a}\|Q_S-X\|_{L_{T}^{3}L_{x_1}^{18}L_{x_2}^3}.
    \end{aligned}
\end{equation*}
Recall that, using again~\eqref{est:t}, we directly have 
\begin{equation}\label{52}
\begin{aligned}
    \|Q_S-X\|^3_{L^3_TL_{x_1}^{18}L_{x_2}^3}
    &\lesssim 
    \int_{T_n}^T \lambda^{-\frac{11}{6}}(t)\|V\|_{L_{y_1}^{18}L_{y_2}^3}^3\dd t\\
    &+\int_{T_n}^T \lambda^{-\frac{11}{6}}(t)\|bP_{b}\|_{L_{y_1}^{18}L_{y_2}^3}^3\dd t\lesssim T^{\frac{20\theta+9}{10(3-\theta)}},
    \end{aligned}
\end{equation}
which implies 
\begin{equation*}
    \mathcal{K}_{3,3,3}\lesssim
    T^{2a}\|Q_S-X\|_{L_{T}^{3}L_{x_1}^{18}L_{x_2}^3}\lesssim T^{2a+\frac{20\theta+9}{30(3-\theta)}}.
\end{equation*}
Combining the above estimates, we obtain 
\begin{equation}\label{est:R3}
    \mathcal{K}_{3,3}\lesssim \mathcal{K}_{3,3,1}+\mathcal{K}_{3,3,2}+\mathcal{K}_{3,3,3}\lesssim
    T^{2a+\frac{17-5\theta}{5(3-\theta)}}+T^{2a+\frac{20\theta+9}{30(3-\theta)}}.
\end{equation}

We see that~\eqref{est:K3} follows from~\eqref{est:R1},~\eqref{est:R2} and~\eqref{est:R3}.

\smallskip
{\textbf{Step 4.}} Estimate of $\mathcal{K}_4$. We claim that 
\begin{equation}\label{est:K4}
    \vertiii{\mathcal{K}_{4}}_{T}\lesssim T^{3a}.
\end{equation}
Indeed, from the H\"older inequality and~\eqref{48}, we directly have 
\begin{equation*}
    \|\partial_{x_{1}}(w^{3})\|_{L_{T}^{1}L_{x}^{2}}\lesssim \|\partial_{x_{1}}w
    \|_{L_T^{\frac{9}{4}}L_x^\infty}\|w\|_{L_T^{\frac{9}{4}}L_x^\infty}\|w\|_{L_T^{9}L_x^2}\lesssim \vertiii{w}_T^3\lesssim T^{3a}.
\end{equation*}
Then, using again Lemma~\ref{le:Leibniz}, we deduce that 
\begin{equation*}
\begin{aligned}
     \sum_{j=1,2}\big\|D_{x_j}^{\frac{7}{9}}\partial_{x_1}(w^3)\big\|_{L^2_x}
    &\lesssim \|\partial_{x_{1}}w\|_{L_{x}^{\infty}}\|w\|_{L_{x}^{\infty}}
    \big\|D_{x_{j}}^{\frac{7}{9}}w\big\|_{L_{x}^{2}}\\
    &+\|w\|_{L_{x}^{\infty}}\big\|wD_{x_{j}}^{\frac{7}{9}}\partial_{x_{1}}w\big\|_{L_{x}^{2}}.
    \end{aligned}
\end{equation*}
It follows from~\eqref{48} that 
\begin{equation*}
    \begin{aligned}
         \sum_{j=1,2}\big\|D_{x_j}^{\frac{7}{9}}\partial_{x_1}(w^3)\big\|_{L_{T}^{1}L^2_x}&\lesssim
         \sum_{j=1,2}\|\partial_{x_{1}}w\|_{L_T^{\frac{9}{4}}L_x^\infty}\|w\|_{L_T^{\frac{9}{4}}L_x^\infty}
         \big\|D_{x_j}^{\frac{7}{9}}w\big\|_{L_T^{9}L_x^2}\\
         &+T^{\frac{1}{18}}\sum_{j=1,2}\|w\|_{L_T^{\frac{9}{4}}L_x^\infty}
         \|w(D_{x_j}^{\frac{7}{9}}\partial_{x_{1}}w)\|_{L_T^2L_x^2}\lesssim T^{3a}.
    \end{aligned}
\end{equation*}
Here, we use the fact that 
\begin{equation*}
     \sum_{j=1,2}\|w(D_{x_j}^{\frac{7}{9}}\partial_{x_{1}}w)\|_{L_T^2L_x^2}\lesssim
    \sum_{j=1,2} \|w\|_{L_{x_{1}}^{2}L_{x_{2}T}^{\infty}}
     \big\|D^{\frac{7}{9}}_{x_{j}}\partial_{x_{1}}w\big\|_{L_{x_{1}}^{\infty}L_{x_{2}T}^{2}}\lesssim \vertiii{w}^{2}_{T}.
\end{equation*}
Therefore, from \eqref{47}, \eqref{412}, \eqref{43}, \eqref{417} and \eqref{est:lambda}, we have
\begin{equation*}
    \vertiii{\mathcal{K}_{4}}_{T}\lesssim
     \|\partial_{x_{1}}(w^{3})\|_{L_{T}^{1}L_{x}^{2}}+
     \sum_{j=1,2}\big\|D_{x_j}^{\frac{7}{9}}\partial_{x_1}(w^3)\big\|_{L_{T}^{1}L^2_x}\lesssim T^{3a},
\end{equation*}
which completes the proof of~\eqref{est:K4}.

\smallskip
\textbf{Step 4.} Conclusion. From~\eqref{est:K1},~\eqref{est:K2},~\eqref{est:K3} and~\eqref{est:K4}, we have strictly improved estimates in~\eqref{48}, and thus, the proof of Proposition~\ref{Prop:bdd} is complete.
\end{proof}

\subsection{End of the proof for Theorem~\ref{thm:main}}

We start with the following result of the weak $H^{\frac{7}{9}}$ stability 
and convergence of the geometric parameters.
\begin{lemma}\label{lemma:wc}
	Let $v_n\in C([0,1];H^{\frac{7}{9}}(\R^{2}))$ be a sequence of solutions to \eqref{CP} with initial data $v_n(0)=v_{0n}\in H^{\frac{7}{9}}(\R^{2})$. Assume that there exits $v_0\in H^{\frac{7}{9}}(\R^{2})$ such that 
    \begin{equation*}
    v_{0n}\rightharpoonup v_{0} \ \ \ \mbox{in}\ H^{\frac{7}{9}}\ \mbox{as}\ n\to \infty\quad \mbox{and}\quad 
       \sup_{n\in \mathbb{N}^{+}} \sup_{t\in[0,1]}\|v_n(t)\|_{H^{{\frac{7}{9}}}}<\infty.
    \end{equation*}
	Denote $v(t)$ by the corresponding solution to \eqref{CP} with initial data $v(0)=v_0\in H^{\frac{7}{9}}$. Then the following holds.
\begin{enumerate}
\item The solution $v(t)$ exists on $[0,1]$ and satisfies
\begin{equation*}
     v_{n}(t)\rightharpoonup v(t) \ \ \ \mbox{in}\ H^{\frac{7}{9}}\ \mbox{as}\ n\to \infty,\quad \mbox{for all}\ t\in [0,1].
\end{equation*} 
\item Suppose that $\left\{v_n(t)\right\}_{n\in \mathbb{N}^{+}}$ satisfy the geometric decomposition and orthogonality conditions in Proposition~\ref{Prop:decomposition} on $[0,1]$.
Assume further that
\begin{equation*}
\sup_{n\in\mathbb{N}^{+}}\|
\left(\log\lambda_n,b_{n},x_{1n}\right)\|_{C^1([0,1])}<\infty.
\end{equation*}
	Then, $v(t)$ satisfies the geometric decomposition in Proposition~\ref{Prop:decomposition} on $[0,1]$ and its decomposition $(\lambda,b,x_{1},\varepsilon)\in (0,\infty)\times \mathbb{R}^2\times  H^{\frac{7}{9}}(\R^{2})$ satisfies
    \begin{equation*}
    \begin{aligned}
        \varepsilon_{n}(t)\rightharpoonup \varepsilon(t),\quad \mbox{in}\ \ H^{\frac{7}{9}}\ \ \  \mbox{as}\ \ n\to \infty,\quad \mbox{for all}\ t\in [0,1],\\
        \lim_{n\to \infty}\left(\lambda_{n}(t),b_{n}(t),x_{1n}(t)\right)=(\lambda(t),b(t),x_{1}(t)),\quad \mbox{for all}\ t\in [0,1].
        \end{aligned}
    \end{equation*}
\end{enumerate}
\end{lemma}

\begin{proof}
    For the sake of completeness and the readers’ convenience, the details of the proof for Lemma~\ref{L1} are given in Appendix \ref{App:wc}.
\end{proof}

\smallskip
We are in a position to complete the proof of Theorem~\ref{thm:main}. 
\begin{proof}[End of the proof of Theorem~\ref{thm:main}]
From~\eqref{est:uniform} and Proposition~\ref{Prop:bdd}, for any $n\in \mathbb{N}^{+}$, there exists a solution $\phi_{n}(t)$ of~\eqref{CP} on $[T_{n},T_{0}]$ such that for all $T_1\in (0,T_0]$, it holds
\begin{equation}\label{est:uniformn}
    \sup_{n\geq n_0}\left(\sup_{t\in[T_{1},T_0]}\|\phi_{n}(t)\|_{H^{\frac{7}{9}}}+\left\|\left(\log \lambda_{n},b_{n},x_{1n}\right)\right\|_{C^1([T_{1},T_0])}\right)<C(T_1)<\infty,
\end{equation}
provided that $n_0=n_0(T_1)\gg1$. Here, $(\lambda_{n}(t),b_{n}(t),x_{1n}(t))$ is the geometric parameters of the decomposition for $\phi_{n}(t)$ on $[T_{n},T_{0}]$. Since the sequence $(\phi_{n}(T_{0}))_{n\in \mathbb{N}^{+}}$ is bounded in $H^{\frac{7}{9}}$, up to the extraction of a subsequence, there exists a $\phi_{0}\in H^{\frac{7}{9}}$ such that $\phi_{n}(T_{0})\rightharpoonup\phi_{0}$ weakly in $H^{\frac{7}{9}}$.
Let $\mathcal{S}(t)$ be the backward-in-time solution with initial data $\mathcal{S}(T_{0})=\phi_0$. Using Lemma \ref{lemma:wc}, we know that  $\mathcal{S}(t)$ exists on $(0,T_0]$. Here, we use the fact that $T_{n}\downarrow0$ as $n\to \infty$ and~\eqref{est:uniformn}. Using again Proposition~\ref{Prop:bdd}, Lemma~\ref{lemma:wc} and~\eqref{est:uniformn},
for all $t\in (0,T_{0}]$, we have 
\begin{equation*}
\begin{aligned}
    \mathcal{S}(t,x)
    &=\frac{1}{\lambda(t)}(W_{b(t),\lambda(t)}+\varepsilon)\left(t,\frac{x_1-x_{1}(t)}{\lambda(t)},\frac{x_2}{\lambda(t)}\right)\\
    &=\frac{1}{\lambda(t)}W_{b(t),\lambda(t)}\left(t,\frac{x_1-x_{1}(t)}{\lambda(t)},\frac{x_2}{\lambda(t)}\right)+w(t,x),
    \end{aligned}
\end{equation*}
and
\begin{gather*}
        \varepsilon_{n}(t)\rightharpoonup \varepsilon(t),\quad \mbox{in}\ \ H^{\frac{7}{9}},\\
        \lim_{n\to \infty}\left(\lambda_{n}(t),b_{n}(t),x_{1n}(t)\right)=(\lambda(t),b(t),x_{1}(t)).
\end{gather*}
Here, $(\lambda(t),b(t),x_{1}(t))$ is the geometric parameters of the decomposition for $\phi(t)$ on $(0,T_{0}]$. In addition, using again~\eqref{est:uniform}, we deduce that 
\begin{equation}\label{est:u}
\left\{
	\begin{aligned}
		&|b(t)|
		\lesssim t^{\frac{4\theta}{3-\theta}},\quad 
        \|w(t)\|_{H^\frac{7}{9}}\lesssim t^{a},
        \\
		&\left|\lambda(t)-\left(\frac{3-\theta}{\theta}\right)^{\frac{1}{3-\theta}}t^{\frac{1}{3-\theta}}\right|
		\lesssim t^{\frac{1+\theta}{3-\theta}}\left|\log t\right|,\\
		&\left|x_{1}(t)+\frac{\theta}{\theta-1}\left(\frac{\theta}{3-\theta}\right)^{\frac{\theta-1}{3-\theta}}t^{\frac{1-\theta}{3-\theta}}\right|\lesssim t^{\frac{1}{3-\theta}}\left|\log t\right|,
	\end{aligned}
	\right.\quad \mbox{for all}\ t\in (0,T_{0}].
\end{equation}
It follows directly that $\mathcal{S}(t)\in C((0,T_0];H^{\frac{7}{9}})$ is a solution of \eqref{CP} which blows up backward at $t=0$. Next, using Lemma \ref{le:estW1}, Proposition \ref{prop:Wb}, \eqref{est:u} and the conservation law of mass, we see that 
\begin{equation*}
    \lim_{t\to 0}\|\mathcal{S}(t)\|_{L^{2}}=\|Q\|_{L^{2}}\Longrightarrow \|\mathcal{S}(t)\|_{L^{2}}\equiv \|Q\|_{L^{2}}.
\end{equation*}
Last, using Lemma \ref{le:estW1} , Proposition \ref{prop:Wb} and \eqref{est:u} again, we have
\begin{equation*}
    \begin{aligned}
        \left\|\mathcal{S}(t)-\frac{1}{\lambda(t)}Q\left(\frac{\cdot-x(t)}{\lambda(t)}\right)\right\|_{H^{\frac{7}{9}}} &\lesssim \frac{1}{\lambda^{\frac{7}{9}}(t)}
        \left\|W_{b(t),\lambda(t)}-Q\right\|_{H^{1}}+\|w(t)\|_{H^{\frac{7}{9}}}\\
        & \lesssim \lambda^{\frac{5\theta}{2}-\frac{7}{9}}(t)+\lambda^{\theta-\frac{71}{45}}(t)+t^a\to 0,\quad \mbox{as}\ \ t\downarrow 0.
    \end{aligned}
\end{equation*}
Here, we denote $x(t)=(x_{1}(t),0)$ and then use the fact that 
\begin{equation*}
    \theta\in \left(\frac{8}{5},\frac{9}{5}\right)\Longrightarrow \frac{5\theta}{2}-\frac{7}{9}>0\quad \mbox{and}\quad \theta-\frac{71}{45}>0.
\end{equation*}
The proof of Theorem~\ref{thm:main} is complete.
\end{proof}

\appendix
\section{Proof of Lemma \ref{L1}}\label{App:linear}
In this appendix, we give a complete proof of Lemma \ref{L1}.
\begin{proof}[Proof of Lemma~\ref{L1}]
{\textbf{Step 1.}} Proof of \eqref{41}--\eqref{44}. Estimate \eqref{41} follows from Fourier-Plancherel's theorem and the definition of $U(t)$. Estimates~\eqref{410} and~\eqref{46} were proved in Linares-Pastor \cite[Lemma 2.6]{LinPas} and Faminskii \cite[Theorem 2.4]{Faminskii}, respectively. Estimate \eqref{42} follows directly from Kato's smoothing effect (see \emph{e.g.}~\cite[Theorem 2.2]{Faminskii}). Estimates \eqref{47} and~\eqref{44} follow from \eqref{41} and \eqref{42}, respectively.

\smallskip
{\textbf{Step 2.}} 
Proof of \eqref{412}. By the Minkowski inequality, we have
\begin{equation*}
\begin{aligned}
    \left\|\int_0^t U(t-\tau)g(\tau)\dd \tau\right\|_{L_T^{\frac{9}{4}}L^\infty_{x}}&\lesssim \int_0^T\left(\int_{\tau}^T\|U(t-\tau)g(\tau)\|^{\frac{9}{4}}_{L_x^\infty}\dd t\right)^{\frac{4}{9}}\dd \tau\\
    &\lesssim \int_0^T\left\|U(t)U(-\tau)g(\tau)\right\|_{L_T^{\frac{9}{4}}L_x^\infty}\dd \tau.
    \end{aligned}
\end{equation*}
Combining the above estimate with~\eqref{410}, we obtain
\begin{equation*}
\begin{aligned}
 \left\|\int_0^t U(t-\tau)g(\tau)\dd \tau\right\|_{L_T^{\frac{9}{4}}L^\infty_{x}}&\lesssim \int_0^T  \|U(-\tau)g(\tau)\|_{L_x^2}\dd \tau\lesssim \|g\|_{L_T^1L_x^2},\\
     \left\|\int_0^t U(t-\tau)g(\tau)\dd \tau\right\|_{L_T^{\frac{9}{4}}L^\infty_{x}}&\lesssim \int_0^T  \|D_{x_1}^{-\frac{2}{9}}U(-\tau)g(\tau)\|_{L_x^2}\dd \tau\lesssim \| D_{x_1}^{-\frac{2}{9}}g\|_{L_T^1L_x^2},
     \end{aligned}
\end{equation*}
which complete the proof of~\eqref{412}.

\smallskip
{\textbf{Step 3.}}
Proof of \eqref{419}. 
Using \eqref{410}, we directly have 
\begin{equation*}
\begin{aligned}
    \left\|\partial_{x_1}\int_0^TU(t-\tau)g(\tau)\dd \tau\right\|_{L^{\frac{9}{4}}_{T}L^\infty_{x}}
    &\lesssim
    \left\|U(t)\int_0^TU(-\tau)\partial_{x_1}g(\tau)\dd \tau\right\|_{L^{\frac{9}{4}}_{T}L^\infty_{x}}\\
    &\lesssim \left\|\partial_{x_{1}}\int_0^tU(t-\tau)g(\tau)\dd \tau\right\|_{L_{T}^{\infty}L^2_{x}}.
    \end{aligned}
\end{equation*}
Based on the above estimate and~\eqref{44}, we obtain 
\begin{equation*}
    \left\|\partial_{x_1}\int_0^TU(t-\tau)g(\tau)\dd \tau\right\|_{L^{\frac{9}{4}}_{T}L^\infty_{x}}\lesssim\|g\|_{L_{x_1}^1L_{x_2T}^2}\lesssim \|\langle x_1\rangle g\|_{L_T^2L_x^2}.
\end{equation*}
Using a similar argument as above, we also obtain 
\begin{equation*}
     \left\|\partial_{x_1}\int_0^TU(t-\tau)g(\tau)\dd \tau\right\|_{L^{\frac{9}{4}}_{T}L^\infty_{x}}\lesssim\|D_{x_{1}}^{-\frac{2}{9}}g\|_{L_{x_1}^1L_{x_2T}^2}\lesssim \|\langle x_1\rangle 
     D_{x_{1}}^{-\frac{2}{9}}
     g\|_{L_T^2L_x^2}.
\end{equation*}
Combining above estimates with \cite[Lemma 3.1]{SmithSogge}, we complete the proof of~\eqref{419}.

\smallskip
\textbf{Step 4.} Proof of~\eqref{43} and~\eqref{45}. Using \eqref{46}, we directly have
\begin{align*}
&\left\|\int_0^TU(t-\tau)g(\tau)\dd \tau\right\|_{L^2_{x_1}L^\infty_{x_2T}}\lesssim\left\|\int_0^TU(-\tau)g(\tau)\dd \tau\right\|_{H^{\frac{7}{9}}(\mathbb{R}^2)}\lesssim \|g\|_{L_T^1H_x^{\frac{7}{9}}}.
\end{align*}
Combining the above estimate with a similar argument as \cite[Lemma 3]{MoRi} and \cite[Proposition 3.7]{RibaudVento2}, we complete the proof of~\eqref{43}.

\smallskip
We denote by $\{\Delta_k\}_{k=0}^{\infty}$ the standard Littlewood-Paley decomposition over $\mathbb{R}^2$. First, from~\eqref{46} and \eqref{44}, for all $k\in \mathbb{N}$ and $0<\delta\ll1$, we have 
\begin{equation*}
    \begin{aligned}
        &\left\|
        \Delta_{k}\partial_{x_{1}}
        \int_0^TU(t-\tau)g(\tau)\dd \tau\right\|_{L^2_{x_1}L^\infty_{x_2T}}\\
&\lesssim\left\|\partial_{x_1}\int_0^TU(-\tau)\Delta_kg(\tau)\dd \tau\right\|_{H^{\frac{3}{4}+\delta}(\mathbb{R}^2)}\\
&\lesssim 2^{k\left(-\frac{1}{36}+\delta\right)}\left\|\partial_{x_1}\int_0^TU(T-\tau)\left(1+D_{x_1}^{\frac{7}{9}}+D_{x_2}^{\frac{7}{9}}\right)\Delta_kg(\tau)\dd \tau\right\|_{L^2(\mathbb{R}^2)}\\
&\lesssim2^{k\left(-\frac{1}{36}+\delta\right)}\left\|\partial_{x_1}\int_0^TU(T-\tau)\left(1+D_{x_1}^{\frac{7}{9}}+D_{x_2}^{\frac{7}{9}}\right)g(\tau)\dd \tau\right\|_{L^2(\mathbb{R}^2)}\\
&\lesssim 2^{k\left(-\frac{1}{36}+\delta\right)}\left\| \left(1+D_{x_1}^{\frac{7}{9}}+D_{x_2}^{\frac{7}{9}}\right)g\right\|_{L_{x_1}^1L_{x_2T}^2},
    \end{aligned}
\end{equation*}
which implies 
\begin{equation*}
        \left\|
        \partial_{x_{1}}
        \int_0^TU(t-\tau)g(\tau)\dd \tau\right\|_{L^2_{x_1}L^\infty_{x_2T}}\lesssim
        \left\|
        \left(1+D_{x_1}^{\frac{7}{9}}+D_{x_2}^{\frac{7}{9}}\right)g
        \right\|_{L^{1}_{x_{1}}L_{x_{2}T}^{2}}.
\end{equation*}
Combining the above estimate with a similar argument as \cite[Lemma 3]{MoRi} and \cite[Proposition 3.7]{RibaudVento2}, we complete the proof of~\eqref{45}.

\smallskip
\textbf{Step 5.} Proof of \eqref{417}. Using again the Minkowski inequality, we have 
\begin{equation*}
    \begin{aligned}
        \left\|\partial_{x_1}\int_0^tU(t-\tau)g(\tau)\dd \tau\right\|_{L^\infty_{x_1}L^2_{x_2T}}
        &\lesssim \left\|\int_0^T\left(\int_{\tau}^T\left|U(t-\tau)\partial_{x_1}g(\tau)\right|^2\dd t\right)^{\frac{1}{2}}\dd \tau\right\|_{L^\infty_{x_1}L_{x_2}^2}\\
        &\lesssim \int_0^T\left\|\partial_{x_1}U(t)U(-\tau)g(\tau)\right\|_{L^\infty_{x_1}L_{x_2T}^2}\dd \tau.
    \end{aligned}
\end{equation*}
Combining the above estimate with~\eqref{42}, we obtain 
\begin{equation*}
    \left\|\partial_{x_1}\int_0^tU(t-\tau)g(\tau)\dd \tau\right\|_{L^\infty_{x_1}L^2_{x_2T}}\lesssim \int_{0}^{T}\|g(\tau)\|_{L_{x}^{2}}\dd \tau \lesssim \|g\|_{L_{T}^{1}L_{x}^{2}},
\end{equation*}
which completes the proof of~\eqref{417}.

\smallskip
{\textbf{Step 6.}} Proof of \eqref{413}. The proof of estimate \eqref{413} is similar to \cite[Theorem 3.5]{KPV} and \cite[Proposition 3.6]{RibaudVento2}, but it is given for completeness and the reader's convenience. 
To simplify notation, we denote 
\begin{equation*}
    D(t,x)=\int_{0}^{t}U(t-\tau)g(\tau,x)\dd \tau,\quad \mbox{for any}\ (t,x)\in [0,T]\times \R^{2}.
\end{equation*}
We decompose 
\begin{equation*}
    D(t,x)=2D_{1}(t,x)+2D_{2}(t,x),
\end{equation*}
with
\begin{equation*}
\begin{aligned}
D_{1}(t,x)&=\int_{\R}U(t-\tau)g(\tau){\rm sgn}(\tau)\dd \tau,\\
    D_{2}(t,x)&=\int_{\R}U(t-\tau)g(\tau){\rm sgn}(t-\tau)\dd \tau.
    \end{aligned}
\end{equation*}

\emph{Estimate on $D_{1}$.} Without loss of generality, we may assume that 
\begin{equation*}
    g\equiv 0,\quad \mbox{for}\ t\notin [0,T]\Longrightarrow
    D_{1}(t,x)=\int_{0}^{T}U(t-\tau)g(\tau)\dd \tau.
\end{equation*}
Note that, from~\eqref{42}, we directly have 
\begin{equation}\label{est:AppD1}
\left\|\partial_{x_{1}}^{2}D_{1}\right\|_{L_{x_{1}}^{\infty}L_{x_{2}T}^{2}}
\lesssim \left\|\partial_{x_{1}}\int_{0}^{T}U(T-\tau)g(\tau)\dd \tau\right\|_{L^{2}}\lesssim \|g\|_{L^1_{x_1}L_{x_2T}^2}.
\end{equation}

\emph{Estimate on $D_{2}$.} By an elementary computation, we find 
\begin{equation*}
     D_{2}(t,x)=\mathcal{F}^{-1}_{(\sigma,\xi)\to (t,x)}\left(\widehat{\rm sgn}(\sigma-\xi_1|\xi|^2)\widehat{g}(\sigma,\xi)\right).
\end{equation*}
Using the Fourier-Plancherel theorem, we have
\begin{equation}\label{est:D}
    \begin{aligned}
\left\|\partial_{x_1}^2D_1\right\|_{L_{x_2T}^2}
=&\left\|\mathcal{F}_{\xi_1\to x_{1}}^{-1}\left(\xi_1^2\widehat{\rm sgn}(\sigma-\xi_1|\xi|^2)\widehat{g}(\sigma,\xi)\right)\right\|_{L^2_{\xi_2\sigma}}\\
=&2\left\|K(\sigma,x_1,\xi_2)\star
\mathcal{F}_{(t,x_2)\to (\sigma,\xi_{2})}(g(x_1))(\sigma,\xi_2)\right\|_{L_{\xi_2\sigma}^2}.
    \end{aligned}
\end{equation}
where $K$ is the inverse Fourier transform (in $\xi_1$) of the tempered distribution given by the principal value (in $\sigma$) of $\frac{\xi_1^2}{\sigma-\xi_1|\xi|^2}$.
Here, $\mathcal{F}_{\xi_1\to x_{1}}^{-1}$ denotes the inverse Fourier transform in $\xi_1$, $\mathcal{F}_{(t,x_2)\to (\sigma,\xi_{2})}$ denotes the Fourier transform in $(t,x_2)$ and $\star$ denotes convolution for $x_1$.

\smallskip
We claim that
\begin{equation*}
    K(\sigma,x_{1},\xi_{2})\in L^{\infty}(\R^{3}).
\end{equation*}
Indeed, we compute 
\begin{equation}
\begin{aligned}
K(\tau,x_1,\xi_2)&=\frac{1}{\sqrt{2\pi}}\int_{\mathbb{R}}e^{ix_1\xi_1}\frac{\xi_1^2}{\sigma-\xi_1(\xi_1^2+\xi_2^2)}\dd\xi_1\\
&=\frac{1}{\sqrt{2\pi}}\int_{\mathbb{R}}e^{i\overline{x}_1\xi_1}\frac{\xi_1^2}{\overline{\sigma}-\xi_1(1+\xi_1^2)}\dd \xi_1,
\end{aligned}
\end{equation}
where $\overline{x}_1=|\xi_2|x_1$ and $\overline{\sigma}=\sigma/|\xi_2|^3$. Here, the integral related to $K$ is understood as an oscillatory integral.
We denote $\overline{z}=\overline{z}(\sigma,\xi_2)$ by the only real root of the polynomial $z(z^2+1)-\overline{\sigma}=0$. Then we decompose
\begin{equation*}
    \sqrt{2\pi}K=K_{1}+K_{2}+K_{3},
\end{equation*}
where
\begin{equation*}
\begin{aligned}
    K_{1}(\sigma,x_{1},\xi_{2})&=\frac{\overline{z}^2}{3\overline{z}^2+1}\int_{\mathbb{R}}\frac{e^{i\overline{x}_{1}\xi_1}}{\overline{z}-\xi_{1}}\dd \xi_1,\\
    K_{2}(\sigma,x_{1},\xi_{2})&=-
    \frac{2\overline{z}^2+1}{3\overline{z}^2+1}\int_{\mathbb{R}}\frac{\xi_1
    e^{i\overline{x}_{1}\xi_1}
    }{\xi_1^2+\overline{z}\xi_1+\overline{z}^2+1}\dd \xi_1
    ,\\
    K_{3}(\sigma,x_{1},\xi_{2})&=-\frac{\overline{z}^3+\overline{z}}{3\overline{z}^2+1}\int_{\mathbb{R}}\frac{e^{i\overline{x}_1\xi_1}}{\xi_1^2+\overline{z}\xi_1+\overline{z}^2+1}\dd\xi_1.
    \end{aligned}
\end{equation*}
First, from an elementary computation, we have 
\begin{equation*}
    \left|K_{1}\right|\lesssim 
    \left|\frac{\overline{z}^2}{3\overline{z}^2+1}e^{i\overline{x}_{1}\overline{z}}\int_{\R}\frac{e^{i\overline{x}_{1}\xi_1}}{\xi_{1}}\dd \xi_{1}\right|
    \lesssim \left|\frac{\overline{z}^2}{3\overline{z}^2+1}e^{i\overline{x}_{1}\overline{z}}{\rm{sgn}}(\overline{x}_{1})\right|
    \lesssim 1.
\end{equation*}
Second, by an elementary computation, we find 
\begin{equation*}
    \frac{\xi_1}
{\xi_1^2+\overline{z}\xi_1+\overline{z}^2+1}=\frac{\xi_{1}+\frac{\overline{z}}{2}}{\left(\xi_{1}+\frac{\overline{z}}{2}\right)^{2}+\frac{3}{4}\overline{z}^{2}+1}+\frac{\overline{z}}{2\left(\xi_1^2+\overline{z}\xi_1+\overline{z}^2+1\right)}.
\end{equation*}
From the change of variable and the computation of the Fourier transform of $e^{-|x|}$,
\begin{equation*}
    \begin{aligned}
        \left|K_{2}\right|&\lesssim \left|\int_{\R}\exp\left({i\sqrt{1+\frac{3}{4}\overline{z}^{2}}\overline{x}_{1}\xi_{1}}\right)\frac{\xi_{1}}{1+\xi_{1}^{2}}\dd \xi\right|+\left|\int_{\R^{2}}\frac{\overline{z}\dd \xi_{1}}{\xi_1^2+\overline{z}\xi_1+\overline{z}^2+1}\right|\lesssim 1.
    \end{aligned}
\end{equation*}
Then, from the Cauchy-Schwarz inequality and the change of variable, we find 
\begin{equation*}
    \left|K_{3}\right|\lesssim |\overline{z}|\int_{\R^{2}}\left|\frac{\dd \xi_{1}}{\xi_1^2+\overline{z}\xi_1+\overline{z}^2+1}\right|\lesssim 1.
\end{equation*}
Combining the above estimates for $K_{1}$, $K_{2}$ and $K_{3}$, we see that $K\in L^{\infty}(\R^{3})$.
Therefore, from Young's inequality and \eqref{est:D}, we obtain 
\begin{equation*}
    \left\|\partial_{x_1}^2D_1\right\|_{L_{x_1}^\infty L_{x_2T}^2}\lesssim \|K\|_{L^\infty}\left\|\mathcal{F}_{(t,x_2)\to (\sigma,\xi_{2})}(g(x_1))(\sigma,\xi_2)\right\|_{L_{x_1}^1L_{\xi_2\sigma}^2}\lesssim\|g\|_{L^1_{x_1}L_{x_2T}^2}.
\end{equation*}
which completes the proof of \eqref{413}.

\end{proof}

\section{Proof of Lemma~\ref{lemma:wc}}\label{App:wc}
In this appendix, we give a complete proof of Lemma~\ref{lemma:wc}. 
\begin{proof}[Proof of Lemma~\ref{lemma:wc}]
\textbf{Step 1}. The $H^{10}$ case. We assume that 
\begin{equation*}
    v_{0n}\rightharpoonup v_{0} \ \ \ \mbox{in}\ H^{10}\ \mbox{as}\ \ n\to \infty\quad \mbox{and}\quad 
       \sup_{n\in \mathbb{N}^{+}} \sup_{t\in[0,1]}\|v_n(t)\|_{H^{{10}}}<\infty.
    \end{equation*}
From the local theory of Cauchy problem~\eqref{CP}, we may {\it a priori} assume that
\begin{equation*}
    v(t)\ \ \mbox{exists on}\ \ [0,1]\quad \mbox{and}\quad 
    \sup_{t\in[0,1]}\|v(t)\|_{H^{10}}<\infty.
\end{equation*}
Without loss of generality, we only consider the weak convergence at $t=1$.

\smallskip
After passing to a subsequence,  there exists $\widetilde{v}:[0,1]\times \R^{2}\to \R$ such that
\begin{equation*}
    v_{n}\rightharpoonup \widetilde{v}\ \ \ \mbox{in}\ \ C([0,1];H^{10})\ \ \mbox{as}\ \ n\to \infty.
\end{equation*}
Since $(v_n)_{n\in \mathbb{N}^{+}}$ are solutions of \eqref{CP} and bounded in $C([0,1];H^{10})$, we find $(v_n)_{n\in \mathbb{N}^{+}}$ are also bounded in $C^1\left([0,1];H^7\right)$. Hence, for any compact set $K\subset \R^{2}$, we have 
\begin{equation*}
    C^1([0,1];H^{7})\hookrightarrow C([0,1];L^2(K))\Longrightarrow
      v_{n}\to \widetilde{v}\ \ \ \mbox{in}\ \ C([0,1];L^{2}(K))\ \ \mbox{as}\ \ n\to \infty.
\end{equation*}

Let $\omega_{n}=v_{n}-v$ and $\zeta_{n}=v_{0n}-v_{0}$.  Then we have
\begin{equation*}
		\partial_t\omega_{n}+\partial_{x_1}\Delta \omega_n+3v_n^2\partial_{x_{1}}\omega_n+3\omega_n(v_n+v)\partial_{x_{1}}v=0,\ \ \mbox{with}\ \ \omega_{n}(0)=\zeta_{n}.
\end{equation*}
Next, for any $g\in C_c^\infty(\mathbb{R}^2)$, we consider $z\in C([0,1];H^3)$ to be the solution of%
\footnote{The linear operator associated to this linear PDE is $$H=\partial_{x_1}\Delta+3\widetilde{v}^2\partial_{x_1}+6\widetilde{v}\partial_{x_{1}}\widetilde{v}-3(\widetilde{v}+v)\partial_{x_{1}}v.$$
It is easy to check that, for any $\lambda>0$ large enough, the operator $H-\lambda$ with domain $H^3$ is maximal dissipative. Therefore, this linear PDE has a unique strong solution in $C([0,1];H^3)$.}
\begin{equation*}
    \partial_tz+\partial_{x_1}\Delta z+3\partial_{x_{1}}\left(\widetilde{v}^2z\right)-3z(\widetilde{v}+v)\partial_{x_{1}}v=0,\quad \mbox{with}\ \ z(1)=g.
\end{equation*}
It is easy to check that 
\begin{equation*}
    \begin{aligned}
    \sup_{n\in \mathbb{N}^{+}}\sup_{t\in[0,1]}\left(\|\omega_n(t)\|_{H^{10}}+\|v_n(t)\|_{H^{10}}\right)&<\infty,\\
        \sup_{t\in [0,1]}\left(\|z(t)\|_{H^{3}}+\|v(t)\|_{H^{10}}+\|\widetilde{v}(t)\|_{H^{10}}\right)&<\infty.
    \end{aligned}
\end{equation*}
Then, from integration by parts and the equation of $\omega_{n}$ and $z$, we find 
\begin{align*}
	&\int_{\mathbb{R}^2}\omega_n(1,x)g(x)\dd x-\int_{\mathbb{R}^2}\zeta_n(x)z(0,x)\dd x\\
	&=3\int_0^1\int_{\mathbb{R}^2}\omega_nz\left(\widetilde{v}-v_n\right)\left(\partial_{x_{1}}v\right)\dd x\dd t+3\int_0^1\int_{\mathbb{R}^2}z\left(\widetilde{v}^2-v_n^2\right)\left(\partial_{x_{1}}\omega_{n}\right)\dd x\dd t\\
	&\lesssim \sup_{t\in[0,1]}\left(\int_{|x|<R}|\widetilde{v}(t,x)-v_n(t,x)|^2\dd x\right)^{\frac{1}{2}}+\int_0^1\int_{|x|>R}|z(t,x)|^2\dd x\dd t.
\end{align*}
Here, we choose $R$ to be large enough and $K=\left\{x\in \R^{2}:|x|\le \R\right\}$ to be a compact set on $\R^{2}$. Recall that, we have 
\begin{equation*}
     v_{n}\to \widetilde{v}\ \ \ \mbox{in}\ \ C([0,1];L^{2}(K))\quad \mbox{and}\quad \lim_{R\rightarrow\infty}\int_0^1\int_{|x|>R}|z(t,x)|^2\dd x\dd t=0.
\end{equation*}
Combining the above estimates, we obtain 
\begin{align*}
	\lim_{n\rightarrow\infty}\left(\int_{\mathbb{R}^2}\omega_n(1,x)g(x)\dd x-\int_{\mathbb{R}^2}\zeta_n(x)z(0,x)\dd x\right)=0,
\end{align*}
which implies $\omega_n(1)\rightharpoonup 0$ in $H^{\frac{7}{9}}$, and thus, $v_{n}(t)\rightharpoonup v(t)$ in $H^{\frac{7}{9}}$.

\smallskip
\textbf{Step 2.} General case. Suppose that  
\begin{equation}\label{est:AppBuni}
    v_{0n}\rightharpoonup v_{0} \ \ \ \mbox{in}\ H^{\frac{7}{9}}\ \mbox{as}\ \ n\to \infty\quad \mbox{and}\quad 
       \sup_{n\in \mathbb{N}^{+}} \sup_{t\in[0,1]}\|v_n(t)\|_{H^{{\frac{7}{9}}}}<\infty.
    \end{equation}
    For any $N\in \mathbb{N}^{+}$ large enough, we denote 
    \begin{equation*}
        v_{0n}^{N}=\mathcal{F}^{-1}\left(\mathbf{1}_{\{|\xi|\leq N\}}(\xi)\widehat{v}_{0n}(\xi)\right)
        \quad \mbox{and}\quad 
        v_{0}^{N}=\mathcal{F}^{-1}\left(\mathbf{1}_{\{|\xi|\leq N\}}(\xi)\widehat{v}_{0}(\xi)\right).
    \end{equation*}
Using~\eqref{est:AppBuni}, we directly have 
\begin{equation}\label{est:AppBsupn}
    \sup_{n\in \mathbb{N}^{+}}\|v_{0n}^N\|_{H^{10}}^2
    \lesssim\int_{\{|\xi|\leq N\}}\langle\xi\rangle^{20}|\widehat{v}_{0n}(\xi)|^2\dd\xi\lesssim N^{20}\sup_{n\in \mathbb{N}^{+}}\|v_{0n}\|_{L^2}^2\lesssim N^{20}.
\end{equation}
Let $v_n^N\in C([0,T];H^{10})$ and $v^N\in C([0,T];H^{10})$ be the solution to \eqref{CP} with initial data $v^N_n(0)=v^N_{0n}$ and $v^N(0)=v^N_0$, respectively. Recall that, the Cauchy problem \eqref{CP} is locally well-posed in $H^{r}$ for all $r>\frac{3}{4}$ and \emph{the initial data to solution} map is smooth. Therefore, from~\eqref{est:AppBsupn} and the conclusion in Step 1, for any $t\in [0,1]$ and $g\in C_{c}^{\infty}(\R^{2})$, we obtain
\begin{equation}\label{est:AppBweakN}
    \sup_{n\in \mathbb{N}^{+}}\sup_{t\in [0,1]}\|v^N_n(t)\|_{H^{10}}<\infty \Longrightarrow
    \lim_{n\rightarrow\infty}\int_{\mathbb{R}^2}\left(v_n^N(t,x)-v^N(t,x)\right)g(x)\dd x=0.
\end{equation}
On the other hand, from~\eqref{est:AppBuni}, for any $r\in\left(\frac{3}{4},\frac{7}{9}\right)$, we have 
\begin{equation*}
\begin{aligned}
    \|v_{0n}^N-v_{0n}\|_{H^{r}}^2+ \|v_{0}^N-v_{0}\|_{H^{r}}^2
    &\lesssim \int_{\left\{\left|\xi\right|\ge N\right\}}\langle\xi\rangle^{2r}
    \left(|\widehat{v}_{0n}(\xi)|^2+|\widehat{v}_{0}(\xi)|^2\right)\dd\xi\\
    &\lesssim N^{-\frac{14}{9}+2r}\left(\|v_{0n}\|_{H^{\frac{7}{9}}}^2+\|v_{0}\|_{H^{\frac{7}{9}}}\right)\to 0, \ \mbox{as}\ N\to \infty.
    \end{aligned}
\end{equation*}
Using again the local theory of the Cauchy problem~\eqref{CP},  we have
\begin{equation}\label{est:AppBHr}
    \sup_{n\in \mathbb{N}^{+}}\sup_{t\in[0,1]}\left(\|v^N_n(t)-v_n(t)\|_{H^{r}}+\|v^N(t)-v(t)\|_{H^{r}}\right)\to 0,\quad \mbox{as}\ N\to \infty.
\end{equation}
Combining estimates~\eqref{est:AppBweakN} and~\eqref{est:AppBHr}, for any $g\in C_{c}^{\infty}(\R^{2})$, we obtain 
\begin{equation*}
    \lim_{n\rightarrow\infty}\int_{\mathbb{R}^2}\left(v_n(t,x)-v(t,x)\right)g(x)\dd x=0,
\end{equation*}
which implies a weak convergence immediately, and thus the proof is complete.

\smallskip
\textbf{Step 3.} Geometric parameters.
By the Arzela-Ascoli Theorem, we know that there exist 
$(\lambda,b,x_{1},\varepsilon)$ and a subsequences $\left(\lambda_{n_k},b_{n_k},x_{1n_k},\varepsilon_{n_{k}}\right)_{k\in \mathbb{N}^{+}}$ satisfy
\begin{equation*}
    \begin{aligned}
        \varepsilon_{n_{k}}(t)\rightharpoonup \varepsilon(t),\quad \mbox{in}\ \ H^{\frac{7}{9}}\ \ \  \mbox{as}\ \ k\to \infty,\quad \mbox{for all}\ t\in [0,1],\\
        \lim_{k\to \infty}\left(\lambda_{n_{k}}(t),b_{n_{k}}(t),x_{1n_{k}}(t)\right)=(\lambda(t),b(t),x_{1}(t)),\quad \mbox{for all}\ t\in [0,1].
        \end{aligned}
    \end{equation*}
    It follows directly that $v(t)$ satisfies the geometric decomposition in Proposition~\ref{Prop:decomposition} with parameters $(\lambda,b,x_{1},\varepsilon)$ on $[0,1]$. Last, using a standard contradiction argument and the uniqueness of the decomposition, we conclude that 
    \begin{equation*}
    \begin{aligned}
        \varepsilon_{n}(t)\rightharpoonup \varepsilon(t),\quad \mbox{in}\ \ H^{\frac{7}{9}}\ \ \  \mbox{as}\ \ n\to \infty,\quad \mbox{for all}\ t\in [0,1],\\
        \lim_{n\to \infty}\left(\lambda_{n}(t),b_{n}(t),x_{1n}(t)\right)=(\lambda(t),b(t),x_{1}(t)),\quad \mbox{for all}\ t\in [0,1].
        \end{aligned}
    \end{equation*}
    The proof of Lemma~\ref{lemma:wc} is complete.

\end{proof}


\begin{thebibliography}{99}
	
	\bibitem{Banica}
	V. Banica, R. Carles and T. Duyckaerts.
	Minimal blow-up solutions to the mass-critical inhomogeneous NLS equation. \emph{Comm. Partial Differential Equations} \textbf{36} (2011), no. 3, 487–531.
	
\bibitem{BGFR}
D. Bhattacharya, L. G. Farah and S. Roudenko.
Global well-posedness for low regularity data in the 2d modified Zakharov-Kuznetsov equation.
\emph{J. Differential Equations} \textbf{268} (2020), no. 12, 7962–7997.

\bibitem{BGMY} 
F. Bozgan, T. E. Ghoul, N. Masmoudi and K. Yang.
Blow-Up Dynamics for the $L^{2}$ critical case of the 2D Zakharov-Kuznetsov equation.
Preprint, arXiv: 2406.06568.

\bibitem{Caz}
T. Cazenave. 
\emph{Semilinear Schr\"odinger equations.}
 Courant Lecture Notes in Mathematics 10. Providence,
RI: American Mathematical Society (AMS); New York, NY: CIMS (2003).

\bibitem{CLY}
G. Chen, Y. Lan and  X. Yuan. 
On the near soliton dynamics for the 2D cubic Zakharov-Kuznetsov equations.
To appear in \emph{Comm. Math. Phys.}

\bibitem{CLYMINI}
G. Chen, Y. Lan and X. Yuan. 
Nonexistence of minimal mass blow-up solution for the 2D cubic Zakharov-Kuznetsov equation. To appear in \emph{SIAM J. Math. Anal.}

\bibitem{Coll}
J.~Colliander and P.~Rapha\"el.
Rough blowup solutions to the $L^{2}$ critical NLS.
\emph{Math. Ann.} \textbf{345} (2009), no. 2, 307–366.

\bibitem{Com}
V. Combet and F. Genoud.
Classification of minimal mass blow-up solutions for an $L^{2}$ critical inhomogeneous NLS. 
\emph{J. Evol. Equ.} \textbf{16} (2016), no. 2, 483–500.

\bibitem{Comkdv}
V. Combet and Y. Martel.
Sharp asymptotics for the minimal mass blow up solution of the critical gKdV equation. 
\emph{Bull. Sci. Math.} \textbf{141} (2017), no. 2, 20–103. 

\bibitem{CMPS} 
R. C\^ote, C. Mu\~noz. D. Pilod and G. Simpson.
 Asymptotic stability of high-dimensional Zakharov-Kuznetsov solitons. 
\emph{Arch. Ration. Mech. Anal}. \textbf{220} (2016), no. 2, 639–710.

\bibitem{DEZK}
A. de Bouard.
Stability and instability of some nonlinear dispersive solitary waves in higher dimension.
\emph{Proc. Roy. Soc. Edinburgh Sect. A} \textbf{126} (1996), no. 1, 89–112.

\bibitem{Dodson1}
B. Dodson. 
A determination of the blowup solutions to the focusing NLS with mass equal to the mass of the soliton. \emph{Ann. PDE} \textbf{9} (2023), no. 1, Paper No. 3, 86 pp.

\bibitem{Dodson2}
B. Dodson.
A determination of the blowup solutions to the focusing, quintic NLS with mass equal to the mass of the soliton. \emph{Anal. PDE} \textbf{17} (2024), no. 5, 1693--1760.

\bibitem{Faminskii}
A. V. Faminskii.
The Cauchy problem for the Zakharov-Kuznetsov equation. (Russian)
\emph{Differentsialnye Uravneniya} \textbf{31} (1995), no. 6, 1070–1081, 1103; translation in
\emph{Differential Equations} \textbf{31} (1995), no. 6, 1002–1012.

\bibitem{FHR}
L.~G.~Farah, J.~Holmer and S.~Roudenko.
Instability of solitons in the 2d cubic Zakharov-Kuznetsov equation. 
\emph{Nonlinear dispersive partial differential equations and inverse scattering}, 295–371, Fields Inst. Commun., 83, Springer, New York, 2019.

\bibitem{FHR1}
L.~G.~Farah, J.~Holmer and S.~Roudenko.
Instability of solitons – revisited, II: The supercritical Zakharov-Kuznetsov equation.
\emph{Nonlinear dispersive waves and fluids}, 89–109, Contemp. Math., 725, 
\emph{Amer. Math. Soc.,} 2019.

\bibitem{FHRY1} 
L.~G.~Farah, J.~Holmer, S.~Roudenko and K. Yang.
Asymptotic stability of solitary waves of the 3D quadratic Zakharov-Kuznetsov equation. 
\emph{Amer. J. Math.} \textbf{145} (2023), no. 6, 1695–1775.
 
\bibitem{FHRY}
L.~G.~Farah, J.~Holmer, S.~Roudenko and K. Yang.
Blow-up in finite or infinite time of the 2D cubic Zakharov-Kuznetsov equation.
Preprint, arXiv: 1810.05121.

%
%
%
\bibitem{KenigMartel}
C. E. Kenig and Y. Martel.
Asymptotic stability of solitons for the Benjamin-Ono equation. 
\emph{Rev. Mat. Iberoam.} \textbf{25} (2009), no. 3, 909–970.

\bibitem{KPV}
C. E. Kenig, G. Ponce, and L. Vega. 
Well‐posedness and scattering results for the generalized Korteweg‐de Vries equation via the contraction principle. 
\emph{Comm. Pure Appl. Math.} \textbf{46} (1993), no. 4, 527–620.

\bibitem{Kinoshita}
S. Kinoshita.
Well-posedness for the Cauchy problem of the modified Zakharov-Kuznetsov equation.
\emph{Funkcial. Ekvac.} \textbf{65} (2022), no. 2, 139–158.

%
%
%
%
%

\bibitem{Krieger}
J. Krieger, E. Lenzmann and P. Rapha\"el.
Nondispersive solutions to the $L^{2}$-critical half-wave equation. 
\emph{Arch. Ration. Mech. Anal.} \textbf{209} (2013), no. 1, 61–129.

%
%
%
%
%
%
%
\bibitem{LinPas}
F. Linares and A. Pastor.
Well-posedness for the two-dimensional modified Zakharov-Kuznetsov equation. 
\emph{SIAM J. Math. Anal.} \textbf{41} (2009), no. 4, 1323–1339.

\bibitem{LinPasJFA}
F. Linares and A. Pastor. 
Local and global well-posedness for the 2D generalized Zakharov-Kuznetsov equation.
\emph{J. Funct. Anal.} \textbf{260} (4) (2011) 1060–1085. 

\bibitem{LeCoz}
S. Le Coz, Y. Martel and P. Rapha\"el.
Minimal mass blow up solutions for a double power nonlinear Schr\"odinger equation.
\emph{Rev. Mat. Iberoam.} \textbf{32} (2016), no. 3, 795–833.

%
%
%
%
%
%
%
%
\bibitem{MMDUKE}
Y. Martel and F. Merle.
Nonexistence of blow-up solution with minimal $L^{2}$-mass for the critical gKdV equation. 
\emph{Duke Math. J.} \textbf{115} (2002), no. 2, 385–408.

%
%
\bibitem{MMR1}
Y. Martel, F. Merle and P. Rapha\"el.
Blow up for the critical gKdV equation. II: Minimal mass dynamics. 
\emph{J. Eur. Math. Soc. (JEMS)} \textbf{17} (2015), no. 8, 1855–1925.

%
%
%

\bibitem{MartelPilodAnn}
Y. Martel and D. Pilod.
Construction of a minimal mass blow up solution of the modified Benjamin-Ono equation. 
\emph{Math. Ann.} \textbf{369} (2017), no. 1-2, 153–245.


\bibitem{MartelPilod}
Y. Martel and D. Pilod.
Finite point blowup for the critical generalized Korteweg–de Vries equation.
\emph{Ann. Sc. Norm. Super. Pisa Cl. Sci.} (5) \textbf{25} (2024), no. 1, 371–425.

\bibitem{Merlenls}
F. Merle.
Determination of blow-up solutions with minimal mass for nonlinear Schr\"odinger equations with critical power. \emph{Duke Math. J.} \textbf{69} (1993), no. 2, 427–454.

\bibitem{Merlenonnls}
F. Merle.
Nonexistence of minimal blow-up solutions of equations 
$iu_{t}=-\Delta u-k(x)|u|^{4/N}u$ in $\R^{N}$.
Ann. Inst. H. Poincar\'e Phys. Th\'eor. \textbf{64} (1996), no. 1, 33–85.

%
%
%
%
%
%
%
%

\bibitem{MoRi}
L.~Molinet and F.~Ribaud.
Well-posedness results for the generalized Benjamin–Ono equation with small initial data. 
\emph{J. Math. Pures Appl.}, \textbf{83} (2004), no. 2, 277--311.

%

\bibitem{PV2} 
D. Pilod and F. Valet. 
Asymptotic stability of a finite sum of solitary waves for the Zakharov-Kuznetsov equation. 
\emph{Nonlinearity} \textbf{37} (2024), no. 10, Paper No. 105001, 41 pp.

\bibitem{PV1} 
D. Pilod and F. Valet. 
Dynamics of the collision of two nearly equal solitary waves for the Zakharov-Kuznetsov equation. 
\emph{Comm. Math. Phys.} \textbf{405} (2024), no. 12, Paper No. 287, 94 pp.


%

\bibitem{RaSz}
P. Rapha\"el and J. Szeftel.
Existence and uniqueness of minimal blow-up solutions to an
inhomogeneous mass critical NLS.
\emph{J. Am. Math. Soc.} \textbf{24} (2011), No. 2, 471-576.

%
%

\bibitem{RibaudVento2}
F. Ribaud and S. Vento.
Well-posedness results for the three-dimensional Zakharov--Kuznetsov equation. 
\emph{SIAM J. Math. Anal.} \textbf{44} (2012), no. 4, 2289--2304. 

\bibitem{SmithSogge}
H. F. Smith and C. D. Sogge. 
Global Strichartz estimates for nontrapping perturbations of the laplacian: Estimates for nontrapping perturbations.
\emph{Comm. Partial Differential Equations} \textbf{25} (2000), no.11--12, 2171-2183.

\bibitem{Suminimal}
Y. Su and D. Zhang.
Construction of minimal mass blow-up solutions to rough nonlinear Schr\"odinger equations.
\emph{J. Funct. Anal.} \textbf{284} (2023), no. 5, Paper No. 109796, 61 pp.


\bibitem{SunZheng}
C. Sun and J. Zheng.
Low regularity blowup solutions for the mass-critical NLS in higher dimensions.
\emph{J. Math. Pures Appl.} (9) \textbf{134} (2020), 255–298.

\bibitem{TangXu}
X. Tang and G. Xu.
Minimal mass blow-up solutions for the $L^{2}$-critical NLS with the delta potential for even data in one dimension. 
\emph{SIAM J. Math. Anal.} \textbf{56} (2024), no. 2, 1727–1769.


%
%
\bibitem{Weinstein}
M. I. Weinstein.
Nonlinear Schr\"odinger equations and sharp interpolation estimates. 
\emph{Comm. Math. Phys.} \textbf{87} (1982/83), no. 4, 567–576.

\end{thebibliography}
\end{document}